\documentclass{amsart}[12 pt]
\usepackage{amsmath, amsfonts, amssymb, euscript, amscd, latexsym, mathrsfs, bm, color, bbm}
 \usepackage{titlesec}
 \usepackage[pdftex,colorlinks=true,
                       pdfstartview=FitV,
                       linkcolor=blue,
                       citecolor=blue,
                       urlcolor=blue,
           ]{hyperref}

 \usepackage{enumitem}

\def\aa#1{ \begin{align*} #1 \end{align*} }
\def\aaa#1{ \begin{align} #1 \end{align} }
\def\mm#1{ \begin{multline*} #1 \end{multline*} }
\def\mmm#1{ \begin{multline} #1 \end{multline} }

\def\begeq{\begin{equation} \begin{cases}} 
\def\endeq{ \end{cases} \end{equation}}
\def\eq#1{ \begeq #1 \endeq }
\renewcommand{\vec}[2]{\ensuremath{\begin{pmatrix} #1 \\ #2\end{pmatrix}}}
\newcommand{\vect}[3]{\ensuremath{\begin{pmatrix} #1 \\ #2 \\ #3\end{pmatrix}}}
\def\bege{\begin{equation*} \begin{cases}} 
\def\ende{ \end{cases} \end{equation*}}

\def\eqn#1{ \begin{equation} \begin{aligned} #1 \end{aligned}\end{equation}}

\newtheorem{thm}{\sc Theorem}[section]
\newtheorem{lem}{\sc Lemma}[section]
\newtheorem{cor}{\sc Corollary}[section]
\newtheorem{df}{\sc Definition}[section]
\newtheorem{pro}{\sc Proposition}[section]
\newtheorem{rem}{\sc Remark}[section]

\newcommand{\eps}{\varepsilon}
\newcommand{\pl}{\partial}
\newcommand{\gt}{\geqslant}
\newcommand{\lt}{\leqslant}
\newcommand{\te}{\theta}
\newcommand{\sub}{\subset}
\newcommand{\dl}{\delta}
\newcommand{\al}{\alpha}
\newcommand{\gm}{\gamma}
 \newcommand{\Gm}{\Gamma}
 \newcommand{\Dl}{\Delta}
 \newcommand{\supp}{{\rm supp\,}}
 \newcommand{\la}{\lambda}
 
 \newcommand{\sg}{\sigma}
  \newcommand{\Sg}{\Sigma}

\newcommand{\om}{\omega}
\newcommand{\mc}{\mathcal}
 
\newcommand{\ms}{\mathscr}
\newcommand{\Om}{\Omega}

\newcommand{\mr}{\mathring}
\newcommand{\mb}{\mathbf}
\newcommand{\C}{{\rm C}}

\newcommand{\td}{\tilde}
\newcommand{\ox}{\otimes}

\newcommand{\E}{\mathbb E}
\newcommand{\we}{\wedge}

\renewcommand{\>}{\rangle}

\newcommand{\x}{\times}
\newcommand{\mto}{\mapsto}
\newcommand{\PP}{\mathbb P}
\newcommand{\mbb}{\mathbb}
\newcommand{\mbf}{\mathbf}

\newcommand{\rf}{\eqref}
\newcommand{\bi}{\begin{itemize}}
\newcommand{\ei}{\end{itemize}}

\DeclareMathOperator{\ind}{\mathbbm{1}}

\newcommand{\lb}{\label}

\newcommand{\fdot}{\,\cdot\,}

\newcommand{\sss}{\scriptscriptstyle}
\newcommand{\tp}{{\sss \top}}

\newcommand{\sig}{\varsigma}

\def\Rnu{{\mathbb R}}

\def\Nnu{{\mathbb N}}

\def\ffi{\varphi}

\def\com#1{}

\long\def\symbolfootnote[#1]#2{\begingroup%
\def\thefootnote{\fnsymbol{footnote}}\footnote[#1]{#2}\endgroup}
\titleformat{\section}[hang]{\large\bfseries}{\thesection.}{1ex}{}{}
\titleformat{\subsection}[hang]{\normalsize\bfseries}{\thesubsection}{2ex}{}{}
\titleformat{\subsubsection}[hang]{\small\bfseries}{\thesubsubsection}{2ex}{}{}

\allowdisplaybreaks

\include{srctex.sty}

\usepackage{lipsum}
\usepackage{algorithmic}

\begin{document}

\title[Smoothness of densities for path-dependent SDEs] {Smoothness of densities for path-dependent SDE\MakeLowercase{s} \\ under H\"ormander's condition}

\author{Alberto Ohashi}

\address{Departamento de Matem\'atica, Universidade de Bras\'ilia, 70910-900, Bras\'ilia, Brazil}
\email{ohashi@mat.unb.br}

\author{Francesco Russo}

\address{ENSTA Paris, Institut Polytechnique de Paris,
 Unit\'e de Math\'ematiques appliqu\'ees, 828, boulevard des Mar\'echaux,
F-91120 Palaiseau, France} \email{francesco.russo@ensta-paris.fr}

\author{Evelina Shamarova}

\address{Departamento de Matem\'atica, Universidade Federal da Para\'iba, 58051-900, Jo\~ao Pessoa, Brazil}
\email{evelina@mat.ufpb.br}

\maketitle

\vspace{-9mm}

\begin{abstract}
We establish the existence of smooth densities for solutions to a broad class of path-dependent SDEs
under a H\"ormander-type condition. The classical scheme based on the reduced Malliavin matrix turns out to be 
unavailable in the path-dependent context. We approach the problem by lifting
the given $n$-dimensional path-dependent SDE into a suitable $L_p$-type Banach space in such a way that
 the lifted Banach-space-valued equation becomes a state-dependent reformulation of the original SDE. 
We then formulate H\"ormander's bracket condition in $\Rnu^n$ for non-anticipative SDE coefficients defining the Lie brackets in terms
 of vertical derivatives in the sense of the functional It\^o calculus. Our pathway to  the main result engages  an interplay between the
 analysis of SDEs in Banach spaces, Malliavin calculus, and rough path  techniques. 
 \end{abstract}

\vspace{3mm}

{\footnotesize

\vspace{0.5mm}

}

\section{Introduction}
We address the fundamental problem on the existence of  a smooth density for the law of the solution to an $n$-dimensional path-dependent SDE
\aaa{
\lb{sde}
X(t) = x_0 + \int_0^t b(s,X^s,X(s)) ds + \int_0^t \sg(s,X^s,X(s))  dB_s, \quad t\in [0,T],
}
under a H\"ormander-type condition.
Above, for $x\in \C([0,T], \Rnu^n)$, $x^t$ denotes the path stopped at $t$, i.e., $x^t(s) = x(t\we s)$, 
 $B_t$ is a $d$-dimensional standard Brownian motion, 
$b$ and $\sg$ are $\Rnu^n$-valued and, respectively, $\Rnu^{n\x d}$-valued maps defined on appropriate spaces.
  The stochastic integral in \rf{sde} is understood in the It\^o sense.

The existence and regularity of densities for path-dependent SDEs
 was studied by various authors in the 80s and 90s \cite{BM91,BM95,burkhdan,hirsch,KSt,stroock}, 
 and also in recent years \cite{BC,CE, AT, AT3}. Among the main references stands the work by 
 S. Kusuoka and D. Stroock \cite{KSt}, where
  smooth densities were proved to exist  for SDEs with non-anticipative smooth coefficients
depending on the path in a general manner;
however, under the strong ellipticity assumption  $\sg^\tp\!\sg\gt \dl E$  (where $\dl>0$ and $E$ is the identity matrix).
Although it is natural to ask whether  one can find a H\"ormander condition
for general path-dependent maps and obtain the smoothness of the density for path-dependent SDEs
under this condition, to the best of authors' knowledge, this question has never been addressed before. 
Many of the aforementioned articles use the strong ellipticity assumption 
on the diffusion coefficients \cite{BC, BM91,  KSt, stroock, AT3}, others formulate a weaker condition (which, however,
is not H\"ormander's condition)  allowing the diffusion matrix to degenerate \cite{BM95, burkhdan, hirsch}.
Only two papers, \cite{AT} and \cite{CE}, deal with H\"ormander's condition for path-dependent SDEs,
both restricted to particular cases. In \cite{AT},
the equation is path-dependent only through the drift coefficient, and H\"ormander's condition
is applied to state-dependent diffusion coefficients; in \cite{CE}, a version of H\"ormander's condition
is formulated for SDEs with discrete delays.

In the present work,  we formulate a H\"ormander bracket condition that naturally extends the classical state-dependent case,
and furthermore, we show that under this condition,
the solution to \rf{sde} admits a smooth density with respect to Lebesgue measure. 
More specifically, H\"ormander's condition is formulated  for non-an\-ti\-ci\-pative maps
from the family $\sg_1,\ldots, \sg_d$,  $[\sg_i, \sg_j]$,  $1\lt i,j \lt d$,  
$[\sg_i, [\sg_j,\sg_k]]$, $1\lt i,j,k \lt d$, etc, where
the Lie brackets $[\fdot,\!\fdot]$ are understood in terms of vertical derivatives  \cite{cont, DUP}
(here, $\sg_k = \sg e_k$, where $\{e_k\}_{k=1}^d$ is the standard basis of $\Rnu^d$).
Importantly, our coefficients
$b(t,\fdot)$ and $\sg(t,\fdot)$ are allowed to depend on the whole history of the path $\{X(s), 0\lt s\lt t\}$ arbitrarily.
It is worth to point out that the strong ellipticity assumption  $\sg^\tp\!\sg\gt \dl E$ is a particular case of our far more general
H\"ormander's bracket condition, and that the latter allows degeneracy  of  the diffusion matrix 
(see the dis\-cussion in \cite{bell2004}, p. 39, on this topic).

The classical H\"ormander bracket condition is a sufficient hypothesis for the hypoellipticity of second-order differential operators
with smooth coefficients,
while the property of hy\-po\-el\-lip\-ti\-ci\-ty is connected with the existence of smooth densities for solutions
to the associated state-dependent SDEs. 
H\"ormander's theorem states that a second-order differential operator is hypoelliptic  if H\"ormander's condition holds, and
 its original proof  \cite{hor67} relies only  on  PDE methods.
The initial goal of the Malliavin calculus \cite{mal} was to directly prove the existence of smooth densities from H\"ormander's condition,
and by this, to provide a probabilistic proof of H\"ormander's theorem.
Subsequently, original Malliavin's proof 
was simplified, extended, or approached differently by a number of authors \cite{BM95-1, bis1, bis2, KSt, KSt2, KSt3, norris, wat84}
(see \cite{bell87} for a survey of approaches to the Malliavin calculus).
In later years, versions of H\"ormander's theorem by
using Malliavin calculus techniques were obtained for a variety
of differential equations: SDEs driven by a fractional Brownian motion \cite{BH, HT} and general Gaussian processes \cite{cass15, HP2},
SDEs with jumps \cite{BJG, cass09, IK, AT2}, rough differential equations (henceforth abbreviated as RDEs) \cite{cass15,HP2, HT}, 
and SPDEs \cite{GH, HM, MP, ocone}. To complement this list, we extend H\"ormander's theorem to path-dependent SDEs. 

Problems on regularity of laws in the non-Markovian context are understood in the literature in two different ways:
the memory may be transmitted to the system 
through the driving noise or through the coefficients depending on the solution path.
 Most of non-Markovian versions of H\"or\-man\-der's theorem were obtained for the first of the aforementioned cases (e.g., \cite{BH, cass15, HP2, HT}). In contrast, our article provides a H\"ormander theorem which is non-Markovian in the
path-dependent context.

Path-dependent SDEs represent a singular phenomenon
that is not found in those types of non-Markovian equations whose coefficients are state-dependent.
More specifically, equation \rf{sde}, due to the path dependency, may not lead to an invertible Jacobian at each point of $[0,T]$.
We give a simple example. Let $b(t,X^t, X(t))$ be equal  to $-X(t)$ if $t\lt 1$ and $-X(1)$ if $t>1$; $\sg = 1$, $n=1$.
The Jacobian $\pl_{x_0} X(t)$ can be then explicitly computed and equals $e^{-1}(2-t)$ which implies that $\pl_{x_0} X(2)^{-1}$ does not exist
(for further examples, see, e.g., \cite{hale}).
The inverse matrix plays, however, a major role
in the classical probabilistic proof of H\"ormander's theorem. As it is clear from the proof exposed in \cite{nualart},
the argument essentially relies on the  \textit{reduced} Malliavin matrix which is
expressed through the inverse matrix; and, furthermore, 
the Lie brackets, involved in H\"ormander's condition, appear through the interaction between
the inverse matrix and  the vector fields $\sg_k$.

 To overcome the issue of non-invertibility of the Jacobian, we reformulate 
 \rf{sde} through its lifting to the space space $\mc E_p = L_p([0,T]\to \Rnu^n,  \la+\mu_s) \oplus \Rnu^n$,
 where $\la$ is Lebesgue measure  and $\mu_s$ is a singular finite positive Borel measure on $[0,T]$.
 The support of $\mu_s$ need not be a discrete set; however, we will require that
 $[\tau_0,\tau) \cap \supp \mu_s  = \varnothing$
 for a left neighborhood of the point $\tau$ in which we want to prove the smoothness of the density for the
 solution to \rf{sde}.  By introducing a singular component $\mu_s$, we aim to encompass a larger class 
 of path-dependent coefficients. However, if a particular type of path dependence  can be described
 by the Lebesgue measure only or if
the singular component is concentrated on a finite set,
then, under assumptions (A1)--(A5) formulated below, the smoothness of the density
can be obtained at any point $\tau\in (0,T]$. Furthermore, if the support of $\mu_s$ is nowhere dense (e.g., Cantor set),
then, again under (A1)--(A5), the smoothness of the density can be obtained at Lebesgue-almost every point $\tau\in (0,T]$.

Next, we introduce infinite-dimensional lifts  to  $\mc E_p$ as follows:
\aaa{
\lb{main-lifts}
X_t = \vec{X^t}{X(t)}, \quad \hat\sg(t,\fdot) = \vec{\ind_{[t,T]} \sg(t,\fdot)}{\sg(t,\fdot)}, \quad
  \hat b(t,\fdot)  = \vec{\ind_{[t,T]} b(t,\fdot) }{b(t,\fdot) },
 } 
 so the SDE \rf{sde} takes the form
 \aaa{
\lb{sde1}
X_t = \vec{x_0}{x_0} + \int_0^t \hat b(s,X_s) ds +  \int_0^t \hat \sg(s,X_s) dB_s.
}
The infinite-dimensional reformulation of equation \rf{sde} is required for the inverse operator of 
the Jacobian of $X_t$ to exist 
 and allow a representation as a solution to a well-defined SDE. 
 Moreover, this SDE must allow a reinterpretation as an RDE.
 Thus, the choice of the space of the lift must meet the two above requirements.

 Dealing with RDEs is, in general, indispensable in infinite-dimensional problems on regularity of laws.
This happens because one wants to show the existence of smooth densities for 
 finite-dimensional projections of an infinite-dimensional solution. 
 In connection to this, we remark that the solution to \rf{sde} is a finite-dimensional projection of the infinite-dimensional
 solution to \rf{sde1}. In our setting, the reduced 
Malliavin covariance operator exists; however, it is infinite-dimensional, and thus,
the argument on the inverse moments of its
``determinant''   (see \cite{nualart}) does not work.
For this reason, we deal with the full Malliavin matrix (which is finite-dimensional).
This leads to the appearance of non-adapted factors attached to the Lie brackets in equations of  the type \rf{m-e} (below)
and makes an application of the classical Norris lemma \cite{nualart} impossible.
 Therefore, one can only hope to be able to use the version of Norris's lemma for rough paths \cite{HP2}.

 Getting back to the requirements on the infinite-dimensional state space for \rf{sde1}, we observe that
a serious obstacle is the presence of the indicator function $\ind_{[t,T]}$ in the definition of $\hat \sg$.
For example, the stochastic integral  $\int_{\tau_0}^t Z_s\pl_x \hat\sg(s,X_s) dB_s$, involved  in the SDE for the 
 inverse operator $Z_t$, must make sense as a rough integral; however,
 we can only hope for the integrand to be controlled by $B_t$ if the norm in the space of the lift 
 makes the map $t\mto \ind_{[t,T]}(\fdot)$ $2\al$-H\"older continuous with $\al\in (\frac13,\frac12)$.
 To achieve this, we choose 
  $p\in (1,\frac32)$ for the above-described space $\mc E_p$.

The first guess would be to choose $\mc H = L_2([0,T]\to \Rnu^n, \la+\mu_s) \oplus \Rnu^n$
as the space of the lift since the analysis of SDEs in Hilbert spaces is much simpler.
However, as we show in Subsection \ref{ind-f}, the map 
\aaa{
\lb{ind-map}
[\tau_0,\tau] \to L_p([0,T],  \la+\mu_s), \quad t\mto \ind_{[t,T]}(\fdot)
}
is $\frac1p$-H\"older  continuous.
This tells us that $\mc H$ is not suitable as the space of the lift since in this space, the map \rf{ind-map} does
not possess the minimal H\"older regularity discussed above.
On the other hand, in the space $\mc E_p$
with $p\in (1,\frac32)$, the map \rf{ind-map} becomes at least $(\frac23 + \eps)$-H\"older regular for every 
arbitrarily small positive number $\eps$. One may think of it as that the topology of $\mc H$ is too strong for
rough integrals containing $\ind_{[t,T]}$ to converge, so one needs to weaken the topology, 
and this is how the space $\mc E_p$ comes into play.

When dealing with RDEs, we restrict  our analysis to the interval $[\tau_0,\tau]$ since the $\frac1{p}$-H\"older regularity
 of the map \rf{ind-map} may not hold on $[0,T]$. This restriction becomes possible
since for the Malliavin covariance matrix $\gm_\tau$ of $X(\tau)$, it holds that
 \aa{
 (\gm_\tau z,z) \gt (\gm^0_\tau z,z), \quad  \text{where} \; \;  (\gm^0_\tau z,z) = \sum_{k=1}^d \int_{\tau_0}^\tau (z,J_{\tau,s} \hat \sg_{k}(s,X_s))^2 \, ds,
 \quad z\in \Rnu^n.
 }

Let us point out the properties of $\mc E_p$ that are important in our analysis.
First, the space $\mc E_p$ allows to consider a rich class of 
 path-dependent coefficients  which can be handled within our framework.
Second, as we show in Subsection \ref{sde-inverse}, it is possible to define a stochastic integral in $\mc L(\mc E_p,\mc H)$,
and by this, to give meaning to the SDE for the inverse operator of the Jacobian of $X_t$ in $\mc L(\mc E_p)$.
More specifically, slightly reformulating the original SDE,
we are able to make use of the theory of stochastic integration in the 2-smooth Banach space 
of $\gm$-radonifying operators $\mc H\to \mc E_p^*$.
Also, recall that the $\mc L(\mc E_p)$-valued SDE for the inverse operator must make sense as an RDE.  
The latter becomes possible (on the interval $[\tau_0,\tau]$) due to the
$(\frac23 + \eps)$-H\"older regularity of the map \rf{ind-map}.
Another advantage of the space $\mc E_p$ is that
we are able to use the theory of SDEs to prove the existence of solutions to  RDEs.
The first step in this direction is finding a sufficient condition for an It\^o stochastic integral to be controlled by 
$B_t$. This becomes possible due to a suitable adaptation of 
the proof of Kolmogorov's criterion for rough paths \cite{FH}.
Next, since the stochastic integral involved in an SDE is a controlled rough path, so is the solution.
Thus, we conclude that $X_t$ and the inverse operator $Z_t$ are controlled by $B_t$. This,
together with the correct H\"older regularity  of the map \rf{ind-map}, implies that the integrands 
$\hat \sg(t,X_t)$ and $Z_t \pl_x \hat \sg(t,X_t)$ are also controlled by $B_t$.
The respective
stochastic integrals can be then regarded as rough integrals, and SDEs as RDEs.

Equation \rf{m-e}, introduced below, 
 is the central object in proving, by means of Norris's lemma for rough paths, 
that $X(\fdot)$ admits a smooth density at point $\tau\in (0,T]$. 
The equation reads
\mmm{
\lb{m-e}
J_{\tau,t} \hat V(t,X_t)   =  J_{\tau,\tau_0}\hat V(\tau_0,X_{\tau_0}) + 
\int_{\tau_0}^t J_{\tau,s}\vec{V(s,X_s)}{0}d\ind_{[s,T]} \\
+ \int_{\tau_0}^t J_{\tau,s}(\widehat{\pl_s V} + [\hat \sg_0, \hat V])(s,X_s) \, ds 
+ \int_{\tau_0}^t   J_{\tau,s} [\hat \sg, \hat V](s,X_s) d \mbf  B_s, \qquad t\in [\tau_0,\tau].
}
Here, $J_{\tau,t} = \Pi_n Y_\tau Z_t$, where
$Y_t$ is the Jacobian of $X_t$, $Z_t$ is its inverse, and
 $\Pi_n: \mc E_p \to\Rnu^n$, $(x,y)\mto y$, is the projection onto $\Rnu^n$.
Next, $\mbf B_t$ is the Brownian rough path  lifted from $B_t$, 
$[\hat \sg, \hat V] =  \sum_{k=1}^d [\hat \sg_k, \hat V]\ox e_k$, and 
$\hat\sg_0 = \hat b - \frac12 \sum_{k=1}^d \pl_x \hat \sg_k\hat\sg_k$.
Furthermore, for each $(t,x)\in [0,T]\x \mc E_p$, the $\mc E_p$-lift $\hat V$ is defined as follows:
\aaa{
\lb{ep-lift}
\hat V(t,x) = \vec{\ind_{[t,T]}(\fdot) V(t,x)}{V(t,x)}.
} 
We emphasize that giving meaning to \rf{m-e} would not
be possible unless our choice of the space $\mc E_p$ with $p\in (1,\frac32)$.
The Lie brackets in \rf{m-e} are generated through the interaction, by means of It\^o's formula, of the inverse operator $Z_t$
 with  $\hat V(t,X_t)$, 
 where the map $\hat V$  comes from the family of iterated Lie brackets
$\hat \sg_1,\ldots, \hat \sg_d$,  $[\hat \sg_i,\hat \sg_j]$,  $1\lt i,j \lt d$,  
$[\hat \sg_i, [\hat \sg_j,\hat \sg_k]]$, $1\lt i,j,k \lt d$, etc.
However, the factor $\ind_{[t,T]}$  unavoidably makes $\hat V$ depending on $t$, and due to
the non-differentiability of this factor, a direct application of any version of It\^o's formula, classical or rough, is not possible.
However, in $\mc E_p$, one can  integrate with respect to $\ind_{[s,T]}$  in the Young sense
which allows to overcome the non-differentiability issue. Also, 
 the integral with respect to $\mbf B_s$ becomes well-defined  since  the map $s\mto Z_s[\hat \sg, \hat V](s,X_s)$ is sufficiently regular.

Another important aspect of this work is a formulation of H\"ormander's condition in $\Rnu^n$ relying on the 
following definition of the Lie brackets. 
For two non-anticipative maps $V_i: [0,T] \x D([0,T],\Rnu^n) \to \Rnu^n$, $i=1,2$,
we define the Lie bracket as follows:
\aa{
[V_1,V_2](t,x^t,x(t)) = (\pl_v V_2 V_1 - \pl_v V_1 V_2) (t,x^t,x(t)).
}
Here, $\pl_v$ denotes the vertical derivative, as it is defined in \cite{cont}. Note that, if $V_1$ and $V_2$ can be extended
to Fr\'echet-differentiable maps $[0,T]\x \mc E_p \to \Rnu^n$, then for each $(t,x) \in [0,T] \x D([0,T],\Rnu^n)$,  
it holds that
\aa{
[V_1,V_2](t,x^t,x(t)) = \Pi_n [\hat V_1, \hat V_2](t,x^t,x(t)),
}
which connects the above-defined Lie brackets  with the classical Lie brackets
 in $\mc E_p$.

Remark that due to the presence of the Young integral,
 equation \rf{m-e}  does not allow to use the term
$\widehat{\pl_s V} + [\hat \sg_0, \hat V]$ 
 in the formulation of H\"ormander's condition.
Also, we note that  $J_{\tau,\tau} [\hat \sg_k, \hat V](\tau,X_\tau) = [\sg_k,V](\tau,X_\tau)$. This is how 
H\"ormander's condition in $\Rnu^n$ enters into the infinite-dimensional analysis of the problem.

We now highlight the three main elements in our proof of H\"or\-man\-der's theorem. 
\begin{itemize}[leftmargin=12pt, topsep=2pt]
 \item[1.] We specify the Banach space $\mc E_p$ that ensures the sufficient H\"older regularity of the 
 map $t\mto \ind_{[t,T]}(\fdot)$. At the same time, $\mc E_p^*$ is 2-smooth, so the inverse
 operator turns out to be a solution to a well-defined SDE.
\item[2.] We define the Lie brackets for non-anticipative maps to formulate
a path-dependent version of H\"or\-man\-der's condition in $\Rnu^n$.
\item[3.]  We interpret stochastic integrals as rough integrals, and SDEs as RDEs.
 The latter allows us to obtain solutions to RDEs using the theory of SDEs which requires less
restrictions on the coefficients. Our results in this direction may also have an independent interest.
 Importantly,  using SDEs, we are able to prove the existence of finite moments 
of the  H\"older seminorms $\|X\|_\al$, $\|Z\|_\al$ and of certain norms of the quantities $R^X$,   $R^Z$ 
(introduced in Subsection \ref{smooth}).
The aforementioned objects are extensively used in the argument on the existence of smooth densities. 
\end{itemize}
We believe that most of the aforementioned works  \cite{BM91, BM95, burkhdan, hirsch, KSt, stroock, AT, AT3} aim
to find H\"ormander's condition alternatives that ensure the existence and smoothness
of densities for path-dependent SDEs; so it is important to demonstrate
how the tools developed by M. Hairer and coauthors \cite{GH, HP2}, e.g., Norris's lemma for rough paths,
can be used to make a path-dependent version of H\"ormander's theorem possible.
It remains to point out that equation \rf{m-e}  resembles equation (1.4) from \cite{GH}, 
and our route to the smoothness of densities for path-dependent SDEs was, in fact, inspired by (1.4). 
However, the nature of our problem requires a new set of ideas; so
the main elements, highlighted above, are new and do not have analogs in \cite{GH} or \cite{HP2}.


Finally, we remark that the scheme introduced in this article is expected to work for path-dependent RDEs
under the assumption that the Jacobian of the lifted equation and its inverse possess finite moments of all orders. 

We structure our work as follows. In Section \ref{s2}, we show that equation \rf{sde} is equivalent 
to the $\mc E_p$-valued equation \rf{sde1}. In Section \ref{s3}, we prove the existence of the inverse
operator $Z_t$ of the Jacobian of $X_t$.
In Section  \ref{me},  
we obtain a sufficient condition when an It\^o stochastic integral is controlled by $B_t$ (Proposition \ref{pro4.4}). 
Furthermore, we prove  that the SDEs for $\hat V(t,X_t)$ and $Z_t$ can be viewed as RDEs. 
In the same section, we derive equation \rf{m-e} (Theorem \ref{thm1}). 
Finally, in Section \ref{s4}, 
we formulate a path-dependent version of H\"ormander's bracket condition in $\Rnu^n$ and prove that $X(\tau)$
admits a smooth density with respect to Lebesgue measure (Theorem \ref{thm2}).
\section{State-dependent reformulation of the SDE \rf{sde}}
\lb{s2}
\subsection{Notation}
\lb{notation}
Let $\mu=\la+\mu_s$, where $\la$ is Lebesgue measure and $\mu_s$ is a 
singular finite positive Borel measure on $[0,T]$.

Define $E_p =  L_p([0,T]\to \Rnu^n,  \mu)$, 
$\mc E_p = E_p \oplus \Rnu^n$, 
$p\in (1,2]$,
 $\mc E =  \C([0,T],\Rnu^n) \oplus\Rnu^n$,
$H=E_2$, $\mc H = \mc E_2$. The squared norm in $\mc E_p$ is defined as follows:
\aa{
\Big(\!\int_0^T |h_1(t)|^p \mu(dt)\Big)^\frac2{p} + |h_2|^2, \qquad  \vec{h_1}{h_2}\in \mc E_p,
 }
 while the first term defines the squared norm in $E_p$.

Note that for all  $p\in (1,2)$, $H$ is identically imbedded into ${E}_p$, and
for all $f\in H$, $\|f\|_{{E}_p} \lt \mu[0,T]^{(\frac1{p}-\frac12)} \|f\|_H$. Likewise, $\mc H$ is identically imbedded into
$\mc E_p$, and for all $f \in \mc H$, 
\aaa{
\lb{norms}
\|f\|_{\mc E_p} \lt (\mu[0,T]^{\frac2{p}-1} + 1)^\frac12 \|f\|_{\mc H}.
}

For the coefficients $\sg: [0,T] \x \mc E_p \to \Rnu^{n\x d}$ and $b: [0,T] \x \mc E_p \to \Rnu^n$ of the SDE \rf{sde},
the lifts $\hat \sg$ and $\hat b$ are defined by \rf{ep-lift}. Furthermore, 
$\sg_k = \sg e_k$, where $\{e_k\}$ is the standard basis of $\Rnu^d$, whose lift $\hat \sg_k$ is also defined by \rf{ep-lift}.
It is convenient to define the maps $\td \sg_k: [0,T] \x \mc E_p \to H$,
$(t,x) \mto  \ind_{[t,T]}(\fdot)\sg_k(t,x)$; $\td b$ is defined likewise. Remark that since $H\sub E_p$,
the maps $\td \sg_k$ and $\td b$ can be regarded as $E_p$-valued.

For a map $V: [0,T]\x \mc E_p \to \Rnu^n$ and its lift $\hat V: [0,T]\x \mc E_p \to \mc E_p$, defined by \rf{ep-lift},
$\pl_xV(t,x)$ and $\pl_x\hat V(t,x)$ always denote  Fr\'echet derivatives  with respect to the second argument, and,
 furthermore, $\pl^{l}_xV(t,x)$ and $\pl^{l}_x\hat V(t,x)$ denote the $l$-th order Fr\'echet derivatives.

Let $\mc F_t$ be the filtration generated by the Brownian motion $B_t = (B^1_t, \ldots, B^d_t)$, 
defined on a probability space $(\Omega, \mc F,\PP)$, completed by the $\PP$-null sets.  
For a Banach space $E$, let $\mc S_q([s_1,s_2], E)$, $q\gt 1$, denote the Banach space of 
$\mc F_t$-adapted $E$-valued 
stochastic processes $\xi_t$ such that $\E\sup_{t\in [s_1,s_2]}\|\xi_t\|_E^q<\infty$. 

Let $U_t$, $t\in[s_1,s_2]$, be an $E$-valued process. 
Then, $\|U\|_\infty  =  \sup_{t\in [s_1,s_2]}\|U_t\|_E$, and $\|U\|_\al$ denotes the H\"older seminorm
computed with respect to the norm of $E$ over $[s_1,s_2]$. 
If  we would like to emphasize that the norm (or seminorm) is computed over a subinterval
$[s_1,s_2]$ of $[0,T]$, we write $\|U\|_{\infty , [s_1,s_2]}$ and $\|U\|_{\al, [s_1,s_2]}$.  
Remark, that it is usually clear from the context, in which Banach space
 $U_t$ takes its values, so the aforementioned notation will not lead to a misunderstanding.

The number $\tau\in (0,T]$ denotes a time point in which we would like to prove the smoothness of the density for
the law of the solution to \rf{sde}; $\tau_0$ is an arbitrary point from $[0,\tau)$ with the property that  
$[\tau_0,\tau)\cap \supp \mu_s = \varnothing$.
\subsubsection*{$\bm X(t)$ vs $\bm X_t$ and similar notation}
Except for the Brownian motion $B_t$, we write $t$ as a subscript for the state-dependent $\mc E_p$-valued or
$\mc L(\mc E_p)$-valued processes  $X_t$, $Y_t$, $Z_t$, and other infinite-dimensional processes.
Furthermore, $X(t)$ indicates the $n$-dimensional state-dependent solution process for \rf{sde},
while $X^t$, defined as $X^t(\fdot) = X(t\we\fdot)$, indicates the path  of $X(\fdot)$ stopped at $t$. 

To improve the readability of formulas throughout  the article, we prefer to write $B_t$, rather than $B(t)$,
for the $d$-dimensional standard Brownian motion. 

\subsection{Lifting to $\mc E_p$}
Here we introduce an $\mc H$-valued SDE, which we regard as a lift of \rf{sde} to $\mc H$. 
 Since $\mc H$ is identically imbedded into $\mc E_p$,  $p\in (1,2)$, the $\mc H$-valued SDE can also be viewed as a lift to $\mc E_p$.
First, we state a result on the existence of a unique solution to \rf{sde} in $S_q([0,T],\Rnu^n)$.
\begin{lem}
\lb{lem11}
Assume that $\sg_k(t,x)$, $k=1,\ldots, d$, and $b(t,x)$ are maps $[0,T] \x \mc E \to \Rnu^n$
satisfying the Lipschitz and the linear growth conditions with respect to $x\in\mc E$ uniformly in $t$.
Then, \rf{sde} possesses a unique solution in $\mc S_q([0,T], \Rnu^n)$
for all $q\gt 1$.
\end{lem}
\begin{proof}
The proof is straightforward.
\end{proof}
\begin{pro}
\lb{pro12}
Assume that $\sg_k(t,x)$, $k=1,\ldots, d$, and $b(t,x)$ are maps $[0,T] \x \mc H \to \Rnu^n$
satisfying the Lipschitz and the linear growth conditions with respect to $x\in\mc H$ uniformly in $t$.
Furthermore, we assume that $b$ and $\sg_k$ are continuous over $[0,T] \x \mc E$.
Then, \rf{sde} is equivalent to the  $\mc H$-valued SDE \rf{sde1}.
Moreover, $X_t$ is in $\mc S_q([0,T],\mc E_p)$ for all $q\gt 1$ and $p\in (1,2]$.
\end{pro}
\begin{rem}
\rm Remark that  \rf{sde} is included in \rf{sde1} as the equation for the second component of $X_t$. 
Furthermore, since the solution to \rf{sde} has a.s. continuous paths, the component $X^t$ of $X_t$
(in spite of being $H$-valued) is well-defined.
\end{rem}
\begin{proof}[Proof of Proposition  \ref{pro12}]
Let $X(t)$ be the solution to \rf{sde} which exists by Lemma \ref{lem11}.
Since $\ind_{[s,T]}(r) = \ind_{[0,r]}(s)$, by the definition of $\td \sg_k$ and $\td b$, we obtain
\mm{
X^{t}(r) = X(t\we r)
= x_0+ \int_0^{r\we t} b(s,X^s,X(s)) ds +\sum_{k=1}^d \int_0^{r\we t} \sg_k(s,X^s,X(s)) dB^k_s \\ =
x_0 + \int_0^t \td b(s,X^s,X(s))(r) ds +\sum_{k=1}^d \int_0^t \td\sg_k(s,X^s,X(s))(r) dB^k_s.
}
This implies that the SDE \rf{sde} will be equivalent to the $H$-valued SDE
\aaa{
\lb{sde4}
X^{t} = x_0+ \int_0^t \td b(s,X^s,X(s))ds +\sum_{k=1}^d \int_0^t \td\sg_k(s,X^s,X(s)) dB^k_s
}
if we prove that the dependence on $r$  in the stochastic integral can be taken outside of the integral sign.
In other words, to show that \rf{sde} and \rf{sde4} are equivalent, it suffices to prove that 
\aaa{
\lb{id45}
\Big(\int_0^t \td\sg_k(s,X^s, X(s)) dB^k_s\Big)(r) = \int_0^t \td\sg_k(s,X^s, X(s))(r) dB^k_s
\quad \text{for each} \; k \; \text{a.s.} 
}
A similar identity for the integrals with respect to $ds$ is obvious.
Considering the approximation of $\sg_k(s,X^s,X(s))$ by simple functions
of the form $\sg^N_k(s) = \sum_{i=1}^N \sg_k(s_i,X^{s_i},X(s_i))\ind_{(s_{i-1},s_{i}]} + \ind_{\{0\}} \sg_k(0,x_0,x_0)$,
we obtain that for all $r$,
\aa{
\Big(\int_0^t \td \sg^N_k(s) dB^k_s\Big)(r) = \int_0^t \td \sg^N_k(s)(r) dB^k_s = \int_0^{r\we t} \sg^N_k(s) dB^k_s.
}
This implies that $f_N(t) = \int_0^t \td \sg^N_k(s) dB^k_s$ is a Cauchy sequence in $\mc S_q([0,T],H)$ since it is a Cauchy sequence in
$\mc S_q([0,T],\C([0,T],\Rnu^n))$.
But the limit of $\{f_N\}$ in  $\mc S_q([0,T],H)$
is $\int_0^t \td \sg_k(s) dB^k_s$, where the stochastic integral is $H$-valued.
This implies that $\{f_N\}$ has the same limit in $\mc S_q([0,T],\C([0,T],\Rnu^n))$. 
In the above equation, 
passing to the limit  as $N\to\infty$ in the aforementioned space, we obtain \rf{id45}.

Next, we note that the $H$-valued SDE with respect to $U_t$
\aa{
U_t = x_0 + \int_0^t \td b(s,U_s, X(s)) ds +\sum_{k=1}^d \int_0^t \td\sg_k(s,U_s, X(s)) dB^k_s.
}
possesses a unique solution. Hence, $U_t = X^t$, and the SDEs \rf{sde} and \rf{sde4} are equivalent.
Finally, we note  that the pair of equations \rf{sde}--\rf{sde4} is exactly equation \rf{sde1}.

The fact that $X_{\fdot}\in \mc S_q([0,T],\mc E_p)$ follows immediately from Lemma \ref{lem11}.
\end{proof}

\subsection{Standing assumptions}
\lb{standing-assumptions}

Fix $p \in (1,\frac32)$ and assume the following:
 \bi
 \item[\bf (A1)] For each $t\in [0,T]$ and $k=1,\ldots, d$, $\sg_k(t,\fdot)$
 and $b(t,\fdot)$ are infinitely Fr\'e\-chet differentiable maps  $\mc E_p \to \Rnu^n$,
whose first-order derivatives are bounded uniformly in $t$ and the higher-order derivatives,
when restricted to $\mc E= \C([0,T],\Rnu^n) \oplus \Rnu^n$,
have at most polynomial growth with respect to the second argument uniformly in $t$; that is,
for each $l \gt 2$, 
 $|\pl^{l}_x b(t,x)| +\sum_k |\pl^{l}_x\sg_k(t,x)| \lt C_l (1+\|x\|^{q_l}_{\mc E})$
 for some $q_l\in\Nnu$  and $C_l>0$. Moreover, $\sg_k$ and $b$ are continuous over
 $[0,T] \x \mc E$.
\item[\bf (A2)]   There exists $\mc \tau_0\in [0,\tau)$ such that 
$[\tau_0,\tau) \cap \supp \mu_s = \varnothing$.
\item[\bf (A3)]   For each $(t,x) \in [\tau_0,\tau]\x \mc E$ and $k=1,\ldots, d$, $\sg_k(t,x)$, 
$b(t,x)$,  and their Fr\'echet derivatives in $x$ of all orders are differentiable with respect to $t$;
all the aforementioned derivatives, including the derivatives in $t$, 
are continuous over $[\tau_0,\tau]\x \mc E$.
\item[\bf (A4)]   
 The derivatives in $t$, mentioned in (A3),
have at most polynomial growth with respect to $x$, i.e., they
are bounded  by $1+\|x\|^q_{\mc E}$ (for some $q\in \Nnu$) multiplied by a constant.
\ei
\begin{rem}
\rm
(A1) implies the statements of Lemma \ref{lem11} and Pro\-po\-sition \ref{pro12}.
\end{rem}

\subsection{Examples of path-dependent coefficients}
\lb{examples}
Here we give examples of path-dependent coefficients $\sg_k$ and $b$ satisfying  assumptions (A1)--(A4).
In the examples below, both coefficients are defined via the maps  $A_k$ and $\Phi_i$.
One immediately verifies that under the
assumptions on $A_k$ and $\Phi_i$ that we formulate in the examples,
 $\sg_k$ and $b$ satisfy (A1)--(A4).

\subsubsection*{Coefficients containing integrals}
Let $\nu_i$, $i=1, \ldots, N$, be signed measures with the property $|\nu_i| \lt K \mu$ for some constant $K$.
Further let $\zeta$ be a map $[0,T] \x E_p \to \Rnu^k$ or 
$E_p \to \Rnu^k$ such that its $i$-th component $\zeta^i$ is given by one of the expressions 
\aa{
 \zeta^i(t,x) =  \int_{[0,t)} \Phi_i(s,x(s))\nu_i(ds)  \quad \text{or} \quad    \zeta^i(x) =  \int_{[0,T]} \Phi_i(s,x(s))\nu_i(ds),
}
where each $\Phi_i: [0,T] \x \Rnu^n \to \Rnu$ has bounded first-order derivatives  w.r.t. the second argument
and higher-order derivatives of at most polynomial growth uniformly w.r.t. the first argument.
Let the coefficients $\sg_k$ (or $b$) be defined on $[0,T]\x E_p \x \Rnu^n$ as follows:
\aa{
 \sg_k(t,x,y) = A_k(t,\zeta(t,x),y) \quad \text{or} \quad   \sg_k(t,x,y) = A_k(t,\zeta(x),y),
}
where $A_k$ are defined on appropriate spaces and
have bounded first-order derivatives  w.r.t. the second and the third arguments, and,
furthermore, the higher-order derivatives w.r.t. the same arguments are of at most polynomial growth.
Moreover, $A_k$ and all the aforementioned derivatives are differentiable in $t$ and the respective derivatives
have at most polynomial growth w.r.t. the second and the third arguments.
Finally,  all the aforementioned derivatives are continuous. 
\begin{rem}
\rm 
Remark that by the inequality $|\nu_i| \lt K \mu$,
the support of the singular component of $\nu_i$ is contained in $\supp \mu_s$. 
This implies that 
\aa{
\int_{[0,t)} \hspace{-1mm}\Phi_i(s,x(s))\nu_i(ds) = \int_{[0,\tau_0)}\hspace{-1mm} \Phi_i(s,x(s))\nu_i(ds) 
+ \int_{[\tau_0,t)}\hspace{-1mm} \Phi_i(s,x(s))\rho(s) ds, \;\; t\in [\tau_0,\tau],
}
where $\rho$ is the density of $\nu_i$ w.r.t. Lebesgue measure on  $[\tau_0,\tau]$.
Therefore, $\zeta^i(t,x)$ is differentiable in $t$ on $[\tau_0,\tau]$.
\end{rem}

\subsubsection*{Coefficients containing multiple integrals}
Let $\nu_i$ be a signed measure  on $[0,T]^N$ such that for any Borel set $A\sub [0,T]$ and $j\in \{1,\ldots, N\}$,
 $|\nu_i|([0,T]^{j-1} \x A \x [0,T]^{N-j})\lt K \mu(A)$, where $K$ is a constant. 
Let $\sg_k(t,x,y)$ be defined on $[0,T]\x E_p \x \Rnu^n$ by $A_k(t,\zeta(x),y)$  with
the $i$-th component of $\zeta$ being given by a multiple integral
\aa{
\zeta^i(x)  = \int_{[0,T]^N} \Phi_i(s_1,\ldots, s_N,x(s_1),  \ldots, x(s_N)) \nu_i(ds_1\ldots ds_N),
}
where each $\Phi_i: [0,T]^N \x \Rnu^{n N} \to \Rnu$ 
 has bounded first-order partial  derivatives in  the last $N$ arguments, 
and the higher-order derivatives w.r.t. the same arguments have at most polynomial growth
 uniformly w.r.t. the first $N$ arguments.
The assumptions on $A_k$ are the same as in the previous example.
\subsubsection*{Coefficients with a continuous delay}
Let $\sg_k(t,x,y)$ be defined on $[0,T]\x E_p \x \Rnu^n$ as $A_k(t,\zeta(t,x),y)$, where
 the $i$-th component of $\zeta$ is
\aa{
\zeta^i(t,x)  = \int_{-T}^0 \Phi_i(t+s,x(t+s)) \rho(s) ds.
}
Take $\mu = \la + \dl_0$ (where $\dl_0$ is the Dirac measure concentrated at zero) 
and set $x(t) = x(0)$ for $t<0$. Above,
 $\Phi_i: [-T,0] \x \Rnu^n \to \Rnu$ is a function with the same properties as in the first example,
$\rho$ is a bounded differentiable function $[-T,0]\to \Rnu$.
One can transform the above expression for $\zeta^i$ as follows:
\aa{
\zeta^i(t,x)  = \int_{-T+t}^t  \Phi_i(s,x(s)) \rho(s-t) ds
}
which shows that $\zeta^i$ is differentiable in $t$. To see that $\zeta^i$ has  a bounded Fr\'echet derivative 
with respect to the second argument, we note that 
$\pl_x\zeta^i(t,x)h = \int_{-T}^0\pl_2 \Phi_i(t+s,x(t+s))\{h(0)\ind_{[-T,-t]}+ h(t+s)\ind_{(-t,0]}\} \rho(s) ds$, where  $\pl_2$
denotes the partial derivative w.r.t. the second argument and $h\in E_p$.
Thus,
 \mmm{
 \lb{zest9}
 |\pl_x\zeta^i(t,x)h| \lt K_1 \Big(
 \int_{-t}^0 |h(t+s)| |\rho(s)| ds  + \int_{-T}^{-t} |h(0)|   |\rho(s)| ds\Big) \\
 \lt K_2\Big( \int_0^t |h(s)| |\rho(s-t)| ds + |h(0)|\Big) 
  \lt K_3 \Big(\int_0^T |h(s)|^p \mu(ds)\Big)^\frac1{p},
 }
where $K_1, K_2, K_3$ are positive constants. 
The higher-order derivatives of $\zeta^i$ can be estimated likewise except we get  
 a factor of the form $1+ \sup_{[0,T]}|x(s)|^l$ on the right-hand side. 
 The assumptions on $A_k$ are the same as in the first example.
\subsubsection*{Continuous delay given by a multiple integral}
Suppose $\sg_k$ is given as in the previous example, i.e., via $A_k$ depending on $\zeta(t,x)$, where 
the $i$-th component of $\zeta$ is defined as follows:
\mm{
\zeta^i(t,x) = \\ = \int_{[-T,0]^N}  \Phi_i(t+s_1,\ldots, t+s_N,x(t+s_1), \ldots, x(t+s_N)) \rho(s_1, \ldots, s_N) ds_1\ldots ds_N,
}
with $x(t) = x(0)$ for $t<0$ ($\mu = \la + \dl_0$). Above, $\Phi_i: [-T,0]^N \x \Rnu^{n N}\to \Rnu$ is 
a function with the same properties as in the second example,
$\rho: [-T,0]^N \to \Rnu$ is a differentiable function with the property
\aa{
\int_{[-T,0]^{N-1}} |\rho(s_1, \ldots, \underbrace{s-t}_j, \ldots, s_N)| ds_1 \ldots ds_N \lt K \quad \forall s\in [0,t]
}
fulfilled for each $j$; the integration on the left-hand side excludes $ds_j$.  It is straightforward to see that $\zeta^i(t,x)$ is 
Fr\'echet-differentiable 
w.r.t. $x$ and the derivative is bounded. Similar to \rf{zest9}, we obtain a bound
on the first derivative of $\zeta^i$. By the boundedness of the first partial derivatives of $\Phi_i$, there exists a constant $K_1>0$ such that
\aa{
|\pl_x\zeta^i(t,x)h| \lt K_1 \sum_{j=1}^N \int_{[-T,0]^N}|h(t+s_j)| |\rho(s_1, \ldots, s_N)|\,  ds_1 \ldots ds_N.
}
Introducing 
$\rho_j(s) =  \int_{[-T,0]^{N-1}} |\rho(s_1, \ldots, \underbrace{s}_j, \ldots, s_N)| ds_1\ldots ds_{j-1} ds_{j+1} \ldots ds_N$,
we obtain that the right-hand side of the previous inequality can be bounded by
\aa{
\sum_{j=1}^N\Big( \int_0^t|h(s_j)|\, \rho_j(s_j-t) \, ds_j + |h(0)|\int_{-T}^{-t}  \rho_j(s_j)\, ds_j\Big)
}
multiplied by a constant, which implies that $\pl_x\zeta^i(t,x)$ is bounded by the same argument as in \rf{zest9}.
The assumptions on $A_k$ are as in the previous example.

\subsubsection*{One more type of continuous delay}
In the situation described in the first example, i.e.,  $\sg_k(t,x,y) = A_k(t,\zeta(t,x),y)$, 
take $h\in (0,T)$ and consider
 \aa{
  \zeta^i(t,x) =  \int_{[0,t-h)} \Phi_i(s,x(s))\nu_i(ds).
 }
In addition to the assumptions formulated in the first example, 
we assume that $\nu_i\{(-h,0)\} = 0$, $\mu\{0\}\ne 0$, and 
$[\tau_0-h,\tau)\cap\supp \mu_s = \varnothing$.

\subsubsection*{Coefficients depending on the path at a finite number of points}
Consider the function $\sg_k$ of the form
\aaa{
\lb{fnp}
\sg_k(s, y, x) = A_k(s,\underbrace{y(t_0),y(t_1), \ldots y(t_{i(s)}),  0, \ldots,0, x}_{m+2}), \quad y\in E_p, \; x\in\Rnu^n,
}
where  $0 = t_0 < t_1 < \dots < t_m = T$,
$A_k: [0,T] \x \Rnu^{m+2}\to \Rnu^n$, and
$i(s)\in \{0,1,\ldots, m-1\}$ is such that $t_{i(s)} \lt s < t_{i(s)+1}$; $i(t_m) = m$.

For this type of path dependence, we set $\mu$ to be a discrete measure
concentrated on the set $\{0 = t_0 < t_1 < \dots < t_m = T\}$. 

Relative to the function $A_k$ we assume (for notational convenience, we represent it as $A_k(t,x)$, where $x\in\Rnu^{m+2}$)
that it is infinitely differentiable in $x$; the first-order derivatives are bounded and
the higher-order derivatives have at most polynomial growth in $x$ uniformly in $t$. Furthermore, $A_k$ and all its derivatives in $x$
 are differentiable in $t$ and the respective derivatives have at most polynomial growth in $x$ uniformly in $t$.
Finally, all the aforementioned derivatives are continuous in all arguments.

\subsubsection*{Coefficients depending on coordinates in a Schauder basis}
Let $\mu = \la$.
It is known that $L_p([0,T]\to\Rnu^n, \la)$ possesses a Schauder basis. 
Let $b_i^*$, $i=1,2, \ldots$, be coordinate functionals  and let $\al_i = \<b_i^*,x\>$, where $x\in\mc E_p$.
Define $\sg_k(t,x) = A_k(t,\al_1,\al_2, \ldots)$, where $A_k$ has partial derivatives of all orders w.r.t. all $\al_i$.
For simplicity, consider the case when $A_k$ depends explicitly just on a finite number of coordinates $\al_i$.
Assume that $A_k$ is infinitely-differentiable in all arguments $\al_i$ and that
the first-order partial derivatives are bounded. 
Also assume that
all the derivatives with respect to $\al_i$ are differentiable with respect to $t$. Furthermore, we assume that
all the aforementioned derivatives are continuous and have at most polynomial growth in all arguments $\al_i$
uniformly in $t$.

 \subsubsection*{Why is it difficult to treat SDEs containing $\bm{\sup_{[0,t]}|x(s)|}$?}
 Suppose $\sg_k(t,x,y) = A_k(t,\sup_{[0,t]}|x(s)|, y)$, where $x\in \C([0,T],\Rnu^n)$.
  First, we note that $\sup_{[0,t]}|x(s)|$ is not Fr\'echet differentiable as a map $\C([0,T],\Rnu^n) \to \Rnu$.
 Indeed, according to \cite{banach}, it is not even Gateaux differentiable except at some points (called peaking functions \cite{cox}).
  
 This implies that $\C([0,T],\Rnu^n) \oplus \Rnu^n$ cannot be taken as the space of the lift.
One may think of lifting \rf{sde} to $\mc E_p$ and replacing $\sup_{[0,t]}|x(s)|$ with ${\rm ess \, sup}_{[0,t]}|x(s)|$. 
However, according to \cite{mazur},
  ${\rm ess\, sup}_{[0,t]}|x(s)|$  as a function $L_\infty([0,T],\Rnu^n)  \to \Rnu$ 
is not Gateaux differentiable in some directions  at any point $x \in L_\infty([0,T],\Rnu^n)$,
where $L_\infty([0,T],\Rnu^n)$ denotes the Banach space of essentially bounded functions. 

The same applies to coefficients given by  $A_k(t,\sup_{[0,t]}x(s), y)$ in the
 case $n=1$, where $x\in \C([0,T],\Rnu)$. 
Indeed, if $x$ is a non-negative function  which is not a peaking function, then
$\sup_{[0,t]}x(s)$ is not Gateaux differentiable along some directions in $\C([0,t],\Rnu)$.

 \section{Inverse operator of the Jacobian}
 \lb{s3}
Consider the pair of SDEs  
\eq{
\lb{sde11}
U_t = \xi + \int_0^t \td b(s,U_s, V(s)) ds + \sum_{k=1}^d \int_0^t \td \sg_k(s,U_s, V(s)) dB^k_s,\\
V(t) = z + \int_0^t  b(s,U_s, V(s)) ds + \sum_{k=1}^d \int_0^t \sg_k(s,U_s, V(s)) dB^k_s,
}
where $(\xi,z)\in \mc E_p$ (the initial condition is not necessarily a constant vector).
Remark that the stochastic integrals
in the equation for $U_t$ are $H$-valued, and thus, \rf{sde11} is well-defined.
Due to the boundedness of the Fr\'echet derivatives $\pl_x\sg_k$ and $\pl_x b$ (Assumption  (A1)) 
and inequality \rf{norms} relating the norms in $\mc E_p$ and $\mc H$, 
it is straightforward to prove that \rf{sde11} has a unique solution $(U_t,V(t))$ in $\mc E_p$.
Also remark that if $(\xi,z) = (x_0,x_0)$,  then $U_t = X^t$ and $V(t) = X(t)$. 
We will be interested in the Fr\'echet derivative of $X_t$ at point $(x_0,x_0)$ with respect to vectors from $\mc E_p$.

\subsection{Remarks on stochastic integration in Banach spaces}
In what follows, we will be dealing with SDEs in 2-smooth Banach spaces and $\gm$-radonifying operators,
so we recall the definitions.

A Banach space $E$ is called 2-smooth if there exists a constant $C>0$ such that for all $x,y\in E$,
\begingroup
\setlength{\belowdisplayskip}{4pt}
\setlength{\abovedisplayskip}{2pt} 
\aa{
\|x+y\|^2 + \|x-y\|^2 \lt 2\|x\|^2 + C\|y\|^2.
}
\endgroup
A Banach space $E$ is called martingale-type 2 if there exists a constant $\bar C>0$ such that for any $E$-valued 
martingale $\{M_n\}_{n\in\Nnu}$, it holds that
\begingroup
\setlength{\belowdisplayskip}{2pt}
\setlength{\abovedisplayskip}{6pt} 
\aaa{
\lb{mt2}
\sup_k \E\|M_k\|^2 \lt \bar C \sum_k \E\|M_k-M_{k-1}\|^2.
}
\endgroup
It was proved in \cite{pisier} that $E$ is 2-smooth if and only if it is martingale-type 2.

Let  $F$ be a (real separable) Hilbert space and $E$ be a Banach space. 
We call an operator $T\in \mc L(F,E)$ $\gm$-radonifying  (cf. \cite{VN}) if for any orthonormal basis
$\{e_n\} \sub F$ and for any sequence $\{\gm_n\}$ of independent standard Gaussian random variables, 
\begingroup
\setlength{\belowdisplayskip}{3pt}
\setlength{\abovedisplayskip}{3pt} 
\aa{
\|T\|^2_{\gm(F,E)} = \E\Big\|\sum_{n=1}^\infty \gm_n T e_n\Big\|_E^2 < \infty.
} 
\endgroup
It can be shown that the expression on the right-hand side is independent of the choice of $\{e_n\}$ and $\{\gm_n\}$, and defines a norm
in the space of $\gm$-radonifying  operators. We denote this space by
$\gm(F,E)$. It is known that (see \cite{VN}) $\gm(F,E)$ is a Banach space which is identically imbedded into $\mc L(F,E)$,
and the imbedding operator is contractive.

\subsubsection{Stochastic integration in a 2-smooth Banach space}
Let $E$ be a 2-smooth Banach space and $F$ be a Hilbert space. 
A stochastic integral $\int_0^t \xi_s dW_s$ can be defined for a progressively measurable integrand $\xi_s$
with values in $\gm(F,E)$ and an $F$-cylindrical Brownian motion  $W_s$.
Let $\mc P = \{0=s_0<s_1< \dots < s_N = t\}$ be a partition.
Define the integral for a simple $\gm(F,E)$-valued integrand $\zeta_s = \sum_{k=0}^{N-1} \zeta_{s_k} \ind_{(s_k,s_{k+1}]}(s)$
($\zeta_{s_k}$ is $\mc F_{s_k}$-measurable)
as follows:
\begingroup
\setlength{\belowdisplayskip}{3pt}
\setlength{\abovedisplayskip}{1pt} 
\aa{
\int_0^t \zeta_s dW_s =  \sum_{k=0}^{N-1} \zeta_{s_k} (W_{s_{k+1}} - W_{s_k}),
}
\endgroup
where $(f\ox x)W_t = (W_t f)\ox x$ for $f\ox x\in F\ox E$.
Further define the class $\mc M_2(\gm(F,E))$ of progressively measurable processes $[0,T] \x \Om \to \gm(F,E)$ such that
 $\E\int_0^T \|\zeta_s\|_{\gm(F,E)}^2 ds < \infty$.
 Proposition \ref{pro8999} below follows from Theorems 4.6 and 4.7 in \cite{VN1}
 (see also Proposition 2.1 in \cite{ZB}).
 \begin{pro}
\lb{pro8999}
 Let $E$ be a 2-smooth Banach space and $F$ be a Hilbert space.
Then, the integral $\int_0^t \zeta_s dW_s$ can be uniquely extended 
from the space of simple integrands to $\mc M_2(\gm(F,E))$. Furthermore,
$\int_0^t \zeta_s dW_s$  has a continuous version, and for each $q\gt 2$,
there exists a constant $C_q>0$ such that 
 \aa{
 \E\sup_{t\in [0,T]}\Big\|\int_0^t \zeta_s dW_s\Big\|^q_E \lt C_q\, \E \Big(\int_0^T \|\zeta_s\|^2_{\gm(F,E)} \, ds\Big)^{\frac{q}2}.
 }
\end{pro}
\begin{rem}
\rm In what follows, the above inequality will be useful in the particular case when $W_s$ is a one-dimensional
standard Brownian motion and $F=\Rnu$. Namely, if $E$, as before, is a 2-smooth Banach space and $\xi_s$
takes values in $E$, it holds that 
\aaa{
\lb{f7777}
 \E\sup_{t\in [0,T]}\Big\|\int_0^t \zeta_s dW_s\Big\|^q_E \lt C_q\, \E \Big(\int_0^T \|\zeta_s\|^2_E \, ds\Big)^{\frac{q}2}.
 }
\end{rem}
\subsubsection{Stochastic integration in $\mc L(\mc E_p,\mc H)$}
\lb{st-int3}
Here we show how one can define an $\mc L(\mc E_p,\mc H)$-valued stochastic integral $\int_0^t \zeta_s dB^1_s$ 
for the class of integrands $\zeta_s$ with values in $\mc L(\mc E_p,\mc H)$ possessing the property that
$\zeta^*$  is a progressively measurable map $[0,T] \x \Om \to \gm(\mc H, \mc E_p^*)$. 

Remark that
$\mc E_p^*$ is a 2-smooth Banach space. From the definition of the norm in $\gm(\mc H,\mc E^*_p)$, it follows that
the latter is also a 2-smooth Banach space.
By Proposition \ref{pro8999}, the stochastic integral $\int_0^t \zeta_s^* dB^1_s$ is well-defined 
as an element of $\gm(\mc H, \mc E_p^*)\sub \mc L(\mc H, \mc E^*_p)$. Therefore, we define
\aaa{
\lb{LEpH}
\int_0^t \zeta_s dB^1_s = \Big(\int_0^t \zeta_s^* dB^1_s\Big)^*.
}
Now let $\zeta_t = \sum_{k=1}^d \zeta^k_t \ox e_k$ be 
such that  for each $k$, $(\zeta^k)^*$  is a progressively measurable map $[0,T] \x \Om \to \gm(\mc H, \mc E_p^*)$.
The stochastic integral with respect to a $d$-dimensional Brownian motion $B_t$ is then defined as follows:
$\int_0^t \zeta_s dB_s = \sum_{k=1}^d  \Big(\int_0^t (\zeta^k_s)^* dB^k_s\Big)^*$.

\subsection{SDEs for  the Jacobian and its inverse}
\lb{sde-inverse}
Let $X_t$ be the solution to \rf{sde1}.
Consider the SDE  
\aaa{
 \lb{eq-y}
Y_t =I + \int_0^t \pl_x\hat b(s,X_s) Y_s\, ds + \int_0^t \pl_x \hat\sg(s,X_s) Y_s dB_s.
}
Here $I$ is the identity operator on $\mc E_p$ and
$\hat\sg(\fdot) = \sum_{k=1}^d \hat\sg_k(\fdot) \ox e_k$, where
 $\{e_k\}$ is the standard basis of $\Rnu^d$.
 The SDE \rf{eq-y} is regarded as $\mc L(\mc E_p)$-valued.
 Remark that a priori we do not know why the stochastic integral in the above equation makes sense. 
 However, due to Proposition \ref{pro90-},  following below,  the stochastic integral in \rf{eq-y} is well-defined
 in the sense of paragraph \ref{st-int3}.
 \begin{rem}
\rm
Although it is possible to prove that $Y_t$, defined as the solution to \rf{eq-y}, is, indeed, the Fr\'echet derivative 
of $X_t$ at point $(x_0,x_0)$
with respect to vectors from $\mc E_p$, it is only important for us that $Y_t$ is the solution to \rf{eq-y}. We will refer to $Y_t$
as the Jaco\-bian of $X_t$.
\end{rem}
\begin{rem}
\lb{r32}
\rm
Remark that $Y_t \vec{e_i}{e_i} = \vec{\pl_i X^t}{\pl_i X(t)}$.
This makes a connection between the infinite-dimensional Jacobian of $X_t$ and the finite-dimensional Jacobian of $X(t)$.
\end{rem}
\begin{pro}
\lb{pro90-}
Assume (A1). Let 
$X_t$ be a continuous version of the solution to  \rf{sde1}.
Then, for all $q\gt 2$, there exists a unique solution $Y_t$ to \rf{eq-y} in $\mc S_q([0,T],  \mc L(\mc  E_p))$.
In particular, the stochastic integral in \rf{eq-y} is  well-defined as an element of $\mc L(\mc E_p,\mc H)$.
Moreover, the solution to \rf{eq-y} takes the form $Y_t = I + \td Y_t$, where $\td Y_t^*$ is in
 $\mc S_q([0,T],\gm (\mc H, \mc E_p^*))$. 
\end{pro}
To prove Proposition \ref{pro90-} and following below Proposition \ref{pro4}, we need Lemmas \ref{lem18}, \ref{lem79}, and Corollary \ref{lem78}.
\begin{lem}
\lb{lem18}
Let $p\in (1,2]$, $A\in \mc L(\mc E_p,\Rnu^n)$, $t\in (0,T]$. Define $\hat A = \vec{\ind_{[t,T]}(\fdot) A}{A}$. Then,
$\hat A\in \mc L(\mc E_p,\mc H)$ and $\hat A^* \in \gm(\mc H,\mc E_p^*)$. 
Moreover,
\aaa{
\lb{A-b}
 \|\hat A^*\|_{\gm(\mc H,\mc E_p^*)} \lt  \sqrt{n(\mu[0,T] + 1)}\, \|A\|_{\mc L(\mc E_p,\Rnu^n)}. 
}
\end{lem}
\begin{proof}
Take $h=\vec{h_1}{h_2} \in \mc H$. The explicit computation of the adjoint 
operator $\hat A^* \in \mc L(\mc H,\mc E_p^*)$ implies that 
\aaa{
 \lb{ad67}
\hat A^* h =  \td A^*h_1 + A^* h_2 = A^*\Big(\int_t^T h_1(s)\mu(ds) + h_2\Big),
}
where $\td A = \ind_{[t,T]}(\fdot) A \in \mc L(\mc E_p, H)$.
Note that the operator
\aaa{
\lb{op12}
S(t): \; \mc H \to \Rnu^n, \quad h \mto \int_t^T h_1(s)\mu(ds) + h_2
}
 is in $\mc L(\mc H,\Rnu^n)\simeq  (\mc H^*)^n \simeq \mc L_2(\mc H,\Rnu^n) =  \gm(\mc H,\Rnu^n)$,
 where $\mc L_2(\mc H,\Rnu^n)$ is the space of Hilbert-Schmidt operators $\mc H\to \Rnu^n$.
  Hence,
  by the ideal property of $\gm$-radonifying operators (see \cite{VN}), $\hat A^* \in \gm(\mc H,\mc E_p^*)$. From the definition of the norm
  in $\mc L_2(\mc H,\Rnu^n)$, it is straightforward to obtain that
\aaa{
\lb{t-norm}
\|S(t)\|_{\mc L_2(\mc H,\Rnu^n)}^2 = n(\mu[t,T] + 1). 
}
This implies \rf{A-b}.
\end{proof}
\begin{cor}
\lb{lem78}
Let the maps $\sg_i(t,\fdot):\mc E_p \to \Rnu^n$ be Fr\'echet differentiable 
for each $i=1,\ldots, d$ and $t\in [0,T]$. Then, $\pl_x\hat \sg_i(t,x)^* \in \gm(\mc H, \mc E_p^*)$ for each  $(t,x) \in [0,T]\x \mc E_p$.
Furthermore, inequality \rf{A-b} holds with $A=\pl_x\sg_i(t,x)$. 
If, moreover, the second Fr\'echet derivative $\pl^2_x\sg_i(t,\fdot)$ exists, then for all $i,j = 1,\ldots, d$ and $(t,x) \in [0,T]\x \mc E_p$,
$(\pl^2_x\hat \sg_i(t,x) \hat\sg_j(t,x))^*  \in \gm(\mc H, \mc E_p^*)$.
 \end{cor}
\begin{proof}
Note that $\pl_x\sg_i(t,x) \in \mc L(\mc E_p,\Rnu^n)$. By Lemma \ref{lem18}, $\pl_x\hat \sg_i(t,x)^* \in \gm(\mc H, \mc E_p^*)$
and inequality \rf{A-b}  holds with $A=\pl_x \sg_i(t,x)$, $\hat A = \pl_x\hat \sg_i(t,x)$. Next, we note that 
 $\pl^2_x\sg_i(t,x)\hat \sg_j(t,x) \in \mc L(\mc E_p,\Rnu^n)$. By Lemma \ref{lem18}, $(\pl^2_x\hat \sg_i(t,x)\hat  \sg_j(t,x))^* 
 \in \gm(\mc H, \mc E_p^*)$.
 \end{proof}
\begin{lem}
\lb{lem79}
Under the assumptions of Corollary \ref{lem78},
for each $k=1,\ldots, d$,  $(t,x) \in [0,T]\x \mc E_p$, $\pl_x \hat \sg_k(t,x)^* \in \mc L(\mc E_q^*, \mc E_p^*)$ for all $q>1$.
Furthermore, the norm of $\pl_x \hat \sg_k(t,x)^*$  is  bounded uniformly in $(t,x)$.
\end{lem}
\begin{proof}
Note that we can also regard $\hat\sg_k$ as a map $[0,T]\x \mc E_p \to \mc E_q$ for all $q>1$.
Further note that the adjoint operator $\pl_x \hat \sg_k(t,x)^*  \in \mc L(\mc E_q^*, \mc E_p^*)$ takes the form \rf{ad67} with $A^*=\pl_x\sg(t,x)^*$; however,
$S(t)$ in \rf{op12} is understood as an operator $\mc E_q^* \to \Rnu^n$, 
and the explicit computation shows that its norm is smaller than $\mu[0,T]^\frac1{q}+1$. This implies
that
$\|\pl_x \hat \sg_k(t,x)^*\|_{\mc L(\mc E_q^*, \mc E_p^*)} \lt \big(\mu[0,T]^\frac1{q}+1\big) \|\pl_x\sg_k(t,x)\|_{\mc L(\mc E_p,\Rnu^n)}$.
\end{proof}
\begin{proof}[Proof of Proposition \ref{pro90-}]
Consider the SDE for the adjoint operator $Y_t^*$ 
\aaa{
\lb{sde-z-p2} 
Y_t^* =I + \int_0^t Y_s^* \pl_x \hat b(s,X_s)^*  \, ds 
+\sum_{k=1}^d \int_0^t  Y_s^* \pl_x \hat \sg_k(s,X_s)^*  d B^k_s, \quad t\in [0,T].
}
We prove that \rf{sde-z-p2} possesses a unique solution in $\mc L(\mc E_p^*)$.
Since, by Corollary \ref{lem78},  $\pl_x \hat \sg_k(s,X_s)^*$  takes values in $\gm(\mc H, \mc E_p^*)$,
the stochastic integral in \rf{sde-z-p2} and the equation itself are well-defined.
Next, \rf{sde-z-p2} can be written with respect to $\td Y_t^* = Y_t^* - I$ as follows:
\mmm{
\lb{tdy}
\td Y_t^* =   \int_0^t  \td Y_s^* \pl_x \hat b(s,X_s)^*  \, ds  +\sum_{k=1}^d \int_0^t  \td Y_s^* \pl_x \hat \sg_k(s,X_s)^*   d B^k_s\\
- \int_0^t \pl_x \hat b(s,X_s)^*  \, ds - \sum_{k=1}^d \int_0^t   \pl_x \hat \sg_k(s,X_s)^*  d B^k_s.
}
Let us prove that \rf{tdy} has a unique solution in $\gm(\mc H,\mc E_p^*)$.    
Since $\gm(\mc H, \mc E_p^*)\sub \mc L_2(\mc H)$,
in \rf{tdy}, we regard $\pl_x \hat \sg_k(s,X_s)^*$  and  $\pl_x \hat b(s,X_s)^*$ in the first two terms 
as elements of $\mc L_2(\mc H)$, and in the last two terms as elements of  $\gm(\mc H, \mc E_p^*)$. 
By inequality \rf{f7777} and
standard arguments, we obtain the existence
and uniqueness of solution $\td Y_s^*$ to \rf{tdy} in the space $\mc S_q([0,T], \gm(\mc H,\mc E_p^*))$
 for all $q\gt 2$. Hence, $\td Y_t$ a.s. takes values in $\mc L(\mc E_p, \mc H)$; therefore, $Y_t = I + \td Y_t$ a.s.
takes values in $\mc L(\mc E_p)$.
We have
\aa{
\E\sup_{s\in[0,T]}\|\td Y_s\|^q_{\mc L(\mc E_p,\mc H)} = 
\E\sup_{s\in[0,T]}\|\td Y^*_s\|^q_{\mc L(\mc H,\mc E_p^*)}  
\lt \E\sup_{s\in[0,T]}\|\td Y^*_s\|^q_{\gm(\mc H,\mc E_p^*)}  
<\infty,
}
and hence, $\E\sup_{s\in [0,T]}\|Y_s\|^q_{\mc L(\mc E_p)}<\infty$.

Let us show that $Y_t$ verifies \rf{eq-y} in $\mc L(\mc E_p)$. 
Note that the stochastic integral in \rf{eq-y} 
is well-defined as an element of
$\mc L(\mc E_p,\mc H)$ in the sense of paragraph \ref{st-int3}.  
Indeed,  $Y_s^*\pl_x \hat \sg_k(s,X_s)^*$  a.s. takes values in $\gm(\mc H, \mc E_p^*)$,
and, moreover, it is adapted and continuous since all functions involved in the this expression are continuous.
Indeed, $Y^*_s$ and $X_s$ are continuous a.s.; furthermore, the indicator function $\ind_{[s,T]}$
is continuous by Lemma \ref{lem12} in Subsection \ref{ind-f}.
Thus,  $\int_0^t \pl_x \hat \sg_k(s,X_s) Y_s  d B^k_s$ is understood as  
$\big(\int_0^t Y_s^*\pl_x \hat \sg_k(s,X_s)^* d B^k_s\big)^*$.
Therefore, adjoining \rf{sde-z-p2}, we obtain \rf{eq-y} in $\mc L(\mc E_p)$. 
The solution to \rf{eq-y} will be unique since the solution to \rf{tdy} is unique.
Remark that all integrals in equation \rf{eq-y}
are $\mc L(\mc E_p,\mc H)$-valued, and the equation is $\mc L(\mc E_p)$-valued only due to the presence of the identity operator.
\end{proof}
\begin{pro}
\lb{pro90}
Under the assumptions of Proposition \ref{pro90-}, for each $q\gt 2$, there exists a unique solution $Y_t$ to \rf{eq-y} 
in $\mc S_q([0,T], \mc L(\mc H))$. Furthermore, it holds that $Y_t = I + \td Y_t$, 
where $\td Y_t$ is in $\mc S_q([0,T], \mc L_2(\mc H))$.  Moreover,
the solution to \rf{eq-y} in $\mc S_q([0,T],  \mc L(\mc  H))$ is the restriction to $\mc H$ (as a linear bounded operator) 
of the solution to the same equation in $\mc S_q([0,T],  \mc L(\mc  E_p))$.
\end{pro}
\begin{proof}
Let $Y_{_{\displaystyle\fdot}}\!\in \mc S_q([0,T],  \mc L(\mc  E_p))$ be the solution to \rf{eq-y} whose existence is proved in Proposition \ref{pro90-}.
Since $Y_t^* = I + \td Y_t^*$, where $\td Y^*_t\in \gm(\mc H, \mc E^*_p) \sub \mc L_2(\mc H)$, 
 the restriction $Y_t$  to $\mc H$ takes the form $Y_t = I + \td Y_t$, where $\td Y_t$ a.s.
takes values in $\mc L_2(\mc H)$. 

To prove the uniqueness, note that if $Y^1_t$ and $Y^2_t$ are two solutions to \rf{eq-y} in $\mc S_q([0,T],  \mc L(\mc  H))$,
then $Y^1_t - Y^2_t$ solves a linear equation in the Hilbert space $\mc L_2(\mc H)$ with bounded coefficients and zero initial 
condition, which implies that $Y^1_t = Y^2_t$ a.s.
\end{proof}

Equation \rf{eq-y} allows to write the SDE for the inverse operator $Z_t$ of $Y_t$
 \aaa{
 \lb{inverse}
Z_t = I - \int_0^t Z_s\,  \hat \sig(s,X_s)  ds - \int_0^t   Z_s \,\pl_x \hat \sg(s,X_s)dB_s,
\quad t\in [0,T],
}
where 
\begingroup
\setlength{\belowdisplayskip}{0pt}
\setlength{\abovedisplayskip}{0pt} 
\aaa{
\lb{hsig}
\hat \sig =    \pl_x \hat b - \sum_{k=1}^d  \pl_x \hat \sg_k  \pl_x \hat \sg_k.
}
\endgroup
\begin{pro}
\lb{pro4}
Assume (A1). Let 
$X_t$ be a continuous version of the solution to \rf{sde1}.
Then, for all $q\gt 2$, there exists a unique solution $Z_t$ to \rf{inverse} in $\mc S_q([0,T],  \mc L(\mc  E_p))$.
In particular, the stochastic integral in \rf{inverse}
is well-defined as an element of
$\mc L(\mc E_p,\mc H)$. 
Moreover, $Z_t = I + \td Z_t$ where $\td Z_t^*$ is in
 $\mc S_q([0,T],\gm (\mc H, \mc E_p^*))$.
\end{pro}
\begin{proof}
Consider the SDE for the adjoint operator $Z_t^*$ on $[0,T]$
\aaa{
\lb{sde-z*}
Z_t^* =I - \int_0^t \hat \sig(s,X_s)^* Z_s^* \, ds 
-\sum_{k=1}^d \int_0^t \pl_x \hat \sg_k(s,X_s)^* Z_s^*   d B^k_s.
}
Rewriting it  with respect to $\td Z_t^* = Z_t ^*- I$, we obtain
 \mmm{
 \lb{eq-tdz}
 \td Z_t^* =   - \int_0^t \hat \sig(s,X_s)^* \td Z_s^* \, ds 
-\sum_{k=1}^d \int_0^t \pl_x \hat \sg_k(s,X_s)^* \td Z_s^*  d B^k_s \\
- \int_0^t \hat \sig(s,X_s)^*  ds 
-\sum_{k=1}^d \int_0^t \pl_x \hat \sg_k(s,X_s)^*  d B^k_s.
} 
Consider this equation in $\gm(\mc H,\mc E^*_p)$.  Note that
$\gm(\mc H, \mc E^*_p) \sub \mc L(\mc H, \mc E^*_p) \sub \mc L(\mc E^*_p)$.
Taking this observation into account, we interpret  $\hat \sig(s,X_s)^*$ and
 $\pl_x\hat \sg_{k}(s,X_s)^*$ in the first two terms as elements of  $\mc L(\mc E_p^*)$, and in the last two terms
as elements of $\gm(\mc H,\mc E_p^*)$. By the similar argument as in Proposition \ref{pro90-}, 
we obtain the existence and uniqueness of solution $\td Z_t^*$ to \rf{eq-tdz} in $\mc S_q([0,T], \gm(\mc H,\mc E^*_p))$
such that $\E\sup_{s\in [0,T]}\|Z_s\|^q_{\mc L(\mc E_p)}<\infty$.

Let us show that $Z_t  =  I + \td Z_t$ verifies \rf{inverse} in $\mc L(\mc E_p)$. 
 Since
$Z_t^* = I + \td Z_t^*$, where $\td Z^*_t\in \gm(\mc H, \mc E^*_p) \sub \mc L_2(\mc H)$,
$Z^*_t$ can be regarded as an element of $\mc L(\mc H)$. Furthermore, by Corollary \ref{lem78}, 
$\pl_x \hat \sg_k(s,X_s)^* \in \gm(\mc H,\mc E_p^*)$. This implies that $\pl_x \hat \sg_k(s,X_s)^* Z_s^*$ a.s. takes values in  $\gm(\mc H,\mc E_p^*)$.
Moreover, it is adapted and continuous by the same argument as in the proof of Proposition \ref{pro90-}.
Therefore, the stochastic integral in \rf{inverse} 
 is well-defined  in the sense of paragraph \ref{st-int3}.

Thus, adjoining  \rf{sde-z*}, we arrive at the SDE \rf{inverse} in $\mc L(\mc E_p)$. 
The solution to \rf{inverse} will be unique since the solution to \rf{eq-tdz} is unique.
\end{proof}
\begin{cor}
\lb{cor45}
Under the assumptions of Proposition \ref{pro4},
there exists a unique solution $Z_t$ to \rf{inverse} in $\mc S_q([0,T],  \mc L(\mc  H))$
such that  $Z_t = I + \td Z_t$, where $\td Z_t$ is in $\mc S_q([0,T], \mc L_2(\mc H))$.
Moreover, the solution to \rf{inverse} in $\mc S_q([0,T],  \mc L(\mc  H))$ is the restriction to $\mc H$ (as a linear  bounded operator) 
of the solution to \rf{inverse} in $\mc S_q([0,T],  \mc L(\mc  E_p))$.
\end{cor}
\begin{proof}
The proof is the same as of Proposition \ref{pro90}.
\end{proof}

Now we prove that $Z_t$ is, indeed, the inverse operator of $Y_t$ for $t\in [0,T]$.
 It is sufficient for our arguments in Section \ref{s4} to prove this fact 
 only for the case when $Y_t$ and $Z_t$ are elements of $\mc L(\mc H)$. 
 In the latter case, the stochastic integrals are $\mc L_2(\mc H)$-valued
 since $\pl_x\hat\sg_k(t,X_t)$ a.s. takes values in $\mc L_2(\mc H)$. 
 \begin{pro}
 \lb{inv34}
Assume (A1). Let $Y_t$ and $Z_t$ be solutions to \rf{eq-y} and \rf{inverse}, respectively, in $\mc L(\mc H)$.
 Then, $Y_t$ is invertible a.s. as an element of $\mc L(\mc H)$. Moreover, $Z_t$ is its inverse operator.
\end{pro}
\begin{proof}
We show that $Y_t Z_t = Z_t Y_t =  I$  on $\mc H$. 
Define $\td Y_t = Y_t -  I$, $\td Z_t = Z_t -  I$.
Remark that $Y_t$ and $Z_t$ are not Hilbert-space-valued processes while $\td Y_t$ and $\td Z_t$ are:
 they take values in $\mc L_2(\mc H)$. 
By It\^o's formula in $\mc L_2(\mc H)$, 
\mm{
\td Z_t \td Y_t = -\int_0^t Z_s \hat\sig(s,X_s) \td Y_s ds  -  \int_0^t Z_s \pl_x \hat\sg(s,X_s) \td Y_s dB_s
+ \int_0^t \td Z_s \pl_x \hat b(s,X_s) Y_s ds  \\ +  \int_0^t\td Z_s \pl_x\hat\sg(s,X_s)  Y_s dB_s
- \sum_{k=1}^d \int_0^t Z_s \pl_x \hat \sg_k(s,X_s) \pl_x \hat \sg_k(s,X_s) Y_s ds\\
= \int_0^t Z_s \hat\sig(s,X_s)ds + \int_0^t Z_s \hat\sg(s,X_s)dB_s - \int_0^t \hat b(s,X_s) Y_s ds
- \int_0^t \hat \sg(s,X_s) Y_s dB_s\\
= - Z_t - Y_t + 2 I,
}
which immediately implies that $Z_t Y_t = I$ a.s.
Likewise,
\mm{
\td Y_t \td Z_t = - \int_0^t \td Y_s Z_s \, \sig(s,X_s) ds  - \int_0^t \td Y_s Z_s\, \pl_x \hat\sg(s,X_s) dB_s
+\int_0^t \pl_x \hat b(s,X_s) Y_s \td Z_s ds  \\ +  \int_0^t \pl_x \hat\sg(s,X_s) Y_s \td Z_s dB_s
- \sum_{k=1}^d \int_0^t  \pl_x \hat \sg_k(s,X_s)Y_s Z_s \pl_x \hat \sg_k(s,X_s) ds.
}
From here, we obtain an SDE, in $\mc L_2(\mc H)$, for the process $U_t = Y_t Z_t -  I$
\mm{
U_t = - \int_0^t U_s  \sig(s,X_s) ds  - \int_0^t U_s \, \pl_x \hat\sg(s,X_s)  dB_s
+ \int_0^t \pl_x \hat b(s,X_s) U_s  ds  \\ +  \int_0^t  \pl_x\hat\sg(s,X_s) U_s dB_s
- \sum_{k=1}^d \int_0^t  \pl_x \hat \sg_k(s,X_s) U_s \pl_x \hat \sg_k(s,X_s) ds.
}
Note that $U_t = 0$ is a solution. By uniqueness, it is unique. Therefore, $Y_t Z_t =   I$ a.s.,
and hence, $Z_t$ is the inverse operator of $Y_t$ on $\mc H$.
\end{proof}

 \section{Deriving equation \rf{m-e}}
 \lb{me}
Here we derive equation \rf{m-e}. 
Once we have this equation,
the smoothness of the density can be  obtained by employing the scheme from \cite{HP2}. 
Since \rf{m-e} involves rough integrals, we need to obtain  sufficient conditions allowing to interpret stochastic
integrals in $\mc H$ and $\mc L(\mc E_p,\mc H)$ as rough integrals in $\mc E_p$ and $\mc L(\mc E_p)$, respectively,
 and the corresponding SDEs as RDEs.

\subsection{H\"older regularity of the map  ${t\mto \ind_{[t,T]}}$}
\lb{ind-f}
As we have mentioned, the choice of the Banach space 
$\mc E_p$, $p\in (1,\frac32)$, as the space of the lift
 is made, in particular, for achieving  a sufficient  H\"older regularity of the map $t\mto \ind_{[t,T]}$, which is involved
 in the definition of all maps lifted to $\mc E_p$ by formula \rf{ep-lift} and, therefore, determines their H\"older regularity.
 Due to the importance of this argument in our analysis, 
we place it in this separate subsection.
 
\begin{lem}
\lb{lem12}
The map $[0,T]\to L_p([0,T], \mu)$, $t\mto \ind_{[t,T]}$ is $\frac1{p}$-H\"older continuous 
on each interval $[s_1,s_2]$ with the property that $[s_1,s_2)\cap \supp \mu_s = \varnothing$; furthermore, 
the H\"older norm equals $1$. 
\end{lem}
\begin{proof}
Let $t, t+\Dl t\in [s_1,s_2]$.
We have
\aa{
\|\ind_{[t,T]}-\ind_{[t+\Dl t,T]}\|_{L_p(\mu)} =
\begin{cases}
\mu[t,t+\Dl t)^\frac1{p}, \quad \text{if} \;  \Dl t>0\\
\mu[t+\Dl t,t)^\frac1{p}, \quad \text{if} \;  \Dl t<0
\end{cases}
\hspace{-3mm}= |\Dl t|^\frac1{p}.
}
Indeed, if $\Dl t>0$, then $t\ne s_2$; if $\Dl t<0$, then $t\ne s_1$. Hence, $\mu_s[t,t+\Dl t) = \mu_s[t+\Dl t,t) = 0$.
\end{proof}
\begin{rem}
\rm
Remark that the map $t\mto \ind_{[t,T]}$ from Lemma \ref{lem12} is, in general, just measurable. However, due to (A2), its
restriction to $[\tau_0,\tau]$ is $\frac1{p}$-H\"older continuous. 
\end{rem}

\subsection{$\mb X_t$ and $\mb Z_t$ as controlled rough paths}
\lb{sub3.3}

We start by introducing necessary notation.
For $\al\in (\frac13,\frac12)$ and a Banach space $E$, we define the space $\ms C^\al([r,T], E)$ 
of $\al$-H\"older rough paths as the space of pairs $(\bar X, \mbb X)$ 
satisfying Chen's relation $\mbb X_{s,t} - \mbb X_{s,u} - \mbb X_{u,t} = \dl \bar X_{s,u} \ox \dl \bar X_{u,t}$
and such that
\aa{
\|\bar X\|_\al = \sup_{\substack{s,t\in [r,T]\\ s\ne t}} \frac{\|\dl \bar X_{s,t}\|_E}{|t-s|^\al} < \infty \quad \text{and} \quad 
 \|\mbb X\|_{2\al} = \sup_{\substack{s,t\in [r,T]\\ s\ne t}}  \frac{\|\mbb X_{s,t}\|_{E\ox E}}{|t-s|^{2\al}} < \infty,
}
where $\dl \bar X_{s,t} = \bar X_t - \bar X_s$. 
Define the seminorm
\aa{
\|(\bar X, \mbb X)\|_\al = \|\bar X\|_\al + \|\mbb X\|_{2\al}.
}
Further, $\ms C^\al_g([r,T], E)$ denotes the space of geometric rough paths, i.e., 
the space of pairs $(\bar X,\mbb X) \in \ms C^\al([r,T],E)$ satisfying the relation 
$2 \, {\rm Sym}(\mbb X_{s,t}) = \dl \bar X_{s,t} \ox \dl \bar X_{s,t}$.
Remark that we use the symbol $\bar X$ for the first component of an  $\al$-H\"older rough path $(\bar X, \mbb X)$
to distinguish it from the symbol $X$ which is used for the solution to the SDE \rf{sde1}.

For a fixed $\om$, we let $B_t(\om)$ be an $\al$-H\"older continuous path of the 
$d$-dimensional Brownian motion $B_t = (B^1_t, \ldots, B^d_t)$, $\al \in (\frac13,\frac12)$, and 
$\mathbf B = \mathbf B^{\textrm{Str}} = (B, \mathbb B^{\textrm{Str}})$ 
be its lift to $\ms C_g^\al([r,T],\Rnu^d)$, i.e., the path $B_t(\om)$ itself enhanced with 
$\mathbb B^{\textrm{Str}}_{s,t} (\om)= \big(\int_s^t(B_r - B_s) \ox \circ \, dB_r\big)(\om)$, 
where the integration is understood in the Stratonovich sense. 
It is known that (see \cite{FH}) that $(B, \mathbb B^{\textrm{Str}} )\in \ms C_g^\al([r,T],\Rnu^d)$ a.s. 
In addition, we will deal with an It\^o-enhanced
Brownian motion, and the corresponding lift will be denoted by 
$\mathbf B^{\textrm{It\^o}} = (B, \mathbb B^{\textrm{It\^o}})$, where  $\mathbb B^{\textrm{It\^o}}_{s,t}(\om) 
= \big(\int_s^t(B_r - B_s) \ox  dB_r\big)(\om)$.

Further let  $\C^\al([r,T], E)$ denote the space of $\al$-H\"older continuous paths
with the norm $\|\fdot\|_{\C^\al} = \|\fdot\|_\infty + \|\fdot\|_\al$. 

Finally, let  $F$ be another Banach space and 
$\bar X\in \C^\al([r,T], E)$. For $\al\in (\frac13,\frac12)$, let
$\ms D^{2\al}_{\bar X}([r,T], F)$
denote the space of controlled rough paths with respect to $\bar X$, i.e., the space of pairs
$(U,U')\in \C^\al([r,T], F) \x \C^\al([r,T], \mc L(E, F))$ such that the term $R^U_{s,t}$, given
through the relation
$\dl U_{s,t}= U'_s \, \dl \bar X_{s,t} + R^U_{s,t}$
has the property that $\|R^U\|_{2\al} = \sup_{s\ne t\in [r,T]} \frac{\|R^U_{s,t}\|_F}{|t-s|^{2\al}} < \infty$.
We call $U'$ the Gubinelli derivative of $U$ with respect to $\bar X$. 
The norm in  $\ms D^{2\al}_{\bar X}([r,T],F)$ is defined as follows:
\aa{
\|(U,U')\|_{\bar X,\al} = \|U(r)\|_F + \|U'\|_{\C^\al} + \|R^U\|_{2\al}.
}
Everywhere below, the integrals with respect to $\mbf B^{\textrm{It\^o}}$ and  $\mbf B =  \mathbf B^{\textrm{Str}}$ are understood as in \cite{FH} (subsection 4.3).

Proposition \ref{pro4.4} is the key result to interpret $X_t$ and $Z_t$ as controlled rough paths with respect to $B_t$.
More specifically, we will show that under some assumptions, an It\^o stochastic integral $I_t = \int_r^t Q_s dB_s$
has a version which is controlled by $B_t$.
Moreover, 
the integrand $Q_t$ possesses a version which a.s. coincides with the Gubinelli derivative of $I_t$.
The proof of Proposition \ref{pro4.4} follows the lines of Theorem 3.1
in \cite{FH} (Kolmogorov's criterion for rough paths); however, it is adapted to a different object: instead of the second-order process
$\mbb X_{s,t}$, 
we deal with the It\^o stochastic integral $R_{s,t} =   \int_s^t (Q_r  - Q_s) dB_r$.
 \begin{pro}
 \lb{pro4.4}
 Let $B_t$ be a $d$-dimensional Brownian motion and $Q_t$ be a stochastic process with vales in 
 $\mc L(\Rnu^d, E)$, where $E$ is a Banach space. Assume that the  It\^o stochastic integral 
 $I_t = \int_r^t Q_s dB_s$, $r\in (0,T)$, exists in $E$ and define  $R_{s,t} =  \dl I_{s,t} - Q_s \dl B_{s,t}$.
 Further assume that there exist  $q > 2$, $\beta\in (\frac1{q},\frac12]$, and a constant $C>0$
 such that for all $s,t\in [r,T]$,
 \aaa{
 \lb{3.2}
\E \|\dl Q_{s,t}\|^q_E \lt C |t-s|^{q\beta} \quad \text{and} \quad \E\|R_{s,t}\|^{\frac{q}2}_E \lt C |t-s|^{q\beta}.
 }
 Then, for all $\al\in (0, \beta - q^{-1})$, there exists  a version of $R_{s,t}$ 
 and a positive random variable $k_\al(\om)$ such that $\E|k_\al|^{\frac{q}2}<\infty$ and for all $s,t\in [r,T]$,
 \aaa{
 \lb{rst}
 \|R_{s,t}\|_E \lt k_\al |t-s|^{2\al} \quad \text{a.s.}
 }
If, moreover, $\beta - q^{-1}>\frac13$, then there exist $\al$-H\"older continuous versions of $Q_t$ and $I_t$,
where $\al\in(\frac13, \beta - q^{-1})$,  such that $I_t$ is controlled by $B_t$ and $I'_t = Q_t$ a.s.
 \end{pro}
 \begin{rem}
 \rm
 Proposition \ref{pro4.4} will be applied in the Banach spaces $E=\mc E_p$ and $E=\mc L(\mc E_p)$,
 where stochastic integrals exist either as elements of $\mc H$,
 which is a subspace of $\mc E_p$,
 or as elements of $\mc L(\mc E_p,\mc H)$, which is a subspace of $\mc L(\mc E_p)$. 
 \end{rem}
\begin{proof}[Proof of Proposition \ref{pro4.4}]
For simplicity, we take $T=1$ and $r=0$. Otherwise, we go through the same arguments with
the partition points multiplied by $T-r$ and shifted by $r$ to the right.

For $n\gt 1$, define $D_n = \{k\, 2^{-n}, k=0, \ldots, 2^n -1\}$.  Remark that the number of elements in $D_n$ is $2^n$.
Recall that for the Brownian motion $B_t$ it holds that 
 \aa{
 \E|B_t-B_s|^q = C_q |t-s|^\frac{q}2 \lt C |t-s|^{q\beta}
 }
 for all $s,t\in [0,1]$ and for some constant $C_q>0$.
 The second inequality holds since
 in \rf{3.2}, without loss of generality, we can immediately take $C>C_q$.
In what follows, whenever we deal with the norm of $E$, we skip the index.
Let us define
\aa{
K_n = \sup_{t\in D_n}\max\{\|\dl Q_{t,t+2^{-n}}\|, |\dl B_{t,t+2^{-n}}|\}, \qquad  \bar K_n = \sup_{t\in D_n} \|R_{t,t+2^{-n}}\|.
}
From \rf{3.2}, it follows that
\aa{
&\E (K^q_n) \lt \E \sum_{t\in D_n} ( \|\dl Q_{t,t+2^{-n}}\|^q +  |\dl B_{t,t+2^{-n}}|^q )
\lt 2^{n+1} C 2^{-n \beta q} = 2 C 2^{-n(\beta q -1)}, \\
&\E (\bar K^{\frac{q}{2}}_n) \lt \E \sum_{t\in D_n} \|R_{t,t+2^{-n}}\|^{\frac{q}2} \lt 2^n C 2^{-n\beta q} = C 2^{-n(\beta q -1)}.
}
Let $D= \cup_n D_n$. Take $s,t\in D$, $s<t$, and note that there exists $m\in \Nnu$ such that $2^{-m} \lt t-s < 2^{-{(m-1)}}$. Furthermore,
there exists a partition $\{s=s_0<s_1< \ldots < s_N = t\}\sub  D$ with the property that for each $i\in \{1,\ldots, N\}$, there exists $n\gt m$ such that
$s_{i+1} - s_i = 2^{-n}$  and all the partition intervals have different lengths.
Indeed, there exists $M\in\Nnu$ such that $s,t\in D_M$ and $t-s =  \sum_{i=m}^M a_i 2^{-i}$, where $a_i\in \{0,1\}$.
%
We have
\begingroup
\setlength{\belowdisplayskip}{0pt}
\setlength{\abovedisplayskip}{0pt} 
\aaa{
\lb{qst}
\|\dl Q_{s,t}\| \lt \max_{0< i \lt  N} \|\dl Q_{s,s_{i}}\| \lt \sum_{i=0}^{N-1} \|\dl Q_{s_i,s_{i+1}}\| \lt \sum_{n\gt m} K_n.
}
\endgroup
Likewise,
$|\dl B_{s,t}| \lt \sum_{i=0}^{N-1} |\dl B_{s_i,s_{i+1}}| \lt  \sum_{n\gt m} K_n$.
Next, we note that
\begingroup
\setlength{\belowdisplayskip}{4pt}
\setlength{\abovedisplayskip}{0pt} 
\mm{
\hspace{-2mm} R_{s,t} = \sum_{i=0}^{N-1}( \dl I_{s_i,s_{i+1}} -Q_s \dl B_{s_i,s_{i+1}}) 
=\sum_{i=0}^{N-1}( \dl I_{s_i,s_{i+1}} -Q_{s_i} \dl B_{s_i,s_{i+1}} + \dl Q_{s,s_i} \dl B_{s_i,s_{i+1}})\\
= \sum_{i=0}^{N-1}( R_{s_i,s_{i+1}} +  \dl Q_{s,s_i} \dl B_{s_i,s_{i+1}}).
}
\endgroup
Remark that the above equation and the following estimate are the only essential differences between this proof and the proof of Theorem 3.1
in \cite{FH}. We have
\mmm{
\lb{rst1}
\|R_{s,t}\| 
\lt  \sum_{i=0}^{N-1} \|R_{s_i,s_{i+1}}\| + \max_{0< i <N}  \|\dl Q_{s,s_i}\|  \sum_{i=0}^{N-1} |\dl B_{s_i,s_{i+1}}|\\
\lt \sum_{n\gt m} \bar K_n + \Big(\sum_{n\gt m} K_n\Big)^2.
}
By \rf{qst} and the choice of $m$, 
\begingroup
\setlength{\belowdisplayskip}{4pt}
\setlength{\abovedisplayskip}{0pt} 
\aa{
\max\Big\{\frac{|\dl B_{s,t}|}{|t-s|^\al}, \frac{\|\dl Q_{s,t}\|}{|t-s|^\al}\Big\} \lt \sum_{n\gt m} K_n 2^{m\al} \lt  \sum_{n\gt m} K_n 2^{n\al} 
\lt K_\al,
}
\endgroup
where $K_\al=  \sum_{n\gt 1} K_n2^{n\al}$. Let us prove that $K_\al$ is in $L_q$ whenever $0< \al  < \beta - q^{-1}$.
We have
\aa{
\|K_\al\|_{L_q} \lt \sum_{n\gt 1} 2^{n\al} \big(\E[K_n^q]\big)^\frac1{q} \lt (2C)^{\frac1{q}}
 \sum_{n\gt 1} 2^{-n(\beta - \frac1{q} -\al)} < \infty.
}
Further, by \rf{rst1} and, again, the choice of $m$,
\aa{
\frac{\hspace{-2mm} \|R_{s,t}\|}{|t-s|^{2\al}} \lt  \sum_{n\gt m} \bar K_n 2^{2m\al} +  \Big(\sum_{n\gt m} 2^{m\al}  K_n\Big)^2
\lt \bar K_\al + K_\al^2,
}
 where
$\bar K_\al =  \sum_{n\gt 1} \bar K_n 2^{2n\al}$. Whenever $0< \al  < \beta - q^{-1}$, we obtain that  $\bar K_\al$ is in $L_{q/2}$ since
\aa{
\|\bar K_\al\|_{\frac{q}2} \lt  \sum_{n\gt 1} 2^{2n\al} \big(\E[\bar K_n^{\frac{q}2}]\big)^{\frac2{q}}
\lt C^{\frac2{q}} \sum_{n\gt 1} 2^{-2n(\beta-\frac1{q} -\al)} < \infty.
}
Thus, \rf{rst} holds  with $k_\al = \bar K_\al + K_\al^2 \in L_{q/2}$ for $s,t\in D$, $s<t$.
Since $D$ is dense in $[0,1]$, the above estimates, combined with standard arguments, imply
the existence of an $\al$-H\"older continuous version of $Q_t$. Furthermore, we note that
\aa{
\|\dl I_{s,t}\| \lt (\|Q\|_\infty K_\al + k_\al) |t-s|^\al
}
which implies the existence of an $\al$-H\"older continuous version of $I_t$.
Finally, on $D\x D$, $R_{s,t}$
is continuous as a two-parameter process; therefore, it possesses a unique extension to $[0,1]\x [0,1]$
preserving property \rf{rst}. 
This extension represents a version of $R_{s,t}$ since it is continuous in $(s,t)$ in the $L_q$-norm.
\end{proof}
 \begin{lem}
 \lb{lem780}
 Assume (A1). Then,
for  $\al<\frac12$, there exist $\al$-H\"older continuous versions of the solution $X(t)$ to \rf{sde},
of the solution $X_t$ to \rf{sde1} in $\mc E_p$, 
and of the solution  $Z_t$ to \rf{inverse} in  $\mc L(\mc E_p)$.
Moreover, for all $q\gt 2$, 
\aa{
\E\, \big\{\|X(\fdot)\|_{\al}^q + \|X\|_{\al}^q + \|Z\|_{\al}^q\big\} < \infty.
}
 \end{lem}
 \begin{rem}
 \lb{rem340}
 \rm
 The norms $\|X\|_\al$ and $\|Z\|_\al$ are defined as in Subsection \ref{sub3.3}, with respect
 to the norms of $\mc E_p$ and $\mc L(\mc E_p)$, respectively.
 \end{rem}
 \begin{proof}[Proof of Lemma \ref{lem780}]
 The proof immediately follows from  Kolmogorov's continuity theorem  since for all $q\gt 2$,
\eqn{
\lb{qzx}
& \E|X(t) - X(s)|_{\Rnu^n}^q \hspace{-4mm} &&\lt  \E\|X_t - X_s\|_{\mc E_p}^q \lt K_1 \E\|X_t - X_s\|_{\mc H}^q \lt K_2 |t-s|^\frac{q}2,\\
 &\E\|Z_t-Z_s\|_{\mc L(\mc E_p)}^q  \hspace{-4mm} && \lt 
 K_3\,\E\|Z_t-Z_s\|_{\mc L(\mc E_p,\mc H)}^q 
 = \E\|\td Z^*_t-\td Z^*_s\|_{\mc L(\mc H,\mc E^*_p)}^q \\
& &&\lt \E\|\td Z^*_t- \td  Z^*_s\|_{\gm(\mc H, \mc E^*_p)}^q \lt K_4 |t-s|^\frac{q}2,
}
where $K_i$, $i=1,2,3,4$, are positive constants. The last inequality follows from \rf{f7777}.
Remark that since $Z_t = I +\td Z_t$, where $\td Z_t \in \mc L(\mc E_p,\mc H)$, it holds that
$Z_t - Z_s \in \mc L(\mc E_p,\mc H)$ a.s.
 \end{proof}
 In Propositions \ref{pro3} and \ref{pro5}, following below, we show that $X_t$ and $Z_t$ are
 controlled rough paths with respect to $B_t$ on $[\tau_0,\tau]$.
 \begin{pro}
\lb{pro3}
Assume (A1)--(A3).  Let $X_t$ be the $\al$-H\"older continuous version of the solution to 
\rf{sde1},  where $\al\in (\frac13,\frac12)$.
Then, for almost every $\om$, $(X(\om),X'(\om))\in \ms D^{2\al}_{B(\om)}([\tau_0,\tau], \mc E_p)$.
Moreover, for $t\in [\tau_0,\tau]$, $X'_t = \hat \sg(t,X_t)$, 
and for all $q\gt 1$, $\E\|R^X\|_{2\al}^q < \infty$. 
\end{pro}
 \begin{proof}
Let us apply Proposition \ref{pro4.4}  to show that there exists a version of the stochastic integral 
$I_t = \int_{\tau_0}^t  \hat\sg(s,X_s) dB_s$  controlled by $B_t$ on $[\tau_0,\tau]$ in $\mc E_p$.
Take $Q_t = \hat\sg(t,X_t)$, $E=\mc E_p$, $\beta = \frac12$. For all $s,t\in[\tau_0,\tau]$ and $q\gt 2$,
\aa{
\E \|\dl Q_{s,t}\|^q_{\mc E_p^d}  \lt \check K_q  \E\|\dl Q_{s,t}\|^q_{\mc H^d} \lt \mathring K_q (t-s)^{\frac{q}2},
}
where $\check K_q, \mathring K_q$ are positive constants. The second inequality holds by (A1), (A2) and due to the fact that the 
function $[\tau_0,\tau] \to L_2([0,T],\la)$,  $s\mto \ind_{[s,T]}$ 
is $\frac12$-H\"older continuous (Lemma \ref{lem12}). 
Next, for all $s,t\in[\tau_0,\tau]$, $s<t$, we  define 
$R_{s,t} = \dl I_{s,t}  - \hat\sg(s,X_s) \dl B_{s,t}$.
Then, for $q\gt 2$, by the Burkholder-Davis-Gundy inequality, 
\mm{
\E\|R_{s,t}\|^q_{\mc E_p}   \lt 
\td K_q \E \|R_{s,t}\|^q_{\mc H}  
\lt \hat K_q \sum_{k=1}^d  \Big( \int_s^t \E\|\hat \sg_k(r,X_r)- \hat \sg_k(s,X_s)\|^2_{\mc H} 
dr\Big)^\frac{q}2 \\
\lt \bar K_q (t-s)^{\frac{q}2 - 1}\hspace{-2mm}\int_s^t (r-s)^{\frac{q}2} dr
\lt K_q (t-s)^q,
}
where the constants of type $K_q$ are positive. 
By Proposition \ref{pro4.4}, 
there exists a positive random variable $k_\al$ and a version of $R_{s,t}$ 
such that $\E|k_\al|^q < \infty$ for all $q\gt 2$ and
\aaa{
\lb{xrstal}
\|R_{s,t}\|_{\mc E_p} \lt k_\al(\om) (t-s)^{2\al}.
}
Since $\dl I_{s,t}=  \hat\sg(s,X_s)  \dl B_{s,t}+ R_{s,t}$, then $I'_t = \hat\sg(t,X_t)$. 
Finally, defining
\aaa{
\lb{xrst}
R^X_{s,t} = R_{s,t} + \int_s^t \hat b(r,X_r) dr,
}
from  \rf{sde1} we obtain
\aa{
\dl X_{s,t} = \hat\sg(s,X_s)  \dl B_{s,t} + R^X_{s,t} \qquad \text{for all} \;  s,t \in  [\tau_0,\tau], \; s<t, \quad \text{a.s.}
}
Thus, we have proved that $X_t(\om)$ is controlled by $B_t(\om)$ on $[\tau_0,\tau]$ a.s. in $\mc E_p$.
Finally, $\E\|R^X\|_{2\al}^q < \infty$ for all $q\gt 1$ by \rf{xrstal} and \rf{xrst}.
\end{proof}
\begin{pro}
\lb{pro5}
Assume (A1)--(A3). Let 
$X_t$ and $Z_t$ be the $\al$-H\"older continuous versions of the solutions to \rf{sde1}  and 
\rf{inverse}, respectively, where $\al\in (\frac13,\frac12)$.
Then, for almost every $\om$, $(Z(\om),Z'(\om))\in \ms D^{2\al}_{B(\om)}([\tau_0,\tau], \mc L(\mc E_p))$. 
Moreover, $Z'_t = -Z_t \pl_x\hat\sg(t,X_t)$ for $t\in [\tau_0,\tau]$, 
and for all $q\gt 1$, $\E\|R^{Z}\|_{2\al}^q < \infty$. 
\end{pro}
\begin{proof}
We apply Proposition \ref{pro4.4} to show that
there exists a version of the stochastic integral $I_t = \int_{\tau_0}^t Z_s\pl_x\hat\sg(s,X_s)  dB_s$ 
controlled by $B_t$ on $[\tau_0,\tau]$ in $\mc L(\mc E_p)$. 
Take $Q_t = Z_t\pl_x\hat\sg(t,X_t) $, $E=\mc L(\mc E_p)$, $\beta = \frac12$. 
First, we note that for all $s,t\in[\tau_0,\tau]$,
\aa{
\|\dl Q_{s,t}\|_{\mc L(\mc E_p)^d} \lt  \|\dl Z_{s,t}\|_{\mc L(\mc E_p)} \|\pl_x\hat\sg(t,X_t)\|_{\mc L(\mc E_p)^d} +
\|Z_s\|_{\mc L(\mc E_p)}  \|\pl_x\hat\sg(\fdot,X_{\fdot})_{s,t}\|_{\mc L(\mc E_p)^d}. 
}
Since the moments of $\|Z_s\|_{\mc L(\mc E_p)}$ and $ \|\pl_x\hat\sg(t,X_t)\|_{\mc L(\mc E_p)^d}$ are bounded and
 the  function $[\tau_0,\tau] \to L_p([0,T],\la)$,  $t\mto \ind_{[t,T]}$ is $\frac1{p}$-H\"older continuous (Lemma \ref{lem12}), we obtain that 
for all $s,t\in[\tau_0,\tau]$, $s<t$, and $q\gt 2$,
\aa{
\E \|\dl Q_{s,t}\|^q_{\mc L(\mc E_p)^d}  \lt 
\mathring K_q   (t-s)^{\frac{q}2},
}
where $\mathring K_q$ is a positive constant. Next, for all $s,t\in[\tau_0,\tau]$, $s<t$, and $q\gt 2$, 
\mm{
\E\Big\|\int_s^t (Z_r\pl_x\hat\sg(r,X_r)  - Z_s \pl_x\hat\sg(s,X_s)) dB_r \Big\|^q_{\mc L(\mc E_p)} \\  \lt 
\check K_q \,\E\Big\|\int_s^t (Z_r\pl_x\hat\sg(r,X_r)  - Z_s \pl_x\hat\sg(s,X_s)) dB_r \Big\|^q_{\mc L(\mc E_p,\mc H)} \\
\lt \td K_q \sum_{k=1}^d\E \Big\|\int_s^t (\pl_x \hat \sg_k(r,X_r)^* Z_r^* - \pl_x \hat \sg_k(s,X_s)^* Z_s^*) dB^k_r\Big\|^q_{\gm(\mc H, \mc E_p^*)} \\
\lt \hat K_q \sum_{k=1}^d  \Big( \int_s^t \E\|\pl_x \hat \sg_k(r,X_r)^* Z_r^* - \pl_x \hat \sg_k(s,X_s)^* Z_s^*\|^2_{\gm(\mc H, \mc E_p^*)} 
dr\Big)^\frac{q}2\\
\lt \bar K_q (t-s)^{\frac{q}2 - 1}\hspace{-2mm}\int_s^t (r-s)^{\frac{q}2} dr
\lt K_q (t-s)^q,
}
where the constants of type $K_q$ are positive.
We used inequality \rf{f7777}
of Proposition \ref{pro8999} to obtain the third inequality. The fourth inequality follows from (A1),  inequalities \rf{qzx},
Lemma \ref{lem11}, Corollary \ref{cor45}, and from the following observation. As it was shown in the proof of Lemma \ref{lem18},
$\pl_x \hat \sg_k(t,X_t)^* = \pl_x\sg_k(t,X_t)^*\, S(t)$, where $S(t)$ is the operator defined by \rf{op12}. 
 By \rf{A-b}, for any $q\gt 2$,
\mmm{
\lb{obs}
\E\|\pl_x \hat \sg_k(r,X_r)^*  - \pl_x \hat \sg_k(s,X_s)^*\|^q_{\gm(\mc H, \mc E_p^*)} \\  \lt K_q\big(
\|S(r) - S(s)\|^q_{\mc L_2(\mc H,\mc \Rnu^n)} \, \E\|\pl_x \sg_k(r,X_r)\|^q_{\mc L(\mc E_p, \Rnu^n)}\\
+\|S(s)\|^q_{\mc L_2(\mc H,\mc \Rnu^n)} \, \E \|\pl_x \sg_k(r,X_r) - \pl_x \sg_k(s,X_s)\|^q_{\mc L(\mc E_p,\Rnu^n)}\big).
}
A straightforward computation of the Hilbert-Schmidt norm shows that 
\aa{
\|S(r) - S(s)\|_{\mc L_2(\mc H,\mc \Rnu^n)} = \sqrt{n (r-s)}.
}
Taking into account \rf{t-norm} and \rf{qzx}, we conclude that the right-hand side of
\rf{obs} is smaller than $(r-s)^\frac{q}2$ multiplied by a constant.

The rest of the proof is similar to the proof of Proposition \ref{pro3}. Namely, define 
\aa{
R_{s,t} = \int_s^t (-Z_r\pl_x\hat\sg(r,X_r)  + Z_s\pl_x\hat\sg(s,X_s)) dB_r.
}
By Proposition \ref{pro4.4}, for every $\al\in (\frac13,\frac12)$,
there exist a positive random variable $k_\al$ and a version of $R_{s,t}$ 
such that $\E|k_\al|^q < \infty$ for all $q > 1$ and
\aaa{
\lb{rst2al}
\|R_{s,t}\|_{\mc L(\mc E_p)} \lt k_\al(\om) (t-s)^{2\al}.
}
Since $\dl I_{s,t}= -Z_s  \pl_x\hat\sg(s,X_s)) \dl B_{s,t} + R_{s,t}$, we conclude that
$I_t$ is controlled by $B_t$  and $I'_t = -Z_t\pl_x\hat\sg(t,X_t)$. 
This and  \rf{inverse} imply that 
\aa{
\dl Z_{s,t} = -Z_s\pl_x\hat\sg(s,X_s)  \dl B_{s,t} + R^{Z}_{s,t} \qquad \text{for all} \;  s,t \in  [\tau_0,\tau], \; s<t, \quad \text{a.s.}
}
where 
\aaa{
\lb{zrst}
R^{Z}_{s,t} = R_{s,t} - \int_s^t Z_r \,\hat \sig(r,X_r)) dr
}
and $\hat\sig$ is defined by \rf{hsig}.
Thus, we have proved that $Z_t(\om)$ is controlled by $B_t(\om)$ on $[\tau_0,\tau]$  in $\mc L (\mc E_p)$ a.s.
Finally, $\E\|R^{Z}\|_{2\al}^q < \infty$ for all $q\gt 1$ by \rf{rst2al} and \rf{zrst}.
\end{proof}

\subsection{Stochastic integrals as rough integrals}
\lb{sub3.4}
\begin{pro}
\lb{int-coin1}
Assume $(Q(\om), Q'(\om)) \in \ms D^{2\al}_{B(\om)}([r,T],\mc L(\Rnu^d,\mc E_p))$
for $\al\in (\frac13,\frac12)$ and  for almost every $\om$.
Further assume that $Q_t$  and $Q'_t$ take values in $\mc H^d$ and $\mc H^{d^2}$ (respectively)
and that they are
 $\mc F_t/\ms B(\mc H^d)$- and, respectively,  $\mc F_t/\ms B(\mc H^{d^2})$-adapted processes
whose paths are a.s. continuous functions $[r,T] \to \mc H^d$ and $[r,T] \to \mc H^{d^2}\!$, respectively.
Finally, let the It\^o stochastic integral $\int_r^t Q_s  dB_s$ exist in  $\mc H$.
Then, 
\aaa{
\lb{h788}
 \int_r^t Q_s dB_s = \int_r^t  Q_s d\mbf B^{\text{\rm It\^o}}_s \quad \text{for all} \;\; t\in [r,T] \quad \text{a.s.}
}
\end{pro}
\begin{proof}
We use similar arguments as those in the proof of Proposition 5.1 in \cite{FH}; however, the latter proposition cannot be applied
directly to our case since we deal with Banach-space-valued rough integrals. 

Note that the rough integral $\int_r^t Q_s d\mbf B^{\text{\rm It\^o}}_s$  exists in $\mc E_p$  for all $t\in [r,T]$ by
Theorem 4.10 in \cite{FH}. For each $M>0$, introduce the stopping time $\tau_M = \inf\{r<t \lt T: \|Q'_t\|_{\mc H^{d^2}} > M\}\we T$
and show that
$\int_r^{t\we\tau_M} Q_s d\mbf B^{\text{\rm It\^o}}_s = \int_r^{t\we\tau_M} Q_s  dB_s$ a.s.
 Remark that by the assumption on the adaptedness of $Q'_t$ and the 
 continuity of  its paths, one immediately verifies that $\tau_M$ is indeed a stopping time.
Furthermore, we note that there exists a sequence of partitions 
 $\mc P_n=  \{r=s^n_0<\dots<s^n_{N(n)} = t\}$ of $[r,t]$ such that $|\mc P_n|\to 0$ as $n\to\infty$ and
$\int_r^{t\we \tau_M} Q_s dB_s = \lim_{n\to\infty} \sum_k Q_{s^n_k\we\tau_M}(B_{s^n_{k+1}\we\tau_M} - B_{s^n_k\we\tau_M})$
 a.s., where the limit is taken in $\mc H$, and, consequently, it exists $\mc E_p$.
We have to prove that
\aa{
\E\Big\| \sum_i Q'_{s_k\we\tau_M} \mbb B^{{\text{It\^o}}}_{s_k\we\tau_M,s_{k+1}\we\tau_M}\Big\|^2_{\mc E_p} = \mc O(|\mc P|),
}
where $\mc P = \{r=s_0<\dots<s_N = t\}$ is a partition of $[r,t]$. 
By \rf{norms}, we have that the left-hand side of the above identity is smaller (up to a multiplication by a positive constant)  than 
\mm{
\E\Big\| \sum_i Q'_{s_k\we\tau_M} \mbb B^{{\text{It\^o}}}_{s_k\we\tau_M,s_{k+1}\we\tau_M}\Big\|^2_{\mc H} 
\lt dM^2 \sum_i \E|\mbb B_{s_k\we\tau_M,s_{k+1}\we\tau_M}^{{\text{It\^o}}}|^2\\
\lt dM^2 \sum_{i,k}  \int_{s_k}^{s_{k+1}}\E(B^k_{r\we\tau_M} - B^k_{s_k\we\tau_M})^2 dr 
 = \mc O(|\mc P|).
}
We used the fact that since $(B^k_t)^2 - t$ is a martingale, so is $(B^k_{t\we\tau_M})^2 - t\we\tau_M$.
\end{proof}

\begin{pro}
\lb{int-coin}
Assume that $(Q(\om), Q'(\om)) \in \ms D^{2\al}_{B(\om)}([r,T],\mc L(\Rnu^d,\mc L(\mc E_p)))$, for some $r\in (0,T)$,
$\al\in (\frac13,\frac12)$,  and for almost every $\om$.
Further assume that for each $i,j \in \{1,\ldots, d\}$, $(Q_te_i)^*$  and $(Q'_t e_i \ox e_j)^*$ a.s.
take values in $\gm(\mc H,\mc E_p^*)$ and that they are
$\mc F_t/\ms B(\gm(\mc H,\mc E_p^*))$-adapted processes
whose paths are  a.s. continuous functions $[r,T] \to \mc \gm(\mc H,\mc E_p^*)$. 
Let the stochastic integral $\int_r^t Q_s  dB_s$ 
exist in $\mc L(\mc E_p,\mc H)$ in the sense of paragraph \ref{st-int3}.
 Then, 
\aaa{
\lb{h789}
 \int_r^t Q_s dB_s = \int_r^t  Q_s d\mbf B^{\text{\rm It\^o}}_s  \quad \text{for all} \;\; t\in [r,T] \;\; \text{a.s.}
}
\end{pro}
\begin{proof}
The proof goes by a similar argument as in Proposition \ref{int-coin1}.
The existence of the rough integral $\int_r^t Q_s d\mbf B^{\text{\rm It\^o}}_s$ in $\mc L(\mc E_p)$  follows from
Theorem 4.10 in \cite{FH}. Next,
for each $M>0$, introduce the stopping time 
$\tau_M = \inf\{r<t \lt T: \max_{i,j}\|(Q'_t(\om) e_i \ox e_j)^*\|_{\mc \gm(\mc H, \mc E_p^*)} > M\}\we T$
and show that, a.s.,
$\int_r^{t\we\tau_M} Q_s d\mbf B^{\text{\rm It\^o}}_s = \int_r^{t\we\tau_M} Q_s  dB_s$.
By the adaptedness of $(Q'_t e_i \ox e_j)^*$ and the 
 continuity of  its paths, one shows that $\tau_M$ is a stopping time.
Furthermore, there exists a sequence $\mc P_n=  \{r=s^n_0<\dots<s^n_{N(n)} = t\}$ of partitions of $[r,t]$ 
such that $|\mc P_n|\to 0$ as $n\to\infty$ and, a.s.,
\mm{
\int_r^{t\we \tau_M}  (Q_se_i)^* dB^i_s = \int_r^{t}\ind_{[r,\tau_M]}(s)  (Q_se_i)^* dB^i_s\\
= \lim_{n\to\infty} \sum_k (Q_{s^n_k \we \tau_M}e_i)^*(B^i_{s^n_{k+1}\we \tau_M} - B^i_{s^n_k\we \tau_M}),
}
where the limit is taken in $\gm(\mc H, \mc E_p^*)$,
and, consequently, it exists in $\mc L(\mc H,\mc E_p^*)$. Therefore, the limit for the adjoint operator exists 
in $\mc L(\mc E_p,\mc H)$. Namely, by \rf{LEpH}, a.s.,
\aa{
\int_r^{t\we\tau_M} \hspace{-2mm}Q_se_i\, dB^i_s = \Big(\int_r^{t\we \tau_M}   \hspace{-2mm}(Q_se_i)^* dB^i_s\Big)^*  \hspace{-2mm} 
=  \lim_{n\to\infty} 
\sum_k (Q_{s^n_k\we\tau_M}e_i)(B^i_{s^n_{k+1}\we\tau_M} - B^i_{s^n_k\we\tau_M}).
}
Note that $\mc L(\mc E_p,\mc H) \sub \mc L(\mc E_p)$; therefore, $\int_r^{t\we\tau_M}\hspace{-1mm} Q_s  dB_s$ 
is also an element of $\mc L(\mc E_p)$. It remains to prove that
\begingroup
\setlength{\belowdisplayskip}{0pt}
\setlength{\abovedisplayskip}{3pt} 
\aa{
\E\Big\| \sum_k Q'_{s_k\we\tau_M} \mbb B^{{\text{It\^o}}}_{s_k\we\tau_M,s_{k+1}\we\tau_M}\Big\|^2_{\mc L(\mc E_p)} = \mc O(|\mc P|),
}
\endgroup
where $\mc P = \{r=s_0<\dots<s_N = t\}$ is a partition of $[r,t]$.
 Define $\td Q'_t(\om)$ as an element of $\mc L(\Rnu^d \ox \Rnu^d, \mc L(\mc H, \mc E_p^*))$ such that
 $\td Q'_t(\om) e_i \ox e_j = (Q'_t(\om) e_i \ox e_j)^*$. The left-hand side of the above identity can be estimated
by
\mm{
\E\Big\| \sum_k Q'_{s_k\we\tau_M} \mbb B^{{\text{It\^o}}}_{s_k\we\tau_M,s_{k+1}\we\tau_M}\Big\|^2_{\mc L(\mc E_p,\mc H)}  =
\E\Big\| \sum_k \td Q'_{s_k\we\tau_M} \mbb B^{{\text{It\^o}}}_{s_k\we\tau_M,s_{k+1}\we\tau_M}\Big\|^2_{\mc L(\mc H, \mc E^*_p)} \\
\lt K\hspace{-1mm} \sum_{i,j=1}^d \E\Big\| \sum_k  \td Q'_{s_k\we\tau_M} e_i \ox e_j \int_{s_k\we\tau_M}^{s_{k+1}\we\tau_M}  
\hspace{-1mm} (B^i_r - B^i_{s_k\we\tau_M})dB^j_r \Big\|^2_{\gm (\mc H, \mc E^*_p)}\\
\lt K M^2 d \sum_k \int_{s_k}^{s_{k+1}}\E |B_{r\we\tau_M} - B_{s_k\we\tau_M}|^2 dr 
 = \mc O(|\mc P|),
}
where $K>0$ is a constant. 
The inequality in the last line follows from \rf{mt2}, if we consider the discrete $\mc F_{s_k}$-martingale
\aa{
M_k = \sum_{l=1}^k \int_{s_l}^{s_{l+1}} \ind_{[s_l,s_{l+1}\we\tau_M]}(r) \ind_{[r,\tau_M)}(s_l) \td Q'_{s_l} e_i \ox e_j
(B^i_r - B^i_{s_l})dB^j_r,
}
and from the fact that $\gm(\mc H, \mc E^*_p)$ is a martingale-type 2 Banach space.
\end{proof}

\subsection{Equation \rf{m-e}}
 From now on, we fix $\al = \frac1{2p}$, where $p\in (1,\frac32)$, and note that
$\al\in(\frac13,\frac12)$. 
Here, we use the results of Subsections \ref{sub3.3} and \ref{sub3.4} to obtain RDEs for $\hat V(t,X_t)$ and $Z_t$.
Applying rough It\^o's formula to $Z_t\hat V(t,X_t)$, we then derive equation \rf{m-e}. Namely, we have the following result.
\begin{thm}
\lb{thm1}
Let assumptions of Proposition  \ref{pro5} be satisfied with $\al = \frac1{2p}$. 
Further let $V: [\tau_0,\tau] \x \mc E_p \to \Rnu^n$ be a map 
with the following properties: 
\begin{itemize}
 \item[(i)] for each $t\in [\tau_0,\tau] $, $V(t,\fdot)$ is $\C^3(\mc E_p)$; 
 \item[(ii)]  on $[\tau_0,\tau] \x \mc E$, $V(t,x)$, $\pl_x V(t,x)$, $\pl^2_xV(t,x)$, and $\pl^3_xV(t,x)$
 are differentiable with respect to $t$; all the aforementioned functions and their derivatives in $t$
are continuous over $[\tau_0,\tau]\x \mc E$ (recall that $\mc E = \C([0,T],\Rnu^n) \oplus \Rnu^n$);
\item[(iii)] on $[\tau_0,\tau] \x \mc E$, $V(t,x)$, $\pl_x V(t,x)$, $\pl^2_xV(t,x)$ have at most polynomial growth 
in $x$, i.e., they are bounded by $1+\|x\|^q_{\mc E}$ multiplied by a constant.
\end{itemize}
Then, equation \rf{m-e} holds true.
\end{thm}

Our argument  is divided into three steps outlined below.

%
 
 \subsubsection{Rough It\^o formula for maps of the form \rf{ep-lift}}
Proposition \ref{pro675}, following below, is a version of the rough It\^o formula for maps of the form \rf{ep-lift}. 
The major problem of these maps is that they are not differentiable in $t$. Indeed, we cannot differentiate the map
$t\mto \ind_{[t,T]}$; however, we can consider the convergence of the respective sums to a Young integral with respect to $\ind_{[t,T]}$.
In addition, we use the classical It\^o formula and Proposition \ref{int-coin1} on the coincidence 
of the rough and stochastic integrals. This allows us not to require the map to be $\C^3_b$ in both arguments (cf. Proposition 7.6 in \cite{FH}).
 \begin{pro}
 \lb{pro675}
 Assume (A1)--(A3).
 Let $V:[\tau_0,\tau] \x \mc E_p \to \Rnu^n$ be a map satisfying assumptions (i)--(iii) of Theorem \ref{thm1}.
 Further let $X_t$ be the $\frac1{2p}$-H\"older continuous version of the solution to \rf{sde1}.
 Then, the rough integral $\int_{\tau_0}^t   (\pl_x\hat V\hat\sg)(s,X_s) d \mbf  B_s$  exists in $\mc E_p$. 
 Moreover, for all $t\in [\tau_0,\tau]$, it holds that
 \mmm{
 \lb{ito-4}
 \hat V(t,X_t) = \hat V(\tau_0,X_{\tau_0})
 + \vec{\int_{\tau_0}^t V(s,X_s)d\!\ind_{[s,T]}}{0}  \\
+ \int_{\tau_0}^t (\widehat{\pl_s V} + \pl_x\hat V \hat\sg_0)(s,X_s) \, ds 
+ \int_{\tau_0}^t   (\pl_x\hat V\hat\sg)(s,X_s) d \mbf  B_s \quad \text{a.s.},
 }
where $\hat\sg_0 = \hat b - \frac12 \sum_{k=1}^d \pl_x \hat \sg_k\hat\sg_k$. 
Equation \rf{ito-4} is regarded as $\mc E_p$-valued, 
and the integral with respect to $\ind_{[s,T]}$ is understood as a Young integral.
  \end{pro}
 \begin{rem}
 \rm
  \lb{rem5690}
Our choice of the space $\mc E_p$ with $p\in(1,\frac32)$ is crucial for 
formula \rf{ito-4} to make sense. First of all, $\al=\frac1{2p}$ falls into the right interval $(\frac13,\frac12)$ which makes
$(X,X')$ a controlled rough path on $[\tau_0,\tau]$  by Proposition \ref{pro3}.
Secondly, $\frac1{p} = 2\al$. Hence, by Lemma \ref{lem12},  the function $t\mto \ind_{[t,T]}$ is 
$2\al$-H\"older continuous on $[\tau_0,\tau]$. This, in turn, makes the integrand $(\pl_x\hat V\hat\sg)(t,X_t)$ a controlled rough path
with respect to $B_t$
(it is $2\al$-H\"older continuous in the first argument and differentiable in the second argument),
and at the same time, the Young integral in \rf{ito-4} is well-defined  ($3\al>1$). 
   \end{rem}
 \begin{proof}[Proof of Proposition \ref{pro675}]
 Let $\mc P = \{\tau_0=s_0<s_1 < \dots < s_N = t\}$ be a partition of the interval $[\tau_0,t]\sub [\tau_0,\tau]$.
 We have
  \mmm{
 \lb{f-89}
 \hat V(t,X_t) - \hat V(\tau_0,X_{\tau_0}) =  \vec{\sum_j V(s_j,X_{s_j}) (\ind_{[s_{j+1},T]} - \ind_{[s_{j},T]})}{0}\\
 +\sum_j (\hat V(s_{j+1},X_{s_{j+1}}) - \hat V(s_{j+1}, X_{s_{j}}))
 +  \sum_j \vec{\ind_{[s_{j+1},T]} \big(V(s_{j+1},X_{s_j}) - V(s_j,X_{s_j})\big)}
 {V(s_{j+1},X_{s_{j}}) - V(s_{j},X_{s_j})}.
  }
  The  upper component of the first term on the right-hand side converges,  in the space  in $\mc E_p$, to 
 the Young integral $\int_{\tau_0}^t V(s,X_s) d \ind_{[s,T]}$.

In the second term in \rf{f-89}, we regard the map $\hat V$ as $\mc H$-valued.
Define the map $\hat V_{\mc P}: [\tau_0,\tau]\x \mc H \to \mc H$, 
$\hat V_{\mc P}(t,x) = \hat V(\tau_0,x) + \sum_{j=0}^{N-1} \ind_{(s_{j},s_{j+1}]}(t) \hat V(s_{j+1},x)$,
and note that, by It\^o's formula in $\mc H$, the second term equals
\aa{
\int_{\tau_0}^t (\pl_x\hat V_{\mc P}\, \hat b)(s,X_s) ds +  \int_{\tau_0}^t (\pl_x\hat V_{\mc P}\hat\sg)(s,X_s) dB_s
+ \frac12 \sum_k \int_{\tau_0}^t (\pl^2_x\hat V_{\mc P}\, \hat\sg_k \hat\sg_k)(s,X_s) ds,
}
where the stochastic integral is considered in $\mc H$. 

By assumptions  (ii) and (iii) (of Theorem \ref{thm1}) and the continuity of the map $s\mto \ind_{[s,T]}$,
the above expression converges to
\aaa{
\lb{st-in}
\int_{\tau_0}^t \hspace{-1mm} (\pl_x\hat V\hat b)(s,X_s) ds + \int_{\tau_0}^t (\pl_x\hat V\hat\sg)(s,X_s) dB_s
+ \frac12 \sum_k \int_{\tau_0}^t (\pl^2_x\hat V \hat\sg_k \hat\sg_k)(s,X_s) ds.
}
By passing, if necessary, to a subsequence of partitions,
the above convergence holds  in $\mc H$ almost surely.

Finally, the last term on the right-hand side of \rf{f-89} converges to
$\int_{\tau_0}^t \widehat{\pl_s V}(s,X_s) ds$, where $\widehat{\pl_s V}$ is defined by formula \rf{ep-lift}
with respect to the map $\pl_s V$.

Let us show that the stochastic integral in \rf{st-in} coincides with the rough integral 
  $\int_{\tau_0}^t (\pl_x\hat V\hat\sg)(s,X_s)  d\mbf B^{\text{\rm It\^o}}_s$.
Define $\mc V(t,x) = (\pl_x\hat V\hat\sg)(t,x)$ for $t\in [\tau_0,\tau]$, 
and note that $\mc V$ is $\frac1{p}$-H\"older continuous in $t$.
We have to verify the assumptions of Proposition \ref{int-coin1} with $Q_t = \mc V(t,X_t)$.
First, we prove that $Q_t$ is controlled by $B_t$ a.s. and that
 $Q'_t = (\pl_x \mc V) \hat\sg(t,X_t)$. 
We have $\dl \mc V(\fdot,X_{\fdot})_{s,t} = \dl\mc V(\fdot,X_t)_{s,t} + \mc V(s, X_s + \dl X_{s,t}) - \mc V(s,X_s)$.
The first term has the order $(t-s)^{2\al}$ (with respect to the norm of $\mc E_p$).  
Next, since for each $s$, $\mc V(s,x)$ is $\C^2(\mc E_p)$ in $x$ and  continuous over $[\tau_0,\tau]\x \mc E$, we obtain that
the function $R^{\mc V}_{s,t}$, defined by the identity $\dl \mc V(\fdot,X_{\fdot})_{s,t}  = \pl_x \mc V(s,X_s) X'_s\dl B_{s,t} + R^{\mc V}_{s,t}$,
possesses the property that $\|R^{\mc V}\|_{2\al} = \sup_{s\ne t\in [\tau_0,\tau]} \frac{\|R^{\mc V}_{s,t}\|_{\mc E_p}}{|t-s|^{2\al}} < \infty$.
We furthermore take into account that $X'_t = \hat \sg(t,X_t)$ on $[\tau_0,\tau]$ 
to arrive at the aforementioned expression for $Q'_t$.
Hence, the rough integral $\int_{\tau_0}^t (\pl_x\hat V\hat\sg)(s,X_s) d\mbf B^{\text{\rm It\^o}}_s$ exists  in $\mc E_p$
for all $t\in [\tau_0,\tau]$.
We further note that $Q_t$ is $\mc H^d$-valued and $Q'_t$ is $\mc H^{d^2}$-valued.
The assumptions on the adaptedness and the continuity of paths of $Q_t$ and $Q'_t$ are obviously satisfied.
Therefore, by Proposition \ref{int-coin1},
 the rough and stochastic integrals $\int_{\tau_0}^t (\pl_x\hat V\hat\sg)(s,X_s)  d\mbf B^{\text{\rm It\^o}}_s$ 
 and $\int_{\tau_0}^t (\pl_x\hat V\hat\sg) (s,X_s) dB_s$ coincide a.s.  
 
 It remains to convert the integral with respect to $\mbf B^{\text{\rm It\^o}}_s$ into the integral with respect to
 $\mbf B_s = \mbf B^{\text{\rm Str}}_s$.
Since $\mbb B^{\text{\rm Str}}_{s,t} = \mbb B^{\text{\rm It\^o}}_{s,t} + \frac12(t-s)E$  (see \cite{FH}), 
where $E$ is the identity matrix in $\Rnu^{d\x d}$, then, by  Example 4.13 in \cite{FH}, 
\aaa{
\lb{it-st}
\int_{\tau_0}^t Q_s d \mbf B^{\text{\rm It\^o}}_s  = \int_{\tau_0}^t Q_s d \mbf B^{\text{\rm Str}}_s - \frac12 \int_{\tau_0}^t Q'_s E \,ds.
}
This implies that expression \rf{st-in} a.s. equals 
\aa{
\int_{\tau_0}^t (\pl_x\hat V\hat \sg_0)(s,X_s) ds + \int_{\tau_0}^t (\pl_x\hat V\hat\sg)(s,X_s) d \mbf B_s.
}
Remark that the rough integral exists in $\mc E_p$ while the stochastic integral in $\mc H$.
 \end{proof}
  \begin{rem}
\rm
One could think of applying  the rough It\^o formula
(Proposition \ref{rough-Ito-f} below) to $F(U_t)$, where $U_t = (t,X_t,\ind_{[t,T]})$,
\aa{
F: [\tau_0,\tau] \x \mc E_p \x \C^{2\al}([\tau_0,\tau], L_p([0,T]\to\Rnu,\la)) \to \mc E_p, \; 
(t,x,f) \mto \vec{fV(t,x)}{V(t,x)},
} 
and $\Gm_t = (t,0,\ind_{[t,T]})$. However, this
would  require from $F$ to be $\C^3$ in $t$ and $D^2F$ to be bounded, which is not the case.
Therefore, to obtain \rf{ito-4}, we use the classical It\^o formula and then apply Proposition \ref{int-coin1}
to conclude that the stochastic and rough integrals coincide.
\end{rem}

 \subsubsection{Rough differential equation for the inverse operator}

\begin{pro}
\lb{r-inv}
Let the  assumptions of Proposition \ref{pro5} be fulfilled and $\al = \frac1{2p}$.
Then, the rough integral $\int_{\tau_0}^t Z_s\pl_x\hat\sg(s,X_s) d\mbf B^{\text{\rm It\^o}}_s$ exists in $\mc L(\mc E_p)$ 
on $[\tau_0,\tau]$,
and
 $\int_{\tau_0}^t Z_s\pl_x\hat\sg(s,X_s) d\mbf B^{\text{\rm It\^o}}_s= \int_{\tau_0}^t Z_s\pl_x\hat\sg(s,X_s) dB_s$ a.s. 
 Moreover, $Z_t$ satisfies the RDE 
\aaa{
\lb{eq-z-1} 
Z_t = Z_{\tau_0} - \int_{\tau_0}^t Z_s \pl_x \hat \sg_0(s,X_s) \, ds - \int_{\tau_0}^t Z_s 
\pl_x\hat \sg(s,X_s)  d\mbf B_s.
}
\end{pro}
\begin{proof}
Since $X_t(\om)$ and $Z_t(\om)$ are controlled by $B_t(\om)$ on $[\tau_0,\tau]$,
we can conclude that $-Z_t\pl_x\hat\sg(t,X_t)(\om)$ is also controlled by $B_t(\om)$ in $\mc L(\Rnu^d, \mc L(\mc E_p))$.
Indeed, to show that $\pl_x\hat\sg(t,X_t)$ is controlled by $B_t$ in $\mc L(\mc E_p)$, we apply the same argument as in Proposition \ref{pro675}, 
but with $\mc V(t,X_t) = \pl_x\hat\sg(t,X_t)$.  
This argument, in particular, implies that the Gubinelli derivative $(\pl_x\hat\sg(t,X_t))'$ equals $(\pl^2_x\hat\sg \hat\sg)(t,X_t)$.
Let us show the coincidence of the rough and stochastic integrals $\int_{\tau_0}^t Z_s \pl_x\hat\sg(s,X_s)  d\mbf B^{\textrm{It\^o}}_s$
and $\int_{\tau_0}^t Z_s\pl_x\hat\sg(s,X_s) dB_s$ by Proposition \ref{int-coin}. Define
$Q_s = -  Z_s \pl_x\hat\sg(s,X_s)$ which implies that
$Q'_s = - Z_s \big((\pl^2_x\hat \sg) \hat \sg + \pl_x \hat\sg \pl_x \hat\sg\big)(s,X_s)$. 
Therefore,  $Q_s e_i = -Z_s\pl_x \hat\sg_i(s,X_s)$ and $Q'_s e_i \ox e_j 
= - Z_s \big((\pl^2_x\hat \sg_i) \hat \sg_j + \pl_x \hat\sg_i \pl_x \hat\sg_j\big)(s,X_s)$. By Corollary \ref{lem78} 
and the ideal property of $\gm$-radonifying operators, 
$(Q'_s e_i )^*$ and $(Q'_s e_i \ox e_j)^*$ take values in $\gm(\mc H,\mc E_p^*)$. 
The assumptions on the  adaptedness and the continuity of paths of $(Q_s e_i)^*$ and $(Q'_s e_i \ox e_j)^*$ 
as $\gm(\mc H,\mc E_p^*)$-valued maps are obviously satisfied.
Finally, we note that the existence of the stochastic integral $\int_{\tau_0}^t Z_s \pl_x\hat\sg(s,X_s) dB_s$ 
in $\mc L(\mc E_p, \mc H)$ follows from Proposition \ref{pro4}.  Thus, Proposition \ref{int-coin} implies
the coincidence of the aforementioned rough and stochastic integrals. 
Thus, we conclude that $Z_t(\om)$ solves the RDE
\aaa{
\lb{rde1-z}
Z_t = Z_{\tau_0} - \int_{\tau_0}^t Z_s \hat \sig(s,X_s) \, ds - \int_{\tau_0}^t Z_s \pl_x\hat\sg(s,X_s)  d\mbf B^{\text{\rm It\^o}}_s,
}
where $\hat \sig$ is defined by \rf{hsig}.
By \rf{it-st}, the RDEs \rf{eq-z-1} and \rf{rde1-z} are equivalent.
\end{proof}
\begin{rem}
\rm
Similar to Remark \ref{rem5690}, we note here that the choice of the space  $\mc E_p$ with $p\in(1,\frac32)$
 is crucial for  equation \rf{eq-z-1}   to be well-defined as a rough differential equation. We again note that  $\al=\frac1{2p}$
 falls into the interval $(\frac13,\frac12)$ making both $X_t$ and $Z_t$ controlled rough paths (Propositions \ref{pro3} and \ref{pro5}).
Moreover,  the map $t\mto \ind_{[t,T]}$ is $2\al$-H\"older continuous on $[\tau_0,\tau]$; therefore, $\pl_x\hat\sg(t,\fdot)$ is 
$2\al$-H\"older continuous w.r.t. the first argument.
Hence, $Z_t \pl_x\hat\sg(t,X_t)$ is a controlled rough path 
and the rough integral in \rf{eq-z-1}  makes sense.
\end{rem}
\subsubsection{Proof of Theorem \ref{thm1}}
Now we derive equation \rf{m-e}.
First, we formulate It\^o's formula for rough paths (Proposition \ref{rough-Ito-f}) which is a minor modification of It\^o's formula from \cite{FH}
 (Proposition 7.6) convenient for our application.
 \begin{pro}
 \lb{rough-Ito-f}
 Let $E_1,E_2, E_3$ be Banach spaces and $F: E_2\to E_3$ be a $\C^3$ map such that $D^2F$ is $\C^1_b$.
  Further let $\mbf X = (\bar X, \mbb X)\in \ms C^\al([0,T], E_1)$ and
 $(U,U')\in \ms D^{2\al}_{\bar X}([0,T], E_2)$ be a controlled rough path of the form
 \aa{
 U_t = \Gm_t + \int_0^t U'_s d\mbf X_s,
 }
 where $\Gm \in \C^{2\al}([0,T],E_2)$ and $(U',U'')\in \ms D^{2\al}_{\bar X}([0,T], \mc L(E_1,E_2))$. 
 Assume that $DF$ is bounded over the set \, ${\rm Im\;} U = \{U_t, \, t\in [0,T]\}$.
 Then,
 \mmm{
 \lb{ito-1}
 F(U_t) = F(U_0) + \int_0^t DF(U_s) U'_s d\mbf X_s + \int_0^t DF(U_s)d\Gm_s \\+ \frac12 \int_0^t 
 D^2 F(U_s)(U'_s,U'_s) d[\mbf X]_s,
 }
 where $[\mbf X] : [0,T] \to {\rm Sym}(E_1\ox E_1), \; t\mto [\mbf X]_t = \dl \bar X_{0,t} \ox \dl \bar X_{0,t} -
  2 \, {\rm Sym}(\mbb  X_{0,t})$. Here, $DF$ and $D^2F$ denote Fr\'echet derivatives of $F$, 
 and the last two integrals are understood as Young integrals.
  \end{pro}
\begin{proof}
In Proposition 7.6 in \cite{FH}, formula \rf{ito-1} was proved under the assumption $F\in \C^3_b$. However, if we analyze the last expression
(containing the limit) in
the proof of  Proposition 7.6, we notice that the boundedness of $DF$ on ${\rm Im\;} U$,
together with the global boundedness of $D^2F$ and $D^3F,$ is sufficient for the limit to exist and be equal to the right-hand side of \rf{ito-1}.
Indeed, the convergence to zero of lower-order terms holds if $D^3F$ is globally bounded, and $DF$, $D^2F$ are bounded
on ${\rm Im\;} U$. Further, $DF(U_s)$ and $D^2F(U_s)$ need to be $\al$-H\"older continuous which is the case if $D^2F$ and $D^3F$ 
are globally bounded. 
\end{proof}

\begin{proof}[Proof of Theorem \ref{thm1}]
Consider the second term in \rf{ito-4}.  We will view $\vec{V(s,X_s)}{0}$ as a
bounded linear operator $L_p([\tau_0,\tau]\to\Rnu,\la) \to \mc E_p$, acting on elements 
$g\in L_p([\tau_0,\tau]\to\Rnu,\la)$ as follows: $\vec{V(s,X_s)}{0} g = \vec{V(s,X_s)g}{0}$. 
Introducing notation $\mr V(s,x) = \vec{V(s,x)}{0}$, we represent 
 the second term in \rf{ito-4} as
$\int_{\tau_0}^t \mr V(s,X_s) d\ind_{[s,T]}$, where the integral is understood as a Young integral. 

We apply Proposition \ref{rough-Ito-f} to $F(Z_t, M_t)$ for $t\in [\tau_0,\tau]$, 
where $M_t =\hat V(t,X_t)$ and $F: \mc L(\mc E_p) \x \mc E_p \to \mc E_p$, $F(A, x) = Ax$.
Remark that $F$ satisfies assumptions of Proposition \ref{rough-Ito-f}. Indeed, $D^2F$ is bounded by $1$, and $D^3F = 0$. Furthermore,
$F'(Z_t,M_t)(h_1,h_2) = h_1 M_t + Z_t h_2$ is bounded on
${\rm Im\;} (Z,M)$.
Taking into account representation \rf{ito-4} for $\hat V(t,X_t)$ and the RDE  \rf{eq-z-1}  for $Z_t$,
we obtain
\mmm{
\lb{f25}
Z_t \hat V(t,X_t)   =  Z_{\tau_0}\hat V(\tau_0,X_{\tau_0}) + \int_{\tau_0}^t Z_s  \mr V(s,X_s)  d\ind_{[s,T]}\\
+ \int_{\tau_0}^t   Z_s( \widehat{\pl_s V} + [\hat \sg_0, \hat V])(s,X_s) \, ds 
+ \int_{\tau_0}^t     Z_s [\hat \sg, \hat V](s,X_s) d \mbf  B_s,
} 
where  $[\hat \sg, \hat V] = \sum_{k=1}^d [\hat \sg_k, \hat V]\ox e_k$ and $\{e_k\}_{k=1}^d$ is the standard basis of $\Rnu^d$.
Since $Z_s$ takes values in $\mc L(\mc E_p)$,  the integrand $Z_s \mr V(s,X_s)$ becomes a bounded linear operator
$L_p([\tau_0,\tau]\to \Rnu,\la) \to \mc E_p$; therefore, the second integral is well-defined.

Applying $\Pi_n Y_\tau$ to the both sides of \rf{f25}, we obtain equation \rf{m-e}. 
Indeed, taking $\Pi_n Y_\tau$ under the $ds$-integral sign is straightforward. 
Further, let $t\in [\tau_0,\tau]$ and
${\mc P} = \{\tau_0 = s_1 < \ldots < s_N = t\}$ be a partition of $[\tau_0,t]$. 
We have
\mm{
\Pi_n Y_\tau \int_{\tau_0}^t Z_s[\hat\sg,\hat V](s,X_s) d\mbf B_s 
 = \Pi_n Y_\tau \lim_{|{\mc P}|\to 0} \sum_{k=1}^{N-1} 
 \big( Z_{s_k} [\hat\sg,\hat V](s_k,X_{s_k})\dl B_{s_k,s_{k+1}}      \\   
+ \big(Z_{s_k} [\hat\sg,\hat V](s_k, X_{s_k})\big)'\mbb B_{s_k,s_{k+1}}\big),
}
where  the derivative sign in the last term  means the Gubinelli derivative.
By the continuity of $\Pi_n Y_\tau$, we can take it under the limit sign.
 A similar argument holds for the Young integral in \rf{f25}.
\end{proof}

\section{Smoothness of the density under H\"ormander's condition}
 \lb{s4}
 \subsection{Malliavin differentiability of the solution}
 \lb{ss5}

Introduce the isonormal Gaussian process 
$W(\phi) = \sum_{k=1}^d \int_0^T \phi^k(t) dB^k_t$, where $\phi\in L_2([0,T];\Rnu^d)$. 
Let $\td H$ be a separable Hilbert or a finite-dimensional space. 
For an $\td H$-valued random variable $F$,
the $k$-th order Malliavin derivative operator $D^{(k)}:L_q(\Om;\td H) \to  L_q(\Om; L_2([0,T]^k; (\Rnu^d)^{\ox k}) \ox \td H)$
is defined as in \cite{nualart}, and its domain is denoted by $\mbb D^{k,q}(\td H)$. Define
$\mbb D^\infty(\td H) = \cap_{k\gt 1, q\gt 2}\, \mbb D^{k,q}(\td H)$.
\begin{pro}
\lb{lem1}
Assume (A1). Then, the solution $X_t$ to  \rf{sde1} is in $\mbb D^\infty(\mc H)$, and 
the solution $X(t)$ to \rf{sde} is in $\mbb D^\infty(\Rnu^n)$.
Moreover, for each $i=1,\ldots, d$ and for $r\lt t$, the $i$-th component $D^i_r X(t)$ satisfies the SDE
\aaa{
\lb{sde3}
D^i_r X(t) = \sg_i(r,X_r) + \int_r^t \pl_x b(s,X_s) D^i_r X_sds + \int_r^t  \pl_x\sg(s,X_s) D^i_r X_s dB_s,
}
while  for $r>t$, $D^i_rX(t) = 0$.
Furthermore, for $r\lt t$, it holds that 
\aa{
D^i_r X(t) = \Pi_n Y_t Z_r \hat \sg_i(r, X_r).
}
\end{pro}
\begin{proof}
First, we show that $X_t$ is in $\mbb D^{1,\infty}(\mc H) = \cap_{q\gt 2}\, \mbb D^{1,q}(\mc H)$.
This follows from Theorem 2.2.1 in \cite{nualart} (even though the proof is presented for finite-dimensional
equations, it remains valid for Hilbert-space-valued equations).
Moreover, for $r\lt t$, the SDE for the $i$-th component of the Malliavin derivative takes the form 
\aaa{
\lb{sde33}
D^i_r X_t = \hat\sg_i(r,X_r) + \int_r^t \pl_x \hat b(s,X_s) D^i_r X_sds 
+ \int_r^t  \pl_x\hat\sg(s,X_s) D^i_r X_s dB_s.
}
If we substitute $Y_tZ_r \hat\sg_i(r,X_r)$ instead of $D^i_r X_t$, the above equation is satisfied.
By uniqueness, we conclude that $D^i_r X_t  = Y_t Z_r \hat\sg_i(r,X_r)$. Hence,
$D^i _r X(t) = D^i _r \Pi_n X_t = \Pi_n D^i_r X_t  = \Pi_n Y_t Z_r \hat\sg_i(r,X_r)$. 
Applying the projection $\Pi_n$ of the both parts of \rf{sde33}, we obtain \rf{sde3}.

It remains to show that
$X_t$ is in $\mbb D^\infty(\mc H)$. The latter fact immediately implies, by the chain rule, 
that $X(t)$ is in  $\mbb D^\infty(\mc \Rnu^n)$ since $X(t) = \Pi_n X_t$.
The proof is the same as of Theorem 2.2.2 in \cite{nualart}, but we would like to demonstrate that the statement is valid
for coefficients whose higher-order derivatives have polynomial growth in the second argument.

By induction on $k$, one can show that  the SDE for the 
$k$-th order Malliavin derivative $D_{r_1,\ldots, r_k}^{i_1,\ldots, i_k} X_t$
takes the form
\mmm{
\lb{eqdk}
D_{r_1,\ldots, r_k}^{i_1,\ldots, i_k} X_t =  \int_{r_1\vee\dots\vee r_k}^t  \hspace{-3mm} \pl_x \hat b(s,X_s) D_{r_1,\ldots, r_k}^{i_1,\ldots, i_k} X_s ds \\
+ \int_{r_1\vee\dots\vee r_k}^t  \hspace{-3mm} \pl_x\hat\sg(s,X_s) D_{r_1,\ldots, r_k}^{i_1,\ldots, i_k} X_s dB_s
+\mc Q,
}
where $\mc Q$ represents the terms (not containing $D_{r_1,\ldots, r_k}^{i_1,\ldots, i_k} X_t$) 
whose moments of any order $q\gt 2$ are finite.
More specifically, $\mc Q$ involves products of the derivatives
of $\hat\sg_k$ and $\hat b$, evaluated at $(s,X_s)$,
(whose moments are finite by (A1) and Lemma \ref{lem11})
and the Malliavin derivatives of $X_t$ of orders smaller than $k$,
whose moments of all orders are finite by the induction hypothesis.  
Note that for $k=1$, \rf{eqdk} is the same as \rf{sde33}. Applying Lemma 2.2.2 from \cite{nualart}
to \rf{eqdk}, we obtain that $D_{r_1,\ldots, r_k}^{i_1,\ldots, i_k} X_t$ is in $\mbb D^{1,\infty}(\mc H)$
and that $D_{r_1,\ldots, r_k, r_{k+1}}^{i_1,\ldots, i_k, i_{k+1}} X_t$ satisfies an SDE of the form  \rf{eqdk}.
\end{proof}

 \subsection{Lie brackets and H\"ormander's condition}
 Define the space $\mc {\bar E} = D([0,T], \Rnu^n) \oplus \Rnu^n$, where $D([0,T], \Rnu^n)$ is the space of c\`adl\`ag functions,
and let for each $x\in  D([0,T], \Rnu^n)$, $x_t = \vec{x^t}{x(t)}$. Recall the definition of the vertical derivative (cf. \cite{cont}, p. 130, see also \cite{DUP}).
\begin{df}
\lb{vert}
A non-anticipative map $V: [0,T]\x \mc {\bar E} \to \Rnu^l$ is called vertically differentiable at $(t,x_t)$
if the map $\dl \mto V(t,x_t + \dl h_i)$ is differentiable at $\dl = 0$
for each direction  $h_i =\vec{\ind_{[t,T]} e_i}{e_i}$, where $\{e_i\}_{i=1}^n$
is the standard basis of $\Rnu^n$. 
The matrix $\pl_v V(t,x_t)$ with the columns $\pl_\dl V(t,x_t + \dl h_i)|_{\dl=0}$,  $i=1, \ldots, n$, is called then the vertical derivative of $V$ 
at point $(t,x_t)$.
\end{df}

We define the Lie bracket $[V_1,V_2]$ for two non-anticipative maps \cite{cont}
$V_1,V_2: [0,T]\x \mc {\bar E} \to \Rnu^n$, $(t,x) \mto V_i(t,x_t)$, $i=1,2$, as follows:
\aa{
[V_1,V_2](t,x_t) = (\pl_v V_2 V_1 - \pl_v V_1 V_2)(t,x_t),
}
where $\pl_v V_i$ denotes the vertical derivative of $V_i$ in the sense of Definition \ref{vert}.
Note that if $V_1$ and $V_2$ are Fr\'echet differentiable maps $[0,T] \x \mc E_p \to \Rnu^n$, and 
$\hat V_1$, $\hat V_2$ are their $\mc E_p$-lifts, then
\begingroup
\setlength{\belowdisplayskip}{4pt}
\setlength{\abovedisplayskip}{0pt} 
\aa{
[V_1,V_2] = \Pi_n [\hat V_1,\hat V_2].
}
\endgroup
where $[\hat V_1,\hat V_2](t,x_t) = (\pl_x \hat V_2 \hat V_1- \pl_x \hat V_1 \hat V_2)(t,x_t)$ is the classical Lie bracket in $\mc E_p$.

\begin{rem}
\rm
Remark that in  \rf{sde1},
the coefficients become infinite-dimensional only through  the multiplication by $\ind_{[t,T]}(\fdot)$,
regardless of the choice of topology in the space of the lift. Thus, the nature of equation \rf{sde1}
suggests  that the Lie brackets should be defined in terms of differentiation with respect to the present.
\end{rem}

Below, we state property (H) which, by analogy with the state-dependent case, 
we call H\"ormander's condition. However, the correct non-degeneracy condition, as\-sum\-ption (A5) following below,
is somewhat stronger than (H) and implies (H).
Namely, we state H\"ormander's condition at point $(t,x_0)\in [0,T]\x \Rnu^n$ as follows:
\bi
\item[\bf (H)] For all $x \in \C([0,T],\Rnu^n)$ such that $x(0) = x_0$, the vector space spanned by the maps
\aa{
\sg_1,\ldots, \sg_d, \; [\sg_i,\sg_j], \; 1\lt i,j \lt d, \;  [\sg_i, [\sg_j,\sg_k]], \; 1\lt i,j,k \lt d, \ldots,
}
evaluated at $(t,x_t)$, is $\Rnu^n$. 
\ei
\begin{rem}
\rm Remark this is the stronger version of H\"ormander's condition that excludes the map $\sg_0
= b - \frac12 \sum_{k=1}^d \pl_x \sg_k\hat\sg_k$.
\end{rem}
\subsection{Smoothness of the density}
\lb{smooth}
Introduce the sets of maps $[0,T] \x \mc {\bar E}  \to \Rnu^n$ (where $ \mc {\bar E} = D([0,T],\Rnu^n) \oplus \Rnu^n$)
\aa{
\Sg_0 =  \{\sg_1,\ldots, \sg_d\}, \qquad \Sg_j =  \{[\sg_k, V], k=1,\ldots, d, \, V \in \Sg_{j-1} \}.
}
Remark that H\"ormander's condition (H) at point $(t,x_0)$ implies that for each $x\in \C([0,T],\Rnu^n)$, $x(0) = x_0$,
there exists a number $\mc N(t,x)$, depending on $t$ and $x$, such that the vector space spanned by 
$\cup_{j=1}^{\mc N(t,x)} \Sg_j(t, x_t)$  coincides with $\Rnu^n$. 
Furthermore, condition (H) at point $(t,x_0)$ is equivalent to the fact that 
for each $x\in \C([0,T],\Rnu^n)$, $x(0) = x_0$, there exists a number $\mc N(t,x)$ such that
\aaa{
\lb{Gm}
\inf_{|z|=1} \sum_{j=1}^{\mc N(t,x)} \sum_{V\in \Sg_j} (z,V(t,x_t))^2 > 0.
}
However, to obtain the smoothness of the density, we need a stronger
non-degeneracy assumption than  condition (H),
or  the equivalent condition \rf{Gm}.
Namely, if  we want to prove the smoothness of the density at time $\tau\in (0,T]$
for the solution to \rf{sde} with the initial condition
$x_0\in\Rnu^n$,  we assume the following (cf. Assumptions A.3, A.4 in \cite{GH}):
\bi[topsep=2pt]
\item[\bf (A5)]
\bi
\item[(i)] 
There exists a number $\mc N(\tau)>0$ and a measurable function $\te_\tau: \mc E\to (0,\infty)$ such that
for all $x\in \C([0,T],\Rnu^n)$ satisfying $x(0) = x_0$,
 \begingroup
\setlength{\belowdisplayskip}{0pt}
\setlength{\abovedisplayskip}{1pt} 
\aa{
\inf_{|z|=1}\sum_{j=1}^{\mc N(\tau)} \sum_{V\in \Sg_j} (z,V(\tau, x_\tau))^2 \gt \te_\tau(x_\tau).
}
\endgroup
\item[(ii)] If  $X(\fdot)$ is the solution to  \rf{sde} with $X(0) = x_0$, then
the inverse moments of $\te_\tau(X_\tau)$ satisfy the following condition:
There exists a constant $C>0$ and a function $\Phi: [0,T]\x \Rnu^n \to \Rnu$ such that 
for all $q\gt1$,
\aa{
\E[\te^{-q}_\tau(X_\tau)] \lt C \Phi^q(\tau, x_0).
}
\ei
\ei

Remark that (A5)-(i) implies (H).

We now state our goal result on the smoothness of the density of $X(\tau)$:

\begin{thm}
\lb{thm2}
Fix $\tau\in (0,T]$ and assume (A1)--(A5).  Then, the Malliavin covariance matrix of the solution
 $X(\fdot)$ to equation \rf{sde} at time $\tau$  satisfies the estimate
 \aaa{
 \lb{gm-}
\PP\{\inf_{|z|=1}(\gm_\tau z,z)\lt \eps\} \lt C(x_0,\tau,q)\, \eps^q,
 }
 for all $q\gt 1$ and for a constant $C$ depending on $x_0,\tau,q$. Moreover,
 $X(\tau)$ admits a smooth density with respect to Lebesgue measure on $\Rnu^n$.
\end{thm}
Note that  the statement of Theorem \ref{thm2} follows from the following proposition.
\begin{pro}
\lb{pr52}
Assume (A1)--(A4). Then, estimate \rf{gm-} is implied by estimate \rf{gm0}
fulfilled for all $q\gt 1$ and for a constant $C$ depending on $x_0,\tau,q$:
\aaa{
\lb{gm0}
\PP\{\inf_{|z|=1}(\gm^0_\tau z,z)\lt \eps\} \lt C(x_0,\tau,q)\, \eps^q,
 }
 where 
 \begingroup
\setlength{\belowdisplayskip}{0pt}
\setlength{\abovedisplayskip}{2pt} 
  \aa{
 \gm^0_\tau = \int_{\tau_0}^\tau J_{\tau,s} (\hat\sg \hat\sg^*) (s, X_s) J_{\tau,s}^*\, ds, \qquad  
 J_{\tau,s} = \Pi_n Y_\tau Z_s.
 }
  \endgroup
\end{pro}
\begin{rem}
\rm
Above,
$\hat \sg = \sum_{k=1}^d \hat \sg_k \ox e_k$, where $\{e_k\}_{k=1}^d$ is
 the standard basis of $\Rnu^d$, is regarded as an element of $\mc L(\Rnu^d,\mc E_p)$.
\end{rem}
\begin{proof}[Proof of Proposition \ref{pr52}]
By Proposition \ref{lem1}, $D^k_s X(\tau) = J_{\tau,s} \hat \sg_k(s,X_s)$ for $s\lt\tau$. Hence,
for the Malliavin covariance matrix of $X(\tau)$, it hold that
\aa{
\gm_\tau = \int_0^{\tau_0} J_{\tau,s} (\hat\sg \hat\sg^*) (s, X_s) J_{\tau,s}^*\, ds
+\gm^0_\tau,
}
Then, for all $z\in\Rnu^n$, $(\gm^0_\tau z,z) \lt (\gm_\tau z,z)$ a.s.
Therefore, \vspace{2mm}\\
 \phantom{\hspace{1.5cm}} $\PP\{\inf\limits_{|z|=1}(\gm_\tau z,z)\lt \eps\} \lt \PP\{\inf\limits_{|z|=1}(\gm^0_\tau z,z)\lt \eps\} \lt C(x_0,\tau,q)\, \eps^q$.
\end{proof}

We now proceed with the proof of \rf{gm0} and restrict our analysis to $[\tau_0,\tau]$.

Let $\mbf B = (B, \mbb B) \in \ms C_g^\al([0,T],\Rnu^d)$, as before, be a Stratonovich enhanced Brownian motion. 
It is known that the sample paths of $B$ are almost surely $\te$-H\"older rough for  every $\te>\frac12$ (Proposition 6.11 in \cite{FH}).
Let $L_\te(B)$ denote the modulus of $\te$-H\"older roughness of the Brownian motion $B$ (Definition 6.7 in \cite{FH}).

As before, $p\in (1,\frac32)$ and $\al = \frac1{2p}$. 
We fix $\te\in (\frac12,2\al)$ and define the quantity
\mm{
\mc R = 2+ |x_0| +  L_\te(B)^{-1} +\|Y_\tau\|_{\mc L(\mc E_p)} + \|(B,\mbb B)\|_\al + 
\|X\|_{\al, [0,T]} + \|Z\|_{\al, [0,T]} \\
+ \|R^X\!\|_{2\al} + \|R^{Z}\!\|_{2\al},
}
where $X$ and $Z$ are $\al$-H\"older continuous versions of the solutions to \rf{sde1} and \rf{inverse},
 respectively, and $x_0=X(0)$. 
Furthermore, $R^X$ and  $R^Z$ are naturally defined just on $[\tau_0,\tau]$;
 $\|X\|_{\al,  [0,T]}$ and $\|Z\|_{\al, [0,T]}$ are understood in the sense of Remark \ref{rem340}
 with respect to the norms of $\mc E_p$ and $\mc L(\mc E_p)$, respectively. 
Above, it is convenient to consider $\|X\|_\al$ and $\|Z\|_\al$ over  the entire interval $[0,T]$.

We have the following lemma about finite moments of $\mc R$.
\begin{lem}
Let (A1)--(A3) hold. Then, for all $q\gt 2$, $\E[\mc R^q] <\infty$.
\end{lem}
\begin{proof}
By Lemma \ref{lem780},  for all $q\gt 2$, $\E[\|X\|_{\al,[0,T]}^q + \|Z\|_{\al,[0,T]}^q] < \infty$, and by 
Propositions \ref{pro3} and \ref{pro5}, $\E[\|R^X\!\|^q_{2\al} + \|R^{Z}\!\|^q_{2\al}] < \infty$. 
Further, Proposition \ref{pro90-} implies that $\E\|Y_\tau\|^q_{\mc L(\mc E_p)} < \infty$. 
Finally, by Lemma 3 in \cite{HP2}, for all $q\gt 2$, $\E[L_\te(B)^{-q}] < \infty$.
\end{proof}


The proof of Theorem \ref{thm2}
 is based on Lemmas \ref{lem76}, \ref{lem5}, \ref{lem72}, and  \ref{lem7}.
  \begin{lem}
\lb{lem76}
Let (A1)--(A4) hold. Then, there exist constants $p_0,q_0>0$ such that
for all $z\in\Rnu^n$, $|z| =1$, and  for all initial conditions $x_0\in\Rnu^n$, 
\aaa{
\lb{i0}
\|\big(z,J_{\tau,_{\displaystyle\fdot}} 
\hat \sg(\fdot,X_{_{\displaystyle\fdot}})\big)\|_{\infty, [\tau_0,\tau]} \lt  \mc R^{q_0} (z,\gm^0_\tau z)^{p_0}.
}
\end{lem}
\begin{rem}
\rm
$\big(z,J_{\tau,_{\displaystyle\fdot}} 
\hat \sg(\fdot, X_{_{\displaystyle\fdot}})\big)$ denotes the vector whose $k$-th component ($k=1,\ldots, d$)
is $\big(z,J_{\tau,_{\displaystyle\fdot}} \hat \sg_{k}(\fdot, X_{_{\displaystyle\fdot}})\big)$.
\end{rem}
\begin{rem}
\lb{rem5.5}
\rm
Everywhere below,
$\|X\|_\al$, $\|X\|_\infty$, $\|X(\fdot)\|_\infty$, $\|Z\|_\al$, and $\|Z\|_\infty$ are understood
over $[0,T]$ (in particular, $\|X(\fdot)\|_\infty = \sup_{t\in[0,T]}|X(t)|$). We skip the lower index $[0,T]$
in these quantities for simplicity of notation. 
\end{rem}
\begin{proof}[Proof of Lemma \ref{lem76}]
Let $K_i$, $i=1,2,\ldots$, be positive constants.
We have
\aa{
(\gm^0_\tau z, z) = \sum_{k=1}^d \int_{\tau_0}^\tau (z,J_{\tau,s} \hat \sg_{k}(s,X_s))^2 \, ds =  
\| (z,J_{\tau,s} \hat \sg(s,X_s))\|^2_{L_2([\tau_0,\tau],\Rnu^d)}.
}
We use the following interpolation inequality (Lemma A.3 in \cite{HP1}):
\mmm{
\lb{interp}
\|f\|_\infty \lt 2\sqrt{d}\,\max\{ (\tau-\tau_0)^{-\frac12}\|f\|_{L_2([\tau_0,\tau],\Rnu^d)}, 
\|f\|_{L_2([\tau_0,\tau],\Rnu^d)}^{2\al(2\al+1)^{-1}}\|f\|^{(2\al+1)^{-1}}_\al\}\\
\lt 2\sqrt{d}\,((\tau-\tau_0)^{-\al(2\al+1)^{-1}} \hspace{-2mm} + 1)  \|f\|_{L_2([\tau_0,\tau],\Rnu^d)}^{2\al(2\al+1)^{-1}}\|f\|^{(2\al+1)^{-1}}_{\C^\al}
}
which holds for  an $\al$-H\"older continuous function $f: [\tau_0,\tau]\to \Rnu^d$.

For $s\in [\tau_0,\tau]$, define $\mc Z_s = (z,J_{\tau,s} \hat \sg(s,X_s))$.
Since $\|f\|_{\C^\al} = \|f\|_\infty + \|f\|_\al \lt K_1 \big(|f(\tau_0)| + \|f\|_\al\big)$, from \rf{interp} we obtain 
\aa{
\|\mc Z\|_{\infty}\lt K_2  (\gm^0_\tau z,z)^{\al(2\al+1)^{-1}}\hspace{-1mm}\big(\|Y_\tau\|_{\mc L(\mc E_p)}
(1+|x_0| + \|X\|_{\al} + \|Z\|_{\al}) +\|\mc Z\|_{\al}\big)^{(2\al+1)^{-1}} \hspace{-1mm}.
}
Since, by assumption, $\sg(t,x)$  and $\pl_t \sg(t,x)$
have linear and polynomial growth, respectively,  w.r.t. $x\in \mc E$,  there exists a number $q\in \Nnu$ such that
\mm{
\|\mc Z\|_{\al} \lt K_3 \|Y_\tau\|_{\mc L(\mc E_{p})} \big(\|Z\|_\al(1+\|X(\fdot)\|_{\infty}) 
+ \|Z\|_\infty(1+\|X(\fdot)\|^q_{\infty} + \|X\|_\al)\big)\\
\lt K_4\|Y_\tau\|_{\mc L(\mc E_{p})}  (1+\|Z\|_\al)(1+|x_0| + \|X\|_\al)^q.
}
Remark that $\|\mc Z\|_{\infty}$ and $\|\mc Z\|_{\al}$ are computed over $[\tau_0,\tau]$.
The last two estimates imply \rf{i0}. 

We have proved that \rf{i0} holds up to a multiplication by a constant. However, 
since $\mc R>2$, this constant can be 
bounded from above by $\mc R^a$ for some $a>0$. Thus, without loss of generality, the constant can be set equal to $1$.
\end{proof}
Lemma \ref{lem5} below is a weaker
 version of Norris' lemma for rough paths (Theorem 3.1 in \cite{HP2}) since
 we are only able to estimate $\|A\|_\infty$ (but not $\|C\|_\infty$) via  a power of $\|I\|_\infty$. 
This happens because our expression for $I_t$ contains a Young integral. 
If we attempt to repeat the arguments 
of Theorem 3.1 in \cite{HP2}, we arrive at the point when we can estimate the sum
of the last two terms in \rf{int4} via a power of $\|I\|_\infty$, but not each term separately. 
\begin{lem}
\lb{lem5}
Let $\te \in(\frac12, 2\al)$.
Assume $(A,A')\in \ms D_B^{2\al}([\tau_0,\tau], \Rnu^{m})$,
$C, D\in \C^\al([\tau_0,\tau],\Rnu^m)$, and $\ffi \in \C^{2\al}([\tau_0,\tau],\Rnu^m)$.
We define
\aaa{
\lb{int4}
I_t = I_{\tau_0} + \int_{\tau_0}^t A_s d\mbf B_s + \int_{\tau_0}^t C_s ds + \int_{\tau_0}^t D_s d\ffi_s, \qquad t\in [\tau_0,\tau],
}
where the last integral is understood as a Young integral,
and let
\aa{
\bar{\mc  R} = |I_{\tau_0}|+ L_\te(B)^{-1} + \|(B,\mbb B)\|_\al + \|(A,A')\|_{B,\al}
 +\|C\|_{\C^\al}  +\|D\|_{\C^\al} + \|\ffi\|_{2\al}. 
 }
Then, there exist $r,q>0$  such that
\aaa{
\lb{ba}
\|A\|_{\infty} \lt \bar K \bar{\mc R}^q \|I\|^r_{\infty},
}
where 
 $\bar K$ is a constant that only depends on $\al$, $\te$, $\tau_0$, and $\tau$.
\end{lem}
\begin{proof}
In what follows, $\bar K_i$, $i=1,2, \ldots$, are positive constants.

First, we note that the Gubinelli derivative of the last term in \rf{int4} is zero. 
Indeed, it follows from Lemma 4.2 in \cite{FH} (Sewing lemma) and, in particular,
from estimate (4.11),  that there exists
a constant $\bar K_1>0$, depending only on $\al$, such that
\aaa{
\lb{e-y}
\Big|\int_s^t D_r d\ffi_r - D_s \dl \ffi_{s,t}\Big| \lt \bar K_1 \|\Xi\|_{3\al} (t-s)^{3\al} \lt 
\bar K_1\|D\|_\al \|\ffi\|_{2\al}(t-s)^{3\al},
}
where $\Xi_{s,t} = D_s \dl \ffi_{s,t}$.
Therefore, $I'_t = A_t$, and hence, by Proposition 1 in \cite{HP2}, 
\aa{
|A_t| = |I'_t| \lt \bar K_2 L_\te^{-1}(B) \|I\|_{_\infty}^{(1-\frac{\te}{2\al})}\big(\|R^I\|^{\frac{\te}{2\al}}_{2\al} 
+ \|I\|_{_\infty}^{\frac{\te}{2\al}}\big).
}
It is straightforward to obtain an estimate on $\|R^I\|_{2\al}$ since 
Theorem 4.10 in \cite{FH} and estimate \rf{e-y} imply that
for all $s,t\in [\tau_0,\tau]$, $s<t$,
\mm{
|R^I_{s,t}| \lt \bar K_3\big[\big(\|B\|_\al\|R^A\|_{2\al} + \|\mbb B\|_{2\al}\|A'\|_{\C^\al}\big)  (t-s)^{3\al} 
+ \|A'\|_\infty\|\mbb B\|_{2\al}(t-s)^{2\al}\\
+\|C\|_\infty(t-s) + \|D\|_\al \|\ffi\|_{2\al} (t-s)^{3\al} + \|D\|_\infty \|\ffi\|_{2\al} (t-s)^{2\al}\big].
}
Furthermore, it holds that
\mm{
\|I\|_\infty \lt  |I_{\tau_0}| + (\tau-\tau_0)^\al \|I\|_\al  \lt \bar K_4 \big( |I_{\tau_0}| + \|A\|_\infty \|B\|_\al + \|R^I\|_{2\al}\big) \\ \lt 
\bar K_4\big( |I_{\tau_0}|+ \bar K_4(|A(0)|  + \|A'\|_\infty\|B\|_\al +  \|R^A\|_{2\al})\|B\|_\al +  \|R^I\|_{2\al}\big).
}
Thus, both quantities $\|R^I\|_{2\al}$ and $\|I\|_\infty$ can be bounded by $\bar K_5 \mc{\bar R}^3$ which implies \rf{ba}.
\end{proof}
\begin{lem}
\lb{lem72}
Let (A1)--(A4) hold. Then, for each $j\in \Nnu$,
there exist constants $p_j,q_j>0$ such that
for all $z\in\Rnu^n$, $|z| =1$, for all initial conditions $x_0\in\Rnu^n$, and for all $V\in\Sg_j$,
\aaa{
\lb{ineq}
\|(z,J_{\tau,_{\displaystyle\fdot}} \hat V(\fdot,X_{_{\displaystyle\fdot}}))\|_{\infty, [\tau_0,\tau]} \lt 
\mc R^{q_j} (z,\gm^0_\tau z)^{p_j}.
}
\end{lem}
\begin{proof}
Recall that that norms and seminorms associated with $X$ and $Z$ are understood over $[0,T]$ (Remark \ref{rem5.5}),
while the rest of the analysis is done on $[\tau_0,\tau]$.

We prove the lemma by induction on $j$.  From Lemma \ref{lem76} it follows that \rf{ineq} is true for $j=0$. Suppose
\rf{ineq} is true for some $j$, and show that it is true for $j+1$. 

By \rf{m-e}, for each unit vector $z\in\Rnu^n$ and  for $t \in [\tau_0, \tau]$, 
\mmm{
\lb{zv}
 (z,J_{\tau,t}\hat V(t,X_t)) =
(z,J_{\tau,\tau_0} \hat V(\tau_0,X_{\tau_0}))   + \int_{\tau_0}^t (z,J_{\tau,s}\mr V(s,X_s)d\ind_{[s,T]})\\
+ \int_{\tau_0}^t (z,J_{\tau,s}\big( \widehat{\pl_s V} + [\hat \sg_{0}, \hat V]\big)(s,X_s)) \, ds 
+ \int_{\tau_0}^t   (z,J_{\tau,s} [\hat \sg, \hat V](s,X_s) ) d \mbf  B_s,
}
where $\mr V(s,X_s) = \vec{V(s,X_s)}{0}$.
For $s\in [\tau_0,\tau]$, we define 
\aa{
&\mc Z_s = (z,J_{\tau,s}[\hat\sg,\hat V](s,X_s)), \quad \mc {\bar Z}_s = (z,J_{\tau,s}
\big(\widehat{\pl_s V} + [\hat\sg_{0},\hat V]\big)(s,X_s)),
 \\
&\mc{\td Z}_s = (z, J_{\tau,s}\mr V(s,X_s)),
}
where $[\hat\sg,\hat V] = \sum_{k=1}^d [\hat \sg_k, \hat V]\ox e_k$. 
 Lemma \ref{lem5}, applied to equation \rf{zv}, implies that
\aaa{
\lb{rhs}
\|\mc Z\|_{\infty,[\tau_0,\tau]} 
\lt   \bar K \bar{\mc R}^q  
\|(z,J_{\tau,_{\displaystyle\fdot}}\hat V(\fdot,X_{_{\displaystyle\fdot}}))\|^r_{\infty, [\tau_0,\tau]},
}
where $\bar {\mc R}$ is the quantity defined in Lemma \ref{lem5}.
 Note that the expression for $\bar {\mc R}$ contains the terms $\|(\mc Z, \mc Z')\|_{B,\al}$,
$\| \mc {\bar Z}\|_{\C^\al}$, $\|\mc{\td Z}_s \|_{\C^\al}$, and $(z,J_{\tau,\tau_0} \hat V(\tau_0,X_{\tau_0}))$,  
 so we have to obtain bounds on these terms by powers of $\mc R$.
Define $\mc V = [\hat\sg,\hat V]$. We start by obtaining a bound on 
 $\|(\mc Z, \mc Z')\|_{B,\al}$.
 For $s,t \in [\tau_0,\tau]$, $s<t$, we have
 \mm{
 \dl \mc Z_{s,t} = (z,\Pi_n Y_\tau \dl Z_{s,t} \mc V(s,X_s)) 
 + (z,\Pi_n Y_\tau Z_s \dl \mc V(\fdot,X_{\fdot})_{s,t})\\
 = (z,\Pi_n Y_\tau Z'_s \dl B_{s,t}  \mc V(s,X_s))  + (z,\Pi_n Y_\tau Z_s \pl_x \mc V(s,X_s)X'_s \dl B_{s,t})
 + R^{\mc Z}_{s,t},
 }
 where the first two terms represent $\mc Z'_s \dl B_{s,t}$.  The latter expression allows to get a bound on $\|R^{\mc Z}\|_{2\al}$.
 Indeed, since  $\dl Z_{s,t} = Z'_s \dl B_{s,t} + R^Z_{s,t}$ and $\dl \mc V(\fdot,X_{\fdot})_{s,t} = \pl_x\mc V(s,X_s) X'_s \dl B_{s,t} + R^{\mc V}_{s,t}$,
 we obtain
 \aaa{
\lb{rzst9}
 R^{\mc Z}_{s,t} = \dl \mc Z_{s,t}  - \mc Z'_s \dl B_{s,t}
 = (z,\Pi_n Y_\tau R^Z_{s,t} \mc V(s,X_s)) +  (z,\Pi_n Y_\tau Z_s R^{\mc V}_{s,t}).
 }
Next, since $\dl X_{s,t} = X'_s \dl B_{s,t} + R^X_{s,t}$ and, furthermore, 
 $\mc V$, $\pl_s \mc V$, $\pl_x\mc V$, and $\pl_x^2\mc V$ have at most polynomial growth with respect to the second argument, there exists 
  a number $q_1\in \Nnu$ such that
\mm{
\hspace{-3mm} \|R^{\mc V}_{s,t}\|_{\mc L(\Rnu^d,\mc E_p)} \lt 
\|\dl \mc V(\fdot,X_{\fdot})_{s,t} - \pl_x\mc V(s,X_s)\dl X_{s,t}\|_{\mc L(\Rnu^d,\mc E_p)} +
\|\pl_x\mc V(s,X_s) R^X_{s,t}\|_{\mc L(\Rnu^d,\mc E_p)}\\
\lt
\|\dl \mc V(\fdot,X_t)_{s,t}\|_{\mc L(\Rnu^d,\mc E_p)}  +\|\dl \mc V(s,X_{\fdot})_{s,t} 
-  \pl_x\mc V(s,X_s)\dl X_{s,t}\|_{\mc L(\Rnu^d,\mc E_p)}\\
+\|\pl_x\mc V(s,X_s) R^X_{s,t}\|_{\mc L(\Rnu^d,\mc E_p)}  \lt K_1 
(1+\|X(\fdot)\|^{q_1}_\infty)\big((t-s)^{2\al}+\|\dl X_{s,t}\|^2_{\mc E_p} +\|R^X_{s,t}\|_{\mc E_p}\big).
}
The previous inequality, together with \rf{rzst9}, implies that there exists a number $q_2\in\Nnu$ such that
\aa{
\|R^{\mc Z}\|_{2\al}  \lt K_2\|Y_\tau\| (1+\|X(\fdot)\|^{q_2}_{\infty})
\big(\|R^Z\|_{2\al}  + \|Z\|_\infty \big(1+ \|X\|^2_\al+\|R^X\|_{2\al} \big)\big),
}
where $\|Y_\tau\| =\|Y_\tau\|_{\mc L(\mc E_p)}$.
We further notice that $\|X(\fdot)\|_\infty \lt |x_0| + T^\al \|X(\fdot)\|_\al \lt |x_0| + T^\al \|X\|_\al$.
Thus, we conclude that $\|R^{\mc Z}\|_{2\al}$ is bounded by a power of $\mc R$.

To obtain an estimate on $\|\mc Z'\|_\al$,  we note that
$Z'_t = -Z_t \pl_x\hat\sg(t,X_t)$ and $X'_t = \hat\sg(t,X_t)$. Therefore,
$\mc Z'_t = -(z,\Pi_n Y_\tau Z_t [\sg, \mc V](t,X_t))$. Let $\mc {\bar V}_t = [\hat\sg, \mc V](t,X_t)$. 
We have
\aa{
\|\mc Z'\|_{\al} \lt  K_3 \|Y_\tau\| \big( \|Z\|_\al \|\mc {\bar V}\|_{\infty} + \|Z\|_\infty 
\|\mc {\bar V}\|_{\al} \big)
\lt K_4 \|Y_\tau\|(1+\|Z\|_\al)(1+ \|\mc {\bar V}\|_{\al}) \lt \mc R^{q_3}
}
for some $q_3\in\Nnu$. Remark that since $V\in \Sg_j$, $\mc {\bar V}=[\hat\sg,[\hat\sg,\hat V]]$ is associated with the maps from $\Sg_{j+2}$. 
The necessary estimates have been already obtained for maps of this type; they are the same as for the map $\mc V=[\hat\sg,\hat V]$. 
This allows us to conclude that $\|\mc {\bar V}\|_{\al}$ is bounded by a power of $\mc R$.
Also,  notice that 
\aa{
|\mc Z_{\tau_0}| \lt \|Y_\tau\| \|Z_{\tau_0}\| \big(1+\|X\|_{\infty}^{q_4}\big)\lt K_5 \|Y_\tau\| \big(1+|x_0|^{q_4} + \|X\|_{\al}^{q_4} +\|Z\|_\al\big).
}
Furthermore,
\aa{
\|\mc {\bar Z}\|_{\C^\al} + \|\mc {\td Z}\|_{\C^\al}  \lt K_6 \|Y_\tau\| (1+\|Z\|_\al)(1+ |x_0| + \|X\|_{\al})^{q_5}, \quad q_5 \in\Nnu.
}
Finally, to obtain a bound on $|(z,J_{\tau,\tau_0} \hat V(\tau_0,X_{\tau_0}))|$, we notice that
this term is bounded by $\|Y_\tau\|(1+\|Z\|_\al)(1+ |x_0|+ \|X\|_{\al})^{q_6}$ (for some $q_6\in\Nnu$)
multiplied by a constant; therefore, it is bounded by a power of $\mc R$. Thus, \rf{rhs} implies that there exists $q_7\in\Nnu$ such that
\aa{
\|\mc Z\|_{\infty, [\tau_0,\tau]} 
\lt    {\mc R}^{q_7}  
\|(z,J_{\tau,_{\displaystyle\fdot}}\hat V(\fdot,X_{_{\displaystyle\fdot}}))\|^r_{\infty, [\tau_0,\tau]},
}By the induction hypothesis,
\aa{
\|(z,J_{\tau,_{\displaystyle\fdot}}\hat V(\fdot,X_{_{\displaystyle\fdot}}))\|_{\infty, [\tau_0,\tau]}
 \lt  \mc R^{q_j} (z,\gm^0_\tau z)^{p_j}.
}
This implies \rf{ineq}.
\end{proof}
\begin{lem}
\lb{lem7}
Assume (A1)--(A5). Then, there exist positive constants $l$ and $r$ such that
 \begingroup
\setlength{\belowdisplayskip}{0pt}
\setlength{\abovedisplayskip}{0pt} 
\aa{
 \inf_{|z|=1}(z,\gm^0_\tau z) \gt  \te^r_\tau(X_\tau) \mc R^{-l}.
}
 \endgroup
\end{lem}
\begin{proof}
Since $(z,V(\tau,X_\tau)) = (z,J_{\tau,s}\hat V(s,X_s))\big|_{s=\tau}$  for all 
$z\in\Rnu^n$, by (A5)-(i), 
\aa{
\te_\tau(X_\tau) \lt \inf_{|z|=1} \sum_{j=1}^{\mc N(\tau)}  \sum_{V\in \Sg_j} 
\|(z,J_{\tau,_{\displaystyle\fdot}}\hat V(\fdot,X_{\displaystyle\fdot}))\|_{\infty, [\tau_0,\tau]}^2 \lt \mc R^{q_{\mc N}}
 \inf_{|z|=1}(z,\gm^0_\tau z)^{p_{\mc N}} 
}
for some $q_{\sss \mc N}, p_{\sss \mc N} > 0$. This  implies the statement.
\end{proof}
Now we are ready to prove Theorem \ref{thm2}.
\begin{proof}[Proof of Theorem \ref{thm2}]
By Proposition \ref{pr52}, to show \rf{gm-}, it suffices to obtain estimate \rf{gm0}.
By Lemma \ref{lem7} and Chebyshev's inequality, for all $z$ such that $|z|=1$,
\mm{
\PP(\inf_{|z|=1}(z,\gm^0_\tau z) \lt \eps) \lt \PP(\te^r_\tau(X_\tau) \mc R^{-l}\lt \eps) \\
\lt \eps^q \, \E[\te_\tau^{- k q r}(X_\tau)]^\frac1{k} \E[\mc R^{kql}]^\frac1{k} \lt 
 \eps^q\,  C  \Phi^{qr}(x_0,\tau)\,  \E[\mc R^{kql}]^\frac1{k},
}
where $k\in\Nnu$ is sufficiently large.
Therefore, $(\det \gm_\tau)^{-1} \in \cap_{q\gt 1} L^q$. 
Furthermore, by Proposition \ref{lem1}, $X(\tau)$ is in $\mbb D^\infty(\Rnu^n)$. Hence, $X(\tau)$ has a smooth density.
\end{proof}
\subsection{Examples of coefficients satisfying (A5)}

We give examples of the coefficients of the SDE \rf{sde}  satisfying (A5) in addition to  (A1)--(A4).

Suppose $d=n$. Let $\sg_i(\tau,x,y)$, where $x\in E_p$, $y\in\Rnu^n$, be of the form $A_i(\tau,\zeta(x_{\tau-}),\eta(y))$, where
$\zeta:  E_p \to\Rnu^k$   (for some $k\in\Nnu$) is a  bounded function and $x_{\tau-} = x\ind_{[0,\tau)}$.
Furthermore, we assume that $\eta: \Rnu^n \to \Rnu^l$ is infinitely differentiable and compactly supported. 
For example, the function $A_i(\tau,\int_0^T \eta(x(s\we \tau))ds)$ can be represented in the required form
if we replace $x(\tau)$ with $y$.
Imposing assumptions on $A_i$ and $\zeta$ in such a way that (A1)--(A4) are satisfied is straightforward.
In particular, $A_i(\tau,\fdot,\fdot)$ should be continuous together with the partial derivatives 
of all orders with respect to the third argument.
Note that since in H\"ormander's condition, we differentiate only with respect to $y$, $\zeta(x_{\tau-})$
can be regarded as a parameter, so we can think of $A_i(\tau,\zeta(x_{\tau-}),\eta(y))$ as of
$A_i(\tau, a,\eta(y))$,  where $a\in\Rnu^k$ denotes a parameter from a bounded set. 
Further note that for each $\tau$ and $a$, $A_i$'s can be regarded as classical vector fields $\Rnu^n \to\Rnu^n$,
$y\mto A_i(\tau, a,\eta(y))$. The fulfillment of H\"ormander's condition for these vector fields is just a classical assumption;
examples can be found in, e.g., \cite{nualart}. Thus, we assume that the classical H\"ormander condition is fulfilled 
for the maps $y\mto A_i(\tau, a,\eta(y))$
for $\tau\in (0,T]$
and for each $a$ from a compact set $\mc D$. 

\subsubsection*{Verification of (A5)-(i)}  
We need to show that the number $\mc N$ in \rf{Gm} can be chosen independently of 
$a$ and $y$. 
For convenience, we redefine the sets $\Sg_j$  as follows: 
$\Sg_0 =  \{A_1,\ldots, A_d\}$,  $\Sg_j =  \{[A_k, V], k=1,\ldots, d, \, V \in \Sg_{j-1} \}$,
where the differentiation is considered with respect to $y$.
Take $(\bar a,\bar y)\in \mc D\x \supp \eta$
and choose a number $N_{(\bar a, \bar y)}$ with the property that the vector fields from
$\cup_{j=0}^{N_{(\bar a,\bar y)}}\Sg_j$, evaluated at $(\bar a,\bar y)$, generate $\Rnu^n$.
One can choose $n$ vector fields $V_1$, \ldots, $V_n$ from $\cup_{j=0}^{N_{(\bar a,\bar y)}}\Sg_j$
in such a way that the determinant of the matrix whose columns are $V_i(\bar a, \bar y)$, $i=1, \ldots, n$, is non-zero.
By continuity, the determinant of the matrix with the columns  $V_i(a, y)$ remains non-zero in some
neighborhood $\mc U_{(\bar a,\bar y)}\sub \mc D \x \Rnu^n$. 
Finally, choose a finite subcover of $\mc D\x \supp \eta$ from the cover $\mc U_{(\bar a,\bar y)}$, $(\bar a,\bar y)\in \mc D\x \supp \eta$.
Let the numbers $N_1, \dots, N_M$ be associated with the neighborhoods of the finite subcover, and let $\mc N(\tau)$ be the maximal of 
these numbers. Then, the set $\big(\!\cup_{j=0}^{\mc N(\tau)}\Sg_j\big)(a,y)$
generates $\Rnu^n$ for all $(a,y)\in \mc D\x \supp \eta$.

\subsubsection*{Verification of (A5)-(ii)}
We write $\Gm(\tau,a,y)$
for  $\inf_{|z|=1}\sum_{j=1}^{\mc N(\tau)} \sum_{V\in \Sg_j} (z,V(\tau, a, \eta(y)))^2$.
Suppose  $V_i(a,y)$, $i=1, \ldots, n$,  are vectors from $\big(\cup_{j=0}^{\mc N(\tau)}\Sg_j\big)(a,y)$ that generate $\Rnu^n$,
and let $M_n(a,y)$ be the $n\x n$ matrix whose columns are $V_i(a,y)$.
Then,
 \begingroup
\setlength{\belowdisplayskip}{2pt}
\setlength{\abovedisplayskip}{2pt} 
\aa{
\Gm(\tau,a,y) \gt \inf_{|z|=1}\sum_{i=1}^{n}  (z,V_i(a,y))^2 =  \inf_{|z|=1}(M_n^\tp M_n(a,y) z,z) = \la_{min}(a,y) > 0,
}
\endgroup
where $\la_{min}(a,y)$ is the smallest eigenvalue of $M_n^\tp M_n(a,y)$. For each $(\bar a,\bar y)\in \mc D\x \supp \eta$,
consider the neighborhood $\mc U_{(\bar a,\bar y)}\sub \mc D \x \Rnu^n$ such that for all $(a,y)\in \mc U_{(\bar a,\bar y)}$,
$\la_{min}(a,y) > \frac12 \la_{min}(\bar a, \bar y)$. Finally, choose a finite subcover from the cover
$\mc U_{(\bar a,\bar y)}$, $(\bar a,\bar y)\in \mc D\x \supp \eta$. Let $\la_1, \ldots, \la_N$ be the minimal eigenvalues
associated with this finite subcover of $\mc D \x  \supp \eta$. Then, for all $(a,y)\in \mc D \x  \supp \eta$,
$\Gm(\tau,a,y)  > \frac12 \min\{\la_1,\ldots, \la_N\}$, which implies (A5)-(ii).

\subsubsection*{An example of the coefficients with $d\ne n$}
Let $d=2$, $n=3$. Assume the coefficients 
$\sg_i(\tau,x,y)$, $i=1,2$, take the form $A_i(\zeta(x_{\tau-}), y)$, where $\zeta: E_p\to [a_{min},a_{max}]$,  $a_{min}\gt 2$,
$y= (y_1,y_2,y_3) \in \Rnu^3$. As before, we substitute a parameter $a$ instead of $\zeta(x_{\tau-})$ since
computing the Lie brackets, we differentiate only with respect to $y$. We explicitly define $A_1$ and $A_2$
as follows:
\aa{
A_1 = \vect{1}{0}{0}, \quad A_2 = \vect{0}{a+\sin y_2}{y_1},
}
where $a\in  [a_{min},a_{max}]$.
Computing the Lie bracket $[A_1,A_2]$ gives
\aa{
[A_1,A_2] = - A_2'A_1 = \vect{0}{0}{-1}.
}
Note that the determinant of the matrix $M_3$ with the columns $A_1$, $A_2$, $[A_1,A_2]$ equals $-(a+\sin y_2)$, and it is non-zero since $a_{min}\gt 2$.
Therefore, $A_1$, $A_2$,  and $[A_1,A_2]$ always generate $\Rnu^3$.  Thus, $N(\tau) = 3$, and hence, (A5)-(i)
is satisfied.

Let us verify (A5)-(ii). Let $\Gm(\tau,a,y)$ be defined as above. We have
\aa{
\Gm(\tau,a,y)\gt \inf_{|z|=1}(M_3^\tp M_3(a,y) z,z) = \la_{min}(a,y),
}
where $\la_{min}(a,y)$ is the minimal eigenvalue of $M_3^\tp M_3(a,y)$.
An explicit computation shows that
\aa{
\la_{min}(a,y) = 1 \we \frac{2(a+\sin y_2)^2}{(a+\sin y_2)^2 + y_1^2 + 1 
+ \sqrt{((a+\sin y_2)^2 + y_1^2 + 1)^2 - 4(a+\sin y_2)^2}}. 
}
This implies that there exists a constant $k(q,a_{min},a_{max})>0$ such that
\aa{
\la_{min}^{-q}(a,y) \lt k(q,a_{min},a_{max})(1+y_1^{2q}).
}
Remark that under (A1), $\E|X(\tau)|^q\lt \bar k_q(1+|x_0|^q)$ for some constant $\bar k_q>0$. The above estimates
imply that (A5)-(ii) is fulfilled for $\te_\tau(\zeta(X_{\tau-}), X(\tau)) = \la_{min}(\zeta(X_{\tau-}),X(\tau))$. Remark that $\zeta$
should chosen in such a way that (A1)--(A4) are satisfied for $\sg_i(\tau,x,y)=A_i(\zeta(x_{\tau-}), y)$.

\subsection{A note on SDEs with discrete delays}

Assumptions (A1)--(A4), unfortunately, exclude SDEs with discrete delays. By the latter, we mean an SDE
whose coefficients $\sg_k$ (and $b$) take the form
\aa{
\sg_k(t,X^t,X(t)) = A_k(t, X(t-h_m), \ldots, X(t-h_1), X(t)),
}
where $0<h_1<h_2 < \dots < h_m < T$ and $A_k$ is a map $[0,T]\x \Rnu^{m+1}\to \Rnu^n$ with the same regularity properties 
as in the example of the coefficient $\sg_k$ (subsection \ref{examples}) depending on the path at a finite number of points. Furthermore,
$X(t-h_i) = x_0$ if $t\in [0,h_i]$, $i=1,\ldots, m$. However, our strategy is expected to work for above-des\-cribed  equations
as well. 

First, we describe the lift of a discrete-delay SDE into a finite-dimensional space. Introduce stochastic processes
$X^{h_i}(t)= X(t-h_i)$, $i=1, \ldots, m$, on $[0,T]$.
One can write an SDE for each $X^{h_i}(t)$ and transform the resulting equations by means
of the time change $t  \leftrightarrow t-h_i$. Furthermore, we define $B^{h_i}_t$ as follows:
$B^{h_i}_t = B_{t-h_i}$ on $[h_i,T]$ and $B^{h_i}_t  = 0$ on $[0,h_i]$.
Thus, the original equation is complemented by a system of $m$ equations. In the similar manner,
we write SDEs for each of the processes $X^{h_i+h_j}(t)= X(t-h_i-h_j)$,  $i,j=1,\ldots, m$, which occur in the arguments of 
the coefficients of the last $m$ equations. We continue this procedure for each resulting equation until 
we arrive at the trivial equation containing just $x_0$ on the right-hand side. 
We then exclude all the trivial equations of the above form.  The final system of equations represents a finite-dimensional 
lift of the original equation. Let $L$ be the number of equations in the final system and $N=Ln$.  
Remark that on the left-hand side, we have a process whose components and only those are arguments
 of the coefficients on the right-hand side. The general form
 of the process on the left-hand side is $X_t = (X(t),X^{h_1}(t), \ldots, X^{h_{L-1}}(t))$, where
 $X^{h_k}(t) = X(t-h_k)$, $k=1, \ldots, L-1$, and
 $h_k$ takes the form $\sum_{i=1}^m \sum_{\al_i \in \{0,1, \ldots, [\frac{T}{h_i}]\}} \al_i h_i$, i.e., 
 $h_k$ with $k={m+1}, \ldots, L-1$ is a sum of some of the original delays $h_1, \ldots, h_m$. The process $X_t$ is
 then  an $\Rnu^N$-valued lift of the solution $X(t)$. We write the $\Rnu^N$-valued equation for $X_t$ as follows: 
 \begingroup
\setlength{\belowdisplayskip}{3pt}
\setlength{\abovedisplayskip}{3pt}
   \aaa{
   \lb{disc-lift}
   dX_t = \hat b(t,X_t)dt + \hat \sg(t,X_t) d\hat B_t,
   }
    \endgroup
where  $\hat B_t = (B_t, B^{h_1}_t, \ldots, B^{h_{L-1}}_t)$, 
$\hat b$ and $\hat \sg$ are $\Rnu^N$-valued and, respectively, $\Rnu^{N\x Ld}$-valued
maps whose components 
   in the equation for $X^{h_i}_t$ are $b(t-h_i, \dots)$  and $\sg_k(t-h_i, \dots)$, respectively, with the 
   arguments shifted by $-h_i$ compared to the arguments of $b$ and $\sg_k$ in the original SDE. 
   
The process $\hat B_t$ can be lifted to a  geometric rough path. 
   Let $h$ be one of the $h_i$'s, $i=1, \ldots, L-1$.
  Define $\mbb B^{ij}_{s,t} = \int_s^t (B^i_{r-h} - B^i_{s-h}) dB^j_r$  and
  $\mbb B^{ji}_{s,t} = (B^j_r - B^j_s) B^i_{r-h}|^t_s - \int_s^t B^i_{r-h}  dB^j_r -\dl_{i,j} \dl_{0,h} (t-s)$, where 
  the integrals are understood in the It\^o sense. For two delays $h_k$ and $h_l$, $k,l = 1,\ldots, L-1$,
  the components $B^i_{t-h_k}$ and $B^j_{t-h_l}$ are integrated against each other likewise. Let 
  $\mbf B_t$ denote the lift of $\hat B_t$. Remark that each $B^{h_i}_t$ is  a martingale
with respect to the filtration $\mc G_t = \mc F_{t-h_i}$ if $t\in [h_i,T]$ and $\mc G_t = \mc F_0$ 
if $t\in [0,h_i]$, and therefore, each of the equations composing \rf{disc-lift} can be viewed as an SDE.
By Proposition \ref{pro4.4},
   all the stochastic integrals  in \rf{disc-lift} can be reinterpreted  as rough integrals, so one can show that \rf{disc-lift} makes sense
   as an RDE, where the stochastic integral is replaced with the rough integral with respect to $\mbf B_t$.
   
   To introduce $\Rnu^{N\x N}$-valued Jacobian $Y_t$, we formulate \rf{disc-lift} for an arbitrary 
 $\Rnu^N$-valued initial condition $\bar x_0$ and define $Y_t = \pl_{\bar x_0} X_t$.
 Furthermore, we introduce a controlled rough path  $M_t = \int_0^t  \pl_x\hat \sg(s, X_s)   d\mbf B_s$,
 where $\pl_x$ denotes the partial gradient with respect to the spatial arguments of $\hat\sg$, and define the
 lift of $M_t$ as follows: $\mbb M_{s,t} = \int_s^t \dl M_{s,r} \ox dM_r$,
where the integration is understood in the sense of Remark 4.11 in \cite{FH} 
(i.e., when we integrate one controlled rough path against another). 
According to \cite{FH} (Section 7.1), $\mbf M = (M,\mbb M)\in \ms C^\al([0,T],\Rnu^N)$.
Also define $N^i_t = \int_{h_i}^t \pl_x\hat b(s-h_i, X^{h_m+h_i}(s), \ldots, X^{h_1+h_i}(s), X(s))) ds$ 
and  $N_t = (N^1_t, \ldots, N^L_t)$. The RDEs for $Y_t$ and its inverse $Z_t$ can be written as follows:
 \begingroup
\setlength{\belowdisplayskip}{2pt}
\setlength{\abovedisplayskip}{3pt} 
 \aa{
 dY_t =  dN_t \, Y_t + d\mbf M_t \,Y_t   \quad \text{and} \quad dZ_t = -Z_t d\td N_t - Z_t d\mbf M_t,
 }
 \endgroup
 where $\td N_t$ is to be determined from the condition $Z_tY_t = I$.
 The above equations, being linear, are expected to have unique global solutions.
Next, by $\Pi_n$ we denote the projection
 $\Rnu^N \to \Rnu^n$ to the original finite-dimensional space. 
  Using the scheme introduced in this work, one can obtain an analog of equation \rf{m-e}. It will be written for the process 
  $\Pi_n Y_\tau Z_t \hat V(t,X^{h_m}(t),
 \ldots, X^{h_1}(t), X(t))$, where $\hat V$ is the $\Rnu^N$-valued map built as a lift of an $\Rnu^n$-valued map
 $V(t,X^{h_m}(t),\ldots, X^{h_1}(t), X(t))$ by completing its  last $N-n$ components by $0$, and $\tau$ is the point where one wants to prove the smoothness
 of the density for $X(\fdot)$. Defining the Lie brackets for  the maps $V_i(t, x(t-h_m),\ldots, x(t-h_1), x(t))$, where $x\in \C([0,T],\Rnu^n)$,
  in terms of vertical derivatives, which is the same
 as taking partial derivatives with respect to the last argument, we obtain the following version of the non-degeneracy condition (A5):
 \bi
\item[\bf (A5)]
\bi
\item[\!(i)]
Suppose $\tau\in (h_i, h_{i+1}]$, $i=0, \ldots, m-1$, $h_0 = 0$.
There exists a number $\mc N(\tau)>0$ and a measurable function $\te_\tau: \Rnu^{i+2}\to (0,\infty)$ such that
\begingroup
\setlength{\belowdisplayskip}{1pt}
\setlength{\abovedisplayskip}{3pt} 
\aa{
\inf_{|z|=1}\sum_{j=1}^{\mc N(\tau)} \sum_{V\in \Sg_j} (z,V(\tau, \underbrace{x_0, \ldots, x_0}_{m-i}, y_0, \ldots, y_{i}))^2 \gt 
\te_\tau(x_0, y_0, \ldots,  y_{i}).
}
\endgroup
\item[\,(ii)] Assumption (A5)-(ii)  (on the inverse moments of 
$\te_\tau(x_0, X_\tau)$)  from subsection \ref{smooth}  is fulfilled.
 \ei
 \ei
If $\tau\in (h_m,T]$, (A5)-(i) looks similar but $\te_\tau: \Rnu^{m+1}\to (0,\infty)$ does not depend on $x_0$. 
Remark that this is the strong form of H\"ormander's condition which we formulated for simplicity; however,
one can also formulate a weaker form of H\"ormander's condition which includes the coefficient of the drift term.
It is possible  to verify that $\hat B_t$ 
 is $\te$-H\"older rough and its modulus of $\te$-H\"older roughness has finite inverse moments, so
we are able to apply Norris's lemma for rough paths.
 We also notice that it is not possible to interpret the RDE for $Z_t$ as an SDE. Thus, in order
 to show that $\|Z\|_\al$ and $\|R^Z\|_{2\al}$ have finite moments, one can make use of
 the result of \cite{cass13} to prove that $\|Z_t\|_{\Rnu^{N\x N}}$ has  finite moments, 
 along with some estimates from \cite{FH} (Chapter 4.3).
 
  Remark that in \cite{CE}, the authors use a different approach to obtaining the smoothness of the density for SDEs with discrete delays. 
 Define $\mc J_{r,t}$ for $r\lt t$
  as a solution to the SDE
 $\mc J_{r,t} = {\rm id} + \sum_{i=0}^m\big[\int_r^t \pl_i b\, \mc J_{r,s-h_i} ds + \int_r^t  \pl_i \sg \,\mc J_{r,s-h_i} dB_s\big]$, and set
 $\mc J_{r,t} = 0$ for $r>t$. Here, $\pl_i$ denotes the derivative with respect to the argument with the delay $h_i$ (we take $h_0 = 0$).
 One then restricts the analysis to the interval $[\tau-h_1,\tau]$ and notes that in the above equation,
 all the terms in the sum vanish, except for the one containing the derivative $\pl_0$. 
 This implies that  for $r,t\in [\tau-h_1,\tau]$, $r\lt t$, $\mc J_{r,t}$ is invertible and the Malliavin derivative factorizes as $D_r X(\tau) = \mc J_{r,\tau} \sg(r,X_r)$.
 The authors then formulate three versions of H\"ormander's condition. The first one is a uniform H\"ormander condition, which
 appears to be stronger than the one suggested here. The other two do not have analogs in this work.  To obtain the smoothness of the density, 
  \cite{CE} also uses rough path techniques and Norris's lemma for rough paths.

\vspace{-2mm}
\subsection*{Acknowledgements}
This research was supported by the Regional Program MATH-AmSud 2018, 
project ``Stochastic analysis of non-Markovian phenomena''\!, grant 88887.197425/ 2018-00. 
A.O. acknowledges 
the support of CNPq Bol\-sa de Pro\-du\-ti\-vi\-da\-de de Pesquisa grant 303443/2018-9. E.S.
thanks ENSTA Paris, where the work was started, for hospitality.
The authors are grateful to the referees for valuable comments that helped to improve the quality of the paper.

\end{document}